\documentclass[11pt]{article}
\usepackage{amsfonts, amsmath, amssymb, latexsym, eucal, amscd} 
\usepackage{cite}
\usepackage[all]{xy}
\newtheorem{theorem}{Theorem}[section]
\newtheorem{lemma}[theorem]{Lemma}
\newtheorem{prop}[theorem]{Proposition}
\newtheorem{corollary}[theorem]{Corollary}
\newtheorem{exAux}[theorem]{Example}

\newtheorem{Def}[theorem]{Definition}
\newenvironment{defi}{\begin{Def} \rm}{\end{Def}}
\newtheorem{Note}[theorem]{Note}
\newenvironment{note}{\begin{Note} \rm}{\end{Note}}
\newtheorem{Problem}[theorem]{Problem}

\newtheorem{Rem}[theorem]{Remark}

\newtheorem{Not}[theorem]{Notation}

\newtheorem{Conj}[theorem]{Conjecture}

\newtheorem{Ass}[theorem]{Assumption}

\newenvironment{proof}{\medskip\noindent{\bf Proof.\ }}{\qed\medskip}
\newenvironment{proofof}[1]{\medskip\noindent{\bf Proof  of {#1}.\ 
}}{\qed\medskip}
\newcommand{\qed}{\hfill\mbox{$\Box$\qquad\qquad}}

\newcommand{\F}{\mathbb{F}}

\newcommand{\Mat}{\text{\rm Mat}}
\newcommand{\Matd}{\text{\rm Mat}_{d+1}(\F)}
\newcommand{\End}{\text{\rm End}(V)}
\newcommand{\vphi}{\varphi}
\renewcommand{\th}{\theta}

\renewcommand{\b}[1]{\langle #1 \rangle }
\newif\ifDRAFT

\begin{document}

\title{Near-bipartite Leonard pairs}

\author{Kazumasa Nomura and Paul Terwilliger}

\maketitle

\begin{center}
\bf \small Abstract.
\end{center}
\begin{quote}  \small
Let $\F$ denote a field, and let $V$ denote a vector
space over $\F$ with finite positive dimension.
A Leonard pair on $V$ is an ordered pair of diagonalizable $\F$-linear maps
$A: V \to V$ and $A^* : V \to V$
that each act on an eigenbasis for the other in an irreducible tridiagonal fashion.
Let $A,A^*$ denote a Leonard pair on $V$.
Let $\{v_i\}_{i=0}^d$ denote an eigenbasis for $A^*$
on which $A$ acts in an irreducible tridiagonal fashion.
For $0 \leq i \leq d$ define an $\F$-linear map $E^*_i : V \to V$ such that
$E^*_i v_i = v_i$ and $E^*_i v_j = 0$ if $j \neq i$ $(0 \leq j \leq d)$.
The map $F = \sum_{i=0}^d E^*_i A E^*_i$ is called the flat part of $A$.
The Leonard pair $A,A^*$ is bipartite whenever $F=0$.
The Leonard pair $A,A^*$ is said to be near-bipartite whenever the pair $A-F, A^*$
is a Leonard pair on $V$.
In this case, the Leonard pair $A-F, A^*$ is bipartite, and called the bipartite contraction
of $A,A^*$.
Let $B,B^*$ denote a bipartite Leonard pair on $V$.
By a near-bipartite expansion of $B,B^*$
we mean a near-bipartite Leonard pair on $V$ with bipartite contraction $B,B^*$.
In the present paper we have three goals.
Assuming $\F$ is algebraically closed,
(i)
we classify up to isomorphism the near-bipartite Leonard pairs over $\F$;
(ii)
for each near-bipartite Leonard pair over $\F$ we describe its bipartite contraction;
(iii)
for each bipartite Leonard pair over $\F$ we describe its near-bipartite expansions.
Our classification (i) is summarized as follows.
We identify two families of Leonard pairs, said to have Krawtchouk type and dual
$q$-Krawtchouk type.
A Leonard pair of dual $q$-Krawtchouk type is said to be reinforced whenever
 $q^{2i} \neq -1$ for $1 \leq i \leq d-1$.
A Leonard pair $A,A^*$ is said to be essentially bipartite whenever the flat part of $A$
is a scalar multiple of the identity.
Assuming  $\F$ is algebraically closed,
we show that a Leonard pair $A,A^*$ over $\F$ with $d \geq 3$ is 
near-bipartite if and only if at least one of the following holds:
(i) $A,A^*$ is essentially bipartite;
(ii) $A,A^*$ has reinforced dual $q$-Krawtchouk type;
(iii) $A,A^*$ has Krawtchouk type.
\end{quote}

\medskip
{\small
 {\bf Key words.} 
bipartite, Leonard pair, Leonard system, near-bipartite, tridiagonal matrix 
}

\medskip
{\small
{\bf AMS subject classifications.}
05E30,
15A21,
15B10
}

\section{Introduction}
\label{sec:intro}
\ifDRAFT {\rm sec:intro}. \fi

The notion of a Leonard pair was introduced by the second author \cite{T:Leonard}.
We recall the definition of a Leonard pair.
A square matrix is said to be tridiagonal whenever each nonzero entry lies on
the diagonal, the subdiagonal, or the superdiagonal.
A tridiagonal matrix is said to be irreducible whenever each entry on the
subdiagonal is nonzero and each entry on the superdiagonal is nonzero.
Let $\F$ denote a field, and let $V$ denote a vector space over $\F$
with finite positive dimension.
A Leonard pair on $V$ is an ordered pair of $\F$-linear maps $A : V \to V$ and
$A^* : V \to V$ 
that satisfy (i) and (ii) below:
\begin{itemize}
\item[\rm (i)]
there exists a basis of $V$ with respect to which the matrix representing $A$
is irreducible tridiagonal and the matrix representing $A^*$ is diagonal;
\item[\rm (ii)]
there exists a basis of $V$ with respect to which the matrix representing $A^*$
is irreducible tridiagonal and the matrix representing $A$ is diagonal.
\end{itemize}
In this case, we say that $A,A^*$ is over $\F$.
We call $d = \dim V - 1$ the diameter of $A,A^*$.

In the literature, there are two well-known families of Leonard pairs,
said to be bipartite and almost-bipartite (see \cite[Section 1]{Brown}).
In the present paper,
we introduce a family of Leonard pairs called near-bipartite.
We will describe this family shortly.
In order to motivate things, we give some history and background.

The Leonard pairs arose from the study of $Q$-polynomial distance-regular graphs
and orthogonal polynomials.
Delsarte showed in \cite{Del} that a $Q$-polynomial distance-regular graph
yields two orthogonal polynomial sequences that are related by what is now called Askey-Wilson 
duality.
Motivated by Delsarte's result,
Leonard  showed in \cite{L} that the $q$-Racah polynomials give the
most general orthogonal polynomial systems that satisfy Askey-Wilson duality.
Leonard's theorem was improved by Bannai and Ito \cite[Theorem 5.1]{BI} by
treating all the limiting cases.
This version gives a complete classification of the orthogonal polynomial
systems that satisfy Askey-Wilson duality.
It shows that the orthogonal polynomial systems that satisfy Askey-Wilson duality
all come from the terminating branch of the Askey scheme,
consisting of 
the $q$-Racah, $q$-Hahn, dual $q$-Hahn, $q$-Krawtchouk,
dual $q$-Krawtchouk, quantum $q$-Krawtchouk, affine $q$-Krawtchouk,
Racah, Hahn, Krawtchouk, Bannai/Ito, and orphan polynomials.
The Leonard theorem \cite[Theorem 5.1]{BI} is rather complicated. 
The notion of a Leonard pair was introduced in \cite{T:Leonard} to simplify and clarify
Leonard's theorem.
A Leonard system \cite[Definition 1.4]{T:Leonard} is a Leonard pair $A,A^*$ together 
with appropriate orderings of the eigenvalues of $A$ and $A^*$.
The Leonard systems are classified up to isomorphism in \cite[Theorem 1.9]{T:Leonard}.
This result gives a linear algebraic version of Leonard's theorem.

We just mentioned how Leonard pairs are related to orthogonal polynomials.
Leonard pairs have applications to many other areas of mathematics and physics,
such as 
Lie theory \cite{IT, NT:Krawt, BM, Hart, Hart2, IT2b},
quantum groups \cite{Al, BT, IT2, IT3, Bockting, AC, BockT, IRT},
spin models  \cite{Cur:spinLP, NT:spin, CN:spin, CauW},
double affine Hecke algebras \cite{NT:DAHA, H:DAHA, H:DAHA2,Lee, LeeT},
partially ordered sets \cite{Liu, T:poset, Wor, MT},
and exactly solvable models in statistical mechanics 
\cite{Bas1, Bas2, Bas3, Bas4, Bas5, Bas6, Bas7}.
For more information about Leonard pairs and related topics,
see  \cite{NT:K, T:LSnote, NT:idemp,
T:Introduction, T:Kyoto, NT:Krawt, NT:TB, Tanaka, BBIT}.

We now describe the near-bipartite Leonard pairs.
Let $A,A^*$ denote a Leonard pair on $V$.
Let $\{v_i\}_{i=0}^d$ denote a basis of $V$ with respect to which the matrix representing
$A$ is irreducible tridiagonal and the matrix representing $A^*$ is diagonal.
For $0 \leq i \leq d$ define an $\F$-linear map $E^*_i : V  \to V$
such that $E^*_i v_i = v_i$ and $E^*_i v_j = 0$ for $j \neq i$ $(0 \leq j \leq d)$.
Define $F = \sum_{i=0}^d E^*_i A E^*_i$.
We call $F$ the flat part of $A$.
The Leonard pair $A,A^*$ is bipartite if and only if $F=0$.
The Leonard pair $A,A^*$ is said to be near-bipartite whenever the pair $A-F, A^*$ is
a Leonard pair on $V$.
In this case, the Leonard pair $A-F, A^*$ is bipartite, and called the bipartite contraction
of $A,A^*$.
Let $B,B^*$ denote a bipartite Leonard pair on $V$.
By a near-bipartite expansion of $B,B^*$ we mean a near-bipartite Leonard pair
on $V$ with bipartite contraction $B,B^*$.

In the present paper we have three main goals.
Assuming $\F$ is algebraically closed,
\begin{itemize}
\item[\rm (i)]
we classify up to isomorphism the near-bipartite Leonard pairs over $\F$;
\item[\rm (ii)]
for each near-bipartite Leonard pair over $\F$ we describe its bipartite contraction;
\item[\rm (iii)]
for each bipartite Leonard pair over $\F$ we describe its near-bipartite expansions.
\end{itemize}

We now summarize our classification for item (i).
In this summary we assume $d \geq 3$;
the case $d \leq 2$ is a  bit different and given in the main body of the paper.
We will identify two families of Leonard pairs, said to have Krawtchouk type and dual $q$-Krawtchouk type.
These are attached to the Krawtchouk polynomials and the dual $q$-Krawtchouk polynomials.
A Leonard pair of dual $q$-Krawtchouk type is said to be reinforced whenever $q^{2i} \neq -1$ for
$1 \leq i \leq d-1$.
A Leonard pair $A,A^*$ is said to be essentially bipartite whenever the flat part $F$ of $A$
is a scalar multiple of the identity.
Assuming $\F$ is algebraically closed,
we will show that 
a Leonard pair $A,A^*$ over $\F$ is near-bipartite if and only if at least one of the following (i)--(iii) holds:
\begin{itemize}
\item[\rm (i)]
$A,A^*$ is essentially bipartite;
\item[\rm (ii)]
$A,A^*$ has reinforced dual $q$-Krawtchouk type;
\item[\rm (iii)]
$A,A^*$ has Krawtchouk type.
\end{itemize}

The paper is organized as follows.
In Section \ref{sec:preliminary} we fix our notation and recall some basic materials 
from linear algebra.
In Section \ref{sec:LP} we recall some materials concerning Leonard pairs.
In Section \ref{sec:TDDseq} we introduce the concept of a TD/D sequence for a Leonard pair.
In Section \ref{sec:TDDform} we introduce the normalized TD/D form of a Leonard pair.
In Section \ref{sec:bip} we recall the bipartite Leonard pairs and essentially bipartite Leonard pairs.
In Section \ref{sec:F} we introduce the flat part $F$ of a Leonard pair.
In Section \ref{sec:nbip} we introduce the near-bipartite property for  Leonard pairs.
In Section \ref{sec:d1} we classify the near-bipartite Leonard pairs with diameter $d=1$.
In Sections \ref{sec:LPd2} and \ref{sec:d2} we classify the near-bipartite Leonard pairs with diameter $d=2$.
Starting in Section \ref{sec:type} we assume $d \geq 3$.
In Section \ref{sec:type}  we recall the type of a Leonard pair.
In Sections \ref{sec:type1}--\ref{sec:vphiphi} we recall the primary data of a Leonard pair.
In Section \ref{sec:bipessbip} we characterize the essentially bipartite Leonard pairs
in terms of the primary data.
In Section \ref{sec:bipbip} we characterize the bipartite Leonard pairs in terms
of the primary data.
In Section \ref{sec:dualqKrawt} we describe the dual $q$-Krawtchouk Leonard pairs.
In Section \ref{sec:Krawt} we describe the Krawtchouk Leonard pairs.
In Section \ref{sec:bipdualqKrawt} we describe the Leonard pairs that  are bipartite and have
dual $q$-Krawtchouk type.
In Section \ref{sec:bipKrawt} we describe the Leonard pairs that are bipartite
and have Krawtchouk type.
In Section \ref{sec:dualqKrawt2} we determine the near-bipartite Leonard pairs of dual 
$q$-Krawtchouk type, and describe their bipartite contraction.
In Section \ref{sec:Krawt2} we show that a Leonard pair of Krawtchouk type is near-bipartite,
and we describe its bipartite contraction.
In Section \ref{sec:classify} we classify the near-bipartite Leonard pairs.
In Sections \ref{sec:expansiondualqKrawt} and \ref{sec:expansionKrawt}
we describe the near-bipartite expansions of a bipartite Leonard pair.

\section{Preliminaries}
\label{sec:preliminary}
\ifDRAFT {\rm sec:preliminary}. \fi

We now begin our formal argument.
Throughout the paper, the following notation is in effect.
Let  $\F$ denote a field.
Every vector space and algebra discussed in this paper is over $\F$.
Fix an integer $d \geq 0$.
The notation $\{x_i\}_{i=0}^d$ refers to the sequence $x_0,x_1,\ldots,x_d$.
Let $\Matd$ denote the algebra consisting of the $d+1$ by $d+1$
matrices that have all entries in $\F$.
We index the rows and columns by $0,1,\ldots, d$.
The identity element of $\Mat_{d+1}(\F)$ is denoted by $I$.
Let $\F^{d+1}$ denote the vector space consisting of the column
vectors with $d+1$ rows and all entries in $\F$.
We index the rows by $0,1,\ldots,d$.
The algebra $\Matd$ acts on $\F^{d+1}$ by left multiplication.
Throughout the paper,  $V$ denotes a vector space with dimension $d+1$.
Let $\End$ denote the algebra consisting of the
$\F$-linear maps $V \to V$.
The identity element of $\End$ is denoted by $I$.
We recall how each basis $\{v_i\}_{i=0}^d$ of $V$ gives an algebra
isomorphism $\text{\rm End}(V) \to \Mat_{d+1}(\F)$.
For $A \in \text{\rm End}(V)$ and $M \in \Mat_{d+1}(\F)$,
we say that {\em $M$ represents $A$ with respect to $\{v_i\}_{i=0}^d$}
whenever $A v_j = \sum_{i=0}^d M_{i,j} v_i$ for $0 \leq j \leq d$.
The isomorphism sends $A$ to the unique matrix in $\Mat_{d+1}(\F)$ that represents
$A$ with respect to $\{v_i\}_{i=0}^d$.
Let $A \in \End$.
By an {\em eigenspace of $A$}, we mean a subspace $W \subseteq V$ 
such that $W \neq 0$ and there exists $\theta \in \F$ such that $W = \{ v \in V \, |\, A v = \theta v\}$;
in this case $\theta$ is the {\em eigenvalue} of $A$ associated with $W$.
We say that $A$ is {\em diagonalizable} whenever $V$ is spanned by the eigenspaces of $A$.
We say that $A$ is {\em multiplicity-free} whenever $A$ is diagonalizable and its eigenspaces
all have dimension one.
Assume that $A$ is multiplicity-free.
Let $\{\th_i\}_{i=0}^d$ denote an ordering of the eigenvalues of $A$.
For $0 \leq i \leq d$ let $V_i$ denote the eigenspace of $A$ associated with $\th_i$,
and define $E_i \in \End$ such that $(E_i - I) V_i =0$ and $E_i V_j=0$ for $j \neq i$ $(0 \leq j \leq d)$.
We call $E_i$ the {\em primitive idempotent of $A$ associated with $\th_i$}.
We have
(i) $E_i E_j = \delta_{i,j} E_i$ $(0 \leq i,j \leq d)$;
(ii) $I = \sum_{i=0}^d E_i$;
(iii) $A E_i = \th_i E_i = E_i A$ $(0 \leq i \leq d)$;
(iv) $A = \sum_{i=0}^d \th_i E_i$;
(v) $V_i = E_i V$ $(0 \leq i \leq d)$;
(vi) $\text{\rm rank}(E_i) = 1$;
(vii) $\text{\rm tr}(E_i) = 1$ $(0 \leq i \leq d)$,
where tr means trace.
Moreover
\begin{align*}
E_i &= 
 \prod_{\scriptsize \begin{matrix} 0 \leq j \leq d \\ j \neq i \end{matrix} } 
  \frac{A - \th_j I} { \th_i -\th_j }
 && (0 \leq i \leq d).
\end{align*}
Let $\b{A}$ denote the subalgebra of $\text{\rm End}(V)$ generated by $A$.
The algebra $\b{A} $ is commutative.
The elements $\{A^i\}_{i=0}^d$ (resp. $\{E_i\}_{i=0}^d$) form a basis of $\b{ A }$.

\begin{lemma}    \label{lem:As2}    \samepage
\ifDRAFT {\rm lem:As2}. \fi
Assume that $A \in \End$ is multiplicity-free with primitive idempotents $\{E_i\}_{i=0}^d$.
Then for $H \in \End$ the following {\rm (i)--(iii)} are equivalent:
\begin{itemize}
\item[\rm (i)]
$H \in \b{ A }$;
\item[\rm (ii)]
$H$ commutes with $A$;
\item[\rm (iii)]
$H$ commutes with $E_i$ for $0 \leq i \leq d$.
\end{itemize}
\end{lemma}

\begin{proof}
This is a reformulation of the fact that 
a matrix $M \in \Matd$ commutes with each diagonal matrix in $\Matd$ 
if and only if $M$ diagonal.
\end{proof}

A square matrix is said to be {\em tridiagonal} whenever
each nonzero entry lies on the diagonal, subdiagonal, or the superdiagonal.
A tridiagonal matrix is said to be {\em irreducible} whenever
each entry on the subdiagonal is nonzero and each entry on the superdiagonal
is nonzero.

\begin{lemma}    \label{lem:tridpre}    \samepage
\ifDRAFT {\rm lem:tridpre}. \fi
Let $T \in \Matd$ be irreducible tridiagonal.
Then the following are equivalent:
\begin{itemize}
\item[\rm (i)]
$T$ is diagonalizable;
\item[\rm (ii)]
$T$ is multiplicity-free.
\end{itemize}
\end{lemma}

\begin{proof}
(i) $\Rightarrow$ (ii)
Observe that $I,T,T^2,\ldots,T^d$ are linearly independent.
Thus the minimal polynomial of $T$ has degree $d+1$.
So $T$ has $d+1$ mutually distinct eigenvalues.

(ii) $\Rightarrow$ (i)
Clear.
\end{proof}

For elements $r$, $s$ in an algebra, 
the notation $[r,s]$ means the commutator $r s - s r$.

\section{Leonard pairs}
\label{sec:LP} 
\ifDRAFT {\rm sec:LP}. \fi

In this section, we recall the notion of a Leonard pair.
Recall the vector space $V$ with dimension $d+1$.

\begin{defi}  {\rm (See \cite[Definition 1.1]{T:Leonard}.) }
  \label{def:LP}    \samepage
\ifDRAFT {\rm def:LP}. \fi
By a {\em Leonard pair on $V$} we mean an ordered pair $A,A^*$
of elements in $\End$ that satisfy {\rm (i)} and {\rm (ii)} below:
\begin{itemize}
\item[\rm (i)]
there exists a basis of $V$ with respect to which the matrix representing
$A$ is irreducible tridiagonal and the matrix representing $A^*$ is diagonal;
\item[\rm (ii)]
there exists a basis of $V$ with respect to which the matrix representing
$A^*$ is irreducible tridiagonal and the matrix representing $A$ is diagonal.
\end{itemize}
We call $d$ the {\em diameter} of $A,A^*$.
We say that $A,A^*$ is {\em over $\F$}.
We call $V$ the {\em underlying vector space}.
\end{defi}

\begin{note}
By a common notational convention, $A^*$ denotes
the conjugate-transpose of $A$.
We are not using this convention.
In a Leonard pair $A,A^*$ the elements $A$ and $A^*$ are arbitrary subject
to (i) and (ii) above.
\end{note}

\medskip
We recall the notion of an isomorphism of Leonard pairs.
Consider a Leonard pair $A,A^*$ on $V$ and a Leonard pair $B,B^*$ on
a vector space $\mathcal V$.
By an {\em isomorphism of Leonard pairs from $A,A^*$ to $B,B^*$} we mean
an $\F$-linear bijection $S : V \to {\mathcal V}$ such that
$S A = B S$ and $S A^* = B^* S$.
We say that $A,A^*$ and $B,B^*$ are {\em isomorphic} whenever
there exists an isomorphism of Leonard pairs from $A,A^*$ to $B,B^*$.

\begin{lemma}   {\rm (See \cite[Lemma 5.1]{NT:affine}.) }
\label{lem:affineLP}    \samepage
\ifDRAFT {\rm lem:affineLP}. \fi
For a Leonard pair $A,A^*$ on $V$ and scalars 
$\xi$, $\xi^*$, $\zeta$, $\zeta^*$ in $\F$ with $\xi \xi^* \neq 0$,
the pair  $\xi A + \zeta I, \, \xi^* A^* + \zeta^* I$ is a Leonard pair on $V$.
\end{lemma}

\begin{lemma}    {\rm (See \cite[Lemma 3.1]{T:Leonard}.) }
\label{lem:LPmultfree}    \samepage
\ifDRAFT {\rm lem:LPmultfree}. \fi
For a Leonard pair $A,A^*$ on $V$,
each of $A$ and $A^*$ is multiplicity-free.
\end{lemma}

\begin{lemma}   {\rm (See \cite[Lemma 3.3]{T:Leonard}.)}
\label{lem:irred}    \samepage
\ifDRAFT {\rm lem:irred}. \fi
Let $A,A^*$ denote a Leonard pair on $V$.
Then there does not exist a subspace $W$ of $V$ such that
$A W \subseteq W$,
$A^* W \subseteq W$,
$W \neq 0$,
$W \neq V$.
\end{lemma}

Let $A,A^*$ denote a Leonard pair on $V$.
An ordering $\{\th_i\}_{i=0}^d$ of the eigenvalues of $A$ is said to be {\em standard}
whenever there exists a basis of $V$ with respect to which the matrix
representing $A$ is $\text{\rm diag}(\th_0, \th_1, \ldots, \th_d)$
and the matrix representing $A^*$ is irreducible tridiagonal.
For a standard ordering $\{\th_i\}_{i=0}^d$ of the eigenvalues of $A$,
the ordering $\{\th_{d-i}\}_{i=0}^d$ is standard and no further ordering is standard.
A standard ordering of the eigenvalues of $A^*$ is similarly defined.

We recall the notion of a parameter array.

\begin{lemma}   {\rm (See \cite[Theorem 3.2]{T:Leonard}.) }
 \label{lem:parray1}    \samepage
\ifDRAFT {\rm lem:parray1}. \fi
Let $A,A^*$ denote a Leonard pair on $V$.
Let $\{\th_i\}_{i=0}^d$ (resp.\ $\{\th^*_i\}_{i=0}^d$) denote a standard ordering
of the eigenvalues of $A$ (resp.\ $A^*$). 
Then there exists a unique sequence $\{\vphi_i\}_{i=1}^d$ of scalars in $\F$
with the following property:
there exists a basis of $V$ with respect to which the matrices
representing $A$ and $A^*$ are
\begin{align*}
A &: 
\begin{pmatrix}
 \th_0 & & & & & \text{\bf 0}  \\
 1 & \th_1 \\
    & 1 & \th_2 \\
    &    &  \cdot & \cdot \\
    &     &          & \cdot & \cdot \\
\text{\bf 0} & & & & 1 & \th_d
\end{pmatrix},
&
A^* &:
\begin{pmatrix}
\th^*_0 & \vphi_1 & & & & {\bf 0}  \\
          & \th^*_1 & \vphi_2 \\
         &  &  \th^*_1 &  \cdot    \\
         &  &  &  \cdot & \cdot  \\
         &  &  &  &  \cdot & \vphi_d  \\
\text{\bf 0} & & & & & \th^*_d
\end{pmatrix}.
\end{align*}
\end{lemma}

\begin{defi}     \label{def:vphi}    \samepage
\ifDRAFT {\rm def:vphi}. \fi
Referring to Lemma \ref{lem:parray1}, the sequence $\{\vphi_i\}_{i=1}^d$
is called the {\em first split sequence} of $A,A^*$ with respect to the
standard orderings $\{\th_i\}_{i=0}^d$ and $\{\th^*_i\}_{i=0}^d$.
\end{defi}

\begin{lemma}    \label{lem:parray2}    \samepage
\ifDRAFT {\rm lem:parray2}. \fi
Let $A,A^*$ denote a Leonard pair on $V$.
Let $\{\th_i\}_{i=0}^d$ (resp.\ $\{\th^*_i\}_{i=0}^d$) denote a standard ordering
of the eigenvalues of $A$ (resp.\ $A^*$). 
Then there exists a unique sequence $\{\phi_i\}_{i=1}^d$ of scalars in $\F$
with the following property:
there exists a basis of $V$ with respect to which the matrices
representing $A$ and $A^*$ are
\begin{align*}
A &: 
\begin{pmatrix}
 \th_d & & & & & \text{\bf 0}  \\
 1 & \th_{d-1} \\
    & 1 & \th_{d-2} \\
    &    &  \cdot & \cdot \\
    &     &          & \cdot & \cdot \\
\text{\bf 0} & & & & 1 & \th_0
\end{pmatrix},
&
A^* &:
\begin{pmatrix}
\th^*_0 & \phi_1 & & & & {\bf 0}  \\
          & \th^*_1 & \phi_2 \\
         &  &  \th^*_1 &  \cdot    \\
         &  &  &  \cdot & \cdot  \\
         &  &  &  &  \cdot & \phi_d  \\
\text{\bf 0} & & & & & \th^*_d
\end{pmatrix}.
\end{align*}
\end{lemma}

\begin{proof}
Apply Lemma \ref{lem:parray1} to the eigenvalue sequence $\{\th_{d-i}\}_{i=0}^d$
of $A,A^*$.
\end{proof}

\begin{defi}     \label{def:phi}    \samepage
\ifDRAFT {\rm def:vphi}. \fi
Referring to Lemma \ref{lem:parray2}, the sequence $\{\phi_i\}_{i=1}^d$
is called the {\em second split sequence} of $A,A^*$ with respect to the
standard orderings $\{\th_i\}_{i=0}^d$ and $\{\th^*_i\}_{i=0}^d$.
\end{defi}

\begin{defi}    {\rm (See \cite[Definition 11.1]{T:TDD}.) }
\label{def:parray}    \samepage
\ifDRAFT {\rm def:parray}. \fi
Let $A,A^*$ denote a Leonard pair on $V$.
By a {\em parameter array of $A,A^*$} we mean a sequence
\begin{equation}
   (\{\th_i\}_{i=0}^d; \{\th^*_i\}_{i=0}^d; \{\vphi_i\}_{i=1}^d; \{\phi_i\}_{i=1}^d),    \label{eq:parray00}
\end{equation}
where $\{\th_i\}_{i=0}^d$ (resp.\ $\{\th^*_i\}_{i=0}^d$) is a standard ordering of the
eigenvalues of $A$ (resp.\ $A^*$),
and $\{\vphi_i\}_{i=1}^d$ (resp.\ $\{\phi_i\}_{i=1}^d$) is the first split sequence
(resp.\ second split sequence) with respect to the standard orderings
$\{\th_i\}_{i=0}^d$ and $\{\th^*_i\}_{i=0}^d$.
\end{defi}

We comment on the uniqueness of the parameter array.

\begin{lemma}    {\rm (See \cite[Theorem 1.11]{T:Leonard}.) }
\label{lem:parrayunique}    \samepage
\ifDRAFT {\rm lem:parrayunique}. \fi
Let $A,A^*$ denote a Leonard pair on $V$, and let
\[
  (\{\th_i\}_{i=0}^d; \{\th^*_i\}_{i=0}^d; \{\vphi_i\}_{i=1}^d; \{\phi_i\}_{i=1}^d)
\]
denote a parameter array of $A,A^*$.
Then each of the following is a parameter array of $A,A^*$:
\begin{align*}
 &   (\{\th_i\}_{i=0}^d; \{\th^*_i\}_{i=0}^d; \{\vphi_i\}_{i=1}^d; \{\phi_i\}_{i=1}^d),
\\
 &   (\{\th_i\}_{i=0}^d; \{\th^*_{d-i}\}_{i=0}^d; \{\phi_{d-i+1}\}_{i=1}^d; \{\vphi_{d-i+1}\}_{i=1}^d),
\\
 &   (\{\th_{d-i}\}_{i=0}^d; \{\th^*_i\}_{i=0}^d; \{\phi_i\}_{i=1}^d; \{\vphi_i\}_{i=1}^d),
\\
 &   (\{\th_{d-i}\}_{i=0}^d; \{\th^*_{d-i}\}_{i=0}^d; \{\vphi_{d-i+1}\}_{i=1}^d; \{\phi_{d-i+1}\}_{i=1}^d).
\end{align*}
Moreover, $A,A^*$ has no further parameter array.
\end{lemma}

We mention a significance of the parameter array.

\begin{lemma}    {\rm (See \cite[Theorem 1.9]{T:Leonard}.) }
\label{lem:LPunique}
\ifDRAFT {\rm lem:LPunique}. \fi
Two Leonard pairs over $\F$ are isomorphic if and only if
they have a parameter array in common.
\end{lemma}

\begin{defi}    \label{def:parrayoverF}    \samepage
\ifDRAFT {\rm def:parrayoverF}. \fi
By a {\em parameter array over $\F$} we mean the parameter array of a Leonard pair
over $\F$.
\end{defi}

In the next result we classify the parameter arrays over $\F$.

\begin{lemma}   {\rm (See \cite[Theorem 1.9]{T:Leonard}.) }
\label{lem:classify}    \samepage
\ifDRAFT {\rm lem:classify}. \fi
Consider a sequence
\begin{equation}
  (\{\th_i\}_{i=0}^d; \{\th^*_i\}_{i=0}^d; \{\vphi_i\}_{i=1}^d; \{\phi_i\}_{i=1}^d)         \label{eq:parray}
\end{equation}    
of scalars in $\F$.
This sequence is a parameter array over $\F$ if and only if
the following conditions {\rm (i)--(v)} hold:
\begin{itemize}
\item[\rm (i)]
$\th_i \neq \th_j, \quad \th^*_i \neq \th^*_j \quad$ if $i \neq j \quad (0 \leq i,j \leq d)$;
\item[\rm (ii)]
$\vphi_i \neq 0, \quad \phi_i \neq 0 \quad (1 \leq i \leq d)$;
\item[\rm (iii)]
$\vphi_i = \phi_1 \sum_{\ell=0}^{i-1} \frac{\th_\ell - \th_{d-\ell} } { \th_0 - \th_d }
             + (\th^*_i - \th^*_0)(\th_{i-1}-\th_d)  \quad (1 \leq i \leq d)$;
\item[\rm (iv)]
$\phi_i = \vphi_1 \sum_{\ell=0}^{i-1} \frac{\th_\ell - \th_{d-\ell} } { \th_0 - \th_d }
             + (\th^*_i - \th^*_0)(\th_{d-i+1}-\th_0)  \quad (1 \leq i \leq d)$;
\item[\rm (v)]
the expressions
\begin{equation}
\frac{\th_{i-2} - \th_{i+1} } { \th_{i-1} - \th_i },
\qquad\qquad
\frac{\th^*_{i-2} - \th^*_{i+1} } { \th^*_{i-1} - \th^*_i }               \label{eq:indep}
\end{equation}
are equal and independent of $i$ for $2 \leq i \leq d-1$.
\end{itemize}
\end{lemma}

\begin{lemma}   {\rm (See \cite[Lemma 6.1]{NT:affine}.) }
\label{lem:affineLPparam}    \samepage
\ifDRAFT {\rm lem:affineLPparam}. \fi
Let $A,A^*$ denote a Leonard pair over $\F$ with parameter array
\[
 (\{\th_i\}_{i=0}^d; \{\th^*_i\}_{i=0}^d; \{\vphi_i\}_{i=1}^d; \{\phi_i\}_{i=1}^d).
\]
Then for scalars $\xi$, $\xi^*$, $\zeta$, $\zeta^*$ in $\F$ with $\xi \xi^* \neq 0$,
the sequence
\[
 ( \{\xi \th_i + \zeta\}_{i=0}^d; \{\xi^* \th^*_i + \zeta^*\}_{i=0}^d;
 \{\xi \xi^* \vphi_i\}_{i=1}^d, \{\xi \xi^* \phi_i\}_{i=1}^d)
\]
is a parameter array of the Leonard pair $\xi A + \zeta I, \, \xi^* A^* + \zeta^* I$.
\end{lemma} 

\begin{defi}    \label{def:beta0}    \samepage
\ifDRAFT {\rm def:beta0}. \fi
Let $(\{\th_i\}_{i=0}^d; \{\th^*_i\}_{i=0}^d; \{\vphi_i\}_{i=1}^d; \{\phi_i\}_{i=1}^d)$
denote a parameter array over $\F$ with  $d \geq 3$.
Define $\beta \in \F$ such that $\beta+1$ is equal to the common value
of the two fractions in \eqref{eq:indep}.
We call $\beta$ the {\em fundamental constant} of the parameter array.
\end{defi}

\begin{defi}    \label{def:betaLP}    \samepage
\ifDRAFT {\rm def:betaLP}. \fi
Let $A,A^*$ denote a Leonard pair over $\F$ with diameter $d \geq 3$.
The parameter arrays of $A,A^*$ have the same fundamental constant $\beta$;
we call $\beta$ the {\em fundamental constant} of $A,A^*$.
\end{defi}

\section{The TD/D sequences of a Leonard pair}
\label{sec:TDDseq}
\ifDRAFT {\rm sec:TDDseq}. \fi

In this section we introduce the notion of a TD/D sequence of a Leonard pair.
To motivate this notion, we make some comments about Leonard pairs.
An irreducible tridiagonal matrix is said to be {\em normalized} whenever the subdiagonal entires
are all $1$.

\begin{lemma}    \label{lem:normalize}    \samepage
\ifDRAFT {\rm lem:normalize}. \fi
For $A \in \End$ and a basis $\{u_i\}_{i=0}^d$ of $V$,
assume that with respect to $\{u_i\}_{i=0}^d$ the matrix $B$ representing 
$A$ is irreducible tridiagonal.
Then for nonzero scalars $\{\alpha_i\}_{i=0}^d$ in $\F$, the following are
equivalent:
\begin{itemize}
\item[\rm (i)]
With respect to the basis $\{\alpha_i u_i\}_{i=0}^d$ of $V$, the matrix representing
$A$ is normalized irreducible tridiagonal;
\item[\rm (ii)]
$\alpha_{i}/\alpha_{i-1} = B_{i,i-1}$ for $1 \leq i \leq d$.
\end{itemize}
\end{lemma}

\begin{proof}
By linear algebra, the matrix representing $A$ with respect to the basis
$\{\alpha_i u_i\}_{i=0}^d$ is equal to $D^{-1} B D$,
where $D = \text{\rm diag}(\alpha_0, \alpha_1, \ldots, \alpha_d)$.
The result is a routine consequence of this.
\end{proof}

\begin{prop}    \label{prop:TDDseq}    \samepage
\ifDRAFT {\rm prop:TDDseq}. \fi
Let $A,A^*$ denote a Leonard pair on $V$.
Let $\{\th^*_i\}_{i=0}^d$ denote a standard ordering of the eigenvalues of $A^*$.
Then there exists a unique sequence $(\{a_i\}_{i=0}^d; \{x_i\}_{i=1}^d)$
of scalars in $\F$ with the following property:
there exists a basis of $V$ with respect to which the matrices
representing $A$ and $A^*$ are
\begin{align}
A &:
\begin{pmatrix}
 a_0 & x_1 & & & & \text{\bf 0}  \\
 1  & a_1  & x_2 \\
    &  1  &  \cdot & \cdot \\
    &     &   \cdot       & \cdot & \cdot \\
     & & & \cdot & \cdot & x_d  \\
\text{\bf 0} & & & & 1 & a_d
\end{pmatrix},
 & 
A^* &: \text{\rm diag}(\th^*_0, \th^*_1, \ldots, \th^*_d).    \label{eq:AAsTDD}
\end{align}
Moreover, $x_i \neq 0$ for $1 \leq i \leq d$.
\end{prop}

\begin{proof}
We first show that the sequence exists.
By Definition \ref{def:LP}(ii) and the definition of a standard ordering,
there exists a basis  $\{u_i\}_{i=0}^d$ of $V$ with respect to which the matrix
representing $A$ is irreducible tridiagonal and the matrix representing
$A^*$ is equal to $\text{\rm diag}(\th^*_0, \th^*_1, \ldots, \th^*_d)$.
After adjusting the basis $\{u_i\}_{i=0}^d$ using Lemma \ref{lem:normalize},
the matrix representing $A$ is normalized irreducible tridiagonal and the
matrix representing $A^*$ is  $\text{\rm diag}(\th^*_0, \th^*_1, \ldots, \th^*_d)$.
We have shown that the sequence exists.
Next we show that the sequence is unique.
Suppose we have another basis $\{v_i\}_{i=0}^d$ of $V$ with respect to which
the matrix representing $A$ is normalized irreducible tridiagonal
and the matrix representing $A^*$ is 
$\text{\rm diag}(\th^*_0, \th^*_1, \ldots, \th^*_d)$.
Observe that for $0 \leq i \leq d$ each of $u_i$ and $v_i$ is an eigenvector for
$A^*$ with eigenvalue $\th^*_i$.
Therefore there exists a nonzero scalar $\alpha_i \in \F$ such that
$v_i = \alpha_i u_i$.
By Lemma \ref{lem:normalize} and the construction,
$\alpha_i/\alpha_{i-1} = 1$ for $1 \leq i \leq d$.
Therefore $\alpha_i$ is independent of $i$ for $0 \leq i \leq d$.
Consequently $v_i = \alpha_0 u_i$ for $0 \leq i \leq d$.
This implies that the matrix representing $A$ with respect to $\{u_i\}_{i=0}^d$ 
is equal to the matrix  representing $A$ with respect to $\{v_i\}_{i=0}^d$.
It follows that the sequence is unique.
We have $x_i \neq 0$ for $1 \leq i \leq d$ by Lemma \ref{lem:irred}.
\end{proof}

\begin{defi}    \label{def:TDDseq}    \samepage
\ifDRAFT {\rm def:TDDseq}. \fi
Let $A,A^*$ denote a Leonard pair on $V$.
By a {\em TD/D sequence of $A,A^*$} we mean a sequence
\begin{equation}
   (\{a_i\}_{i=0}^d; \{x_i\}_{i=1}^d; \{\th^*_i\}_{i=0}^d),            \label{eq:TDDseq}
\end{equation}
where $\{\th^*_i\}_{i=0}^d$ is a standard ordering of the eigenvalues of $A^*$,
and $(\{a_i\}_{i=0}^d; \{x_i\}_{i=1}^d)$ is the corresponding sequence of scalars
from Proposition \ref{prop:TDDseq}.
\end{defi}

We comment on the uniqueness of the TD/D sequence.

\begin{lemma}     \label{lem:TDDrel}    \samepage
\ifDRAFT {\rm lem:TDDrel}. \fi
Let $A,A^*$ denote a Leonard pair on $V$.
Let 
\[
 (\{a_i\}_{i=0}^d; \{x_i\}_{i=1}^d; \{\th^*_i\}_{i=0}^d)
\]
denote a TD/D sequence of $A,A^*$.
Then
\begin{equation}
  (\{a_{d-i}\}_{i=0}^d; \{x_{d-i+1} \}_{i=1}^d; \{ \th^*_{d-i} \}_{i=0}^d)  \label{eq:TDDseq1}
\end{equation}
is a TD/D sequence of $A,A^*$,
and $A,A^*$ has no further TD/D sequence.
\end{lemma}

\begin{proof}
By Proposition \ref{prop:TDDseq} there exists a basis $\{u_i\}_{i=0}^d$ of $V$ 
with respect to which
the matrices representing $A$ and $A^*$ are as in \eqref{eq:AAsTDD}.
Define
\begin{align*}
  v_i &= x_d x_{d-1} \cdots x_{d-i+1} u_{d-i}   &&  (0 \leq i \leq d).
\end{align*}
The sequence $\{v_i\}_{i=0}^d$ is a basis of $V$ with respect to which 
the matrices representing $A$ and $A^*$ are
\begin{align*}
A &:
\begin{pmatrix}
 a_d & x_d & & & & \text{\bf 0}  \\
 1  & a_{d-1}  & x_{d-1} \\
    &  1  &  \cdot & \cdot \\
    &     &   \cdot       & \cdot & \cdot \\
     & & & \cdot & \cdot & x_1  \\
\text{\bf 0} & & & & 1 & a_0
\end{pmatrix},
&
 A^*&: \text{\rm diag}(\th^*_d, \th^*_{d-1}, \ldots, \th^*_0).
\end{align*}
By the paragraph below Lemma \ref{lem:LPmultfree}, 
the ordering $\{\th^*_{d-i}\}_{i=0}^d$ is standard.
By these comments and Definition \ref{def:TDDseq}, the sequence
\eqref{eq:TDDseq1} is a TD/D sequence of $A,A^*$.
The uniqueness follows from the paragraph below Lemma \ref{lem:LPmultfree}.
\end{proof}

We mention a significance of the TD/D sequence.

\begin{prop}   \label{prop:LPTDDunique}    \samepage
\ifDRAFT {\rm prop:LPTDDunique}. \fi
Two Leonard pairs over $\F$ are isomorphic if and only if
they have a TD/D sequence in common.
\end{prop}

\begin{proof}
Let $A,A^*$ denote a Leonard pair on $V$,
and let $B,B^*$ denote a Leonard pair on a vector space $\mathcal{V}$.
First assume that $A,A^*$ and $B,B^*$ are isomorphic.
Then clearly they have the same TD/D sequences.
Next assume that $A,A^*$ and $B,B^*$ have a common
TD/D sequence $(\{a_i\}_{i=0}^d; \{x_i\}_{i=1}^d; \{\th^*_i\}_{i=0}^d)$.
By Definition \ref{def:TDDseq}, $\{\th^*_i\}_{i=0}^d$ is a standard
ordering of the eigenvalues of $A^*$ and $B^*$.
Moreover, 
there exists a basis $\{u_i\}_{i=0}^d$ of $V$ (resp.\ basis $\{v_i\}_{i=0}^d$ of $\mathcal{V}$) with respect to which the
matrix representing $A$ (resp.\ $B$) is equal to the matrix
on the left in \eqref{eq:AAsTDD}, and the matrix representing
$A^*$ (resp.\ $B^*$) is equal to the matrix on the right in
\eqref{eq:AAsTDD}. 
By linear algebra,
there exists an $\F$-linear bijection $V \to {\mathcal V}$ that sends
$u_i \mapsto v_i$ for $0 \leq i \leq d$.
By construction this bijection is an isomorphism of Leonard pairs
from $A,A^*$ to $B,B^*$.
\end{proof}

\begin{defi}    \label{def:TDDseqoverF}    \samepage
\ifDRAFT {\rm def:TDDseqoverF}. \fi
By a {\em TD/D sequence over $\F$} we mean a TD/D sequence
of a Leonard pair over $\F$.
\end{defi}

\begin{defi}    \label{def:correspond}    \samepage
\ifDRAFT {\rm def:correspond}. \fi
Consider a parameter array 
\begin{equation}
 (\{\th_i\}_{i=0}^d; \{\th^*_i\}_{i=0}^d; \{\vphi_i\}_{i=1}^d; \{\phi_i\}_{i=1}^d)  \label{eq:parrayX}
\end{equation}
over $\F$ and a TD/D sequence
\begin{equation}
(\{a_i\}_{i=0}^d; \{x_i\}_{i=1}^d; \{\th^{* \prime}_i\}_{i=0}^d)                        \label{eq:TDDseqX}
\end{equation}
over $\F$.
Then \eqref{eq:parrayX} and \eqref{eq:TDDseqX} are said to {\em correspond}
whenever $\th^*_i = \th^{* \prime}_i$ for $0 \leq i \leq d$,
and there exists a Leonard pair $A,A^*$ over $\F$ 
that has parameter array \eqref{eq:parrayX} 
and TD/D sequence \eqref{eq:TDDseqX}.
\end{defi}

\begin{lemma}    \label{lem:TDDsumai}    \samepage
\ifDRAFT {\rm lem:TDDsumai}. \fi
Referring to Definition \ref{def:TDDseqoverF}, assume that
\eqref{eq:parrayX} and \eqref{eq:TDDseqX} correspond.
Then
$\sum_{i=0}^d \th_i = \sum_{i=0}^d a_i$.
\end{lemma}

\begin{proof}
Let the matrices $A$, $A^*$ are as in \eqref{eq:AAsTDD}.
On one hand, we have $\text{\rm tr}(A) = \sum_{i=0}^d a_i$.
On the other hand, $\text{\rm tr}(A) = \sum_{i=0}^d \th_i$,
since the eigenvalues of $A$ are $\{\th_i\}_{i=0}^d$.
The result follows.
\end{proof}

\begin{lemma}    \label{lem:correspond}    \samepage
\ifDRAFT {\rm lem:correspond}. \fi
The following hold.
\begin{itemize}
\item[\rm (i)]
Consider a parameter array over $\F$:
\begin{equation}
(\{\th_i\}_{i=0}^d; \{\th^*_i\}_{i=0}^d; \{\vphi_i\}_{i=1}^d; \{\phi_i\}_{i=1}^d).    \label{eq:parrayaux4}
\end{equation}
Then there exists a unique TD/D sequence over $\F$ that corresponds to \eqref{eq:parrayaux4}.
\item[\rm (ii)]
Consider a TD/D sequence over $\F$:
\begin{equation}
(\{a_i\}_{i=0}^d; \{x_i\}_{i=1}^d; \{\th^*_i\}_{i=0}^d).         \label{eq:TDDseqaux2}
\end{equation}
Then there exists at least one parameter array over $\F$ that corresponds to \eqref{eq:TDDseqaux2}.
If
\begin{equation}
(\{\th_i\}_{i=0}^d; \{\th^*_i\}_{i=0}^d; \{\vphi_i\}_{i=1}^d; \{\phi_i\}_{i=1}^d)    \label{eq:parrayaux3}
\end{equation}
is a parameter array over $\F$ that corresponds to \eqref{eq:TDDseqaux2},
then so is
\[
 (\{\th_{d-i}\}_{i=0}^d; \{\th^*_i\}_{i=0}^d; \{\phi_i\}_{i=1}^d; \{\vphi_i\}_{i=1}^d),
\]
and there is no other parameter array over $\F$ that corresponds to \eqref{eq:TDDseqaux2}.
\end{itemize}
\end{lemma}

\begin{proof}
(i)
By Lemma \ref{lem:LPunique}, up to isomorphism there exists a unique
Leonard pair $A,A^*$ over $\F$ that has parameter array \eqref{eq:parrayaux4}.
By Lemma \ref{lem:TDDrel} and Definition \ref{def:correspond},
there exists a unique TD/D sequence over $\F$ that corresponds to $A,A^*$.
By these comments we get the result.

(ii)
By Proposition \ref{prop:LPTDDunique}, up to isomorphism  there exists a unique Leonard pair
$A,A^*$ over $\F$ that has TD/D sequence \eqref{eq:TDDseqaux2}.
In Lemma \ref{lem:parrayunique}, the four parameter arrays of $A,A^*$ are given.
Among these four parameter arrays,
only the first and third involves the eigenvalues of $A^*$ in the same
standard order as \eqref{eq:TDDseqaux2}.
By these comments we get the result.
\end{proof}

To prepare for the next results, we define some notation.
For the moment, let $\{\th^*_i\}_{i=0}^d$ denote any sequence of scalars in $\F$.
For $0 \leq i \leq d$ define the following polynomials in an indeterminate $\lambda$:
\begin{align*}
 \tau^*_i (\lambda)  &= (\lambda- \th^*_0)(\lambda - \th^*_1) \cdots (\lambda - \th^*_{i-1}),
\\
 \eta^*_i (\lambda) &= (\lambda-\th^*_d)(\lambda - \th^*_{d-1}) \cdots (\lambda-\th^*_{d-i+1}).
\end{align*}

\begin{lemma} \label{lem:aixiparam}    \samepage
\ifDRAFT {\rm lem:aixiparam}. \fi
Assume that $d \geq 1$.
Consider a sequence of scalars in $\F$:
\begin{equation}
 (\{a_i\}_{i=0}^d; \{x_i\}_{i=1}^d; \{\th^*_i\}_{i=0}^d).   \label{eq:TDDseq2}
\end{equation}
The sequence \eqref{eq:TDDseq2}
is a TD/D sequence over $\F$ if and only if
there exists a parameter array
\begin{equation}
(\{\th_i\}_{i=0}^d; \{\th^*_i\}_{i=0}^d; \{\vphi_i\}_{i=1}^d; \{\phi_i\}_{i=1}^d)
                                                  \label{eq:parray7}
\end{equation}
over $\F$ such that
\begin{align}
a_0 &= \th_0 + \frac{\vphi_1} {\th^*_0 - \th^*_1},            \label{eq:a0param}
\\
a_i &= \th_i + \frac{ \vphi_i } { \th^*_i - \th^*_{i-1} }
          + \frac{ \vphi_{i+1} } {\th^*_i - \th^*_{i+1} }
      &&  (1 \leq i \leq d-1),                                         \label{eq:aiparam}
\\
a_d &= \th_d + \frac{\vphi_d } {\th^*_d - \th^*_{d-1} },    \label{eq:adparam}
\\
x_i &= \vphi_i \phi_i
  \frac{ \tau^*_{i-1}(\th^*_{i-1} ) \eta^*_{d-i}(\th^*_i) }
         { \tau^*_i (\th^*_i) \eta^*_{d-i+1} (\th^*_{i-1} ) }
      &&                  (1 \leq i \leq d).                    \label{eq:xiparam}
\end{align}
In this case,
the TD/D sequence \eqref{eq:TDDseq2} and
the parameter array \eqref{eq:parray7} 
correspond.
\end{lemma}

\begin{proof}
First assume that the sequence \eqref{eq:TDDseq2} is a TD/D sequence over $\F$.
By Definition \ref{def:TDDseqoverF} there exists a Leonard pair $A,A^*$ over $\F$
that has TD/D sequence \eqref{eq:TDDseq2}.
Then \eqref{eq:a0param}--\eqref{eq:xiparam} hold by
\cite[Theorems 17.8, 17.9]{T:qRacah}.
We have proved the result in one direction.
Concerning the converse,
assume that there exists a parameter array \eqref{eq:parray7} over $\F$ that
satisfies \eqref{eq:a0param}--\eqref{eq:xiparam}.
Let $A,A^*$ denote a Leonard pair over $\F$ that has parameter array
\eqref{eq:parray7}.
By \eqref{eq:a0param}--\eqref{eq:xiparam} and \cite[Theorems 17.8, 17.9]{T:qRacah}, 
we find that the sequence \eqref{eq:TDDseq2} is a TD/D sequence of $A,A^*$.
The last assertion follows from our above comments.
\end{proof}

\begin{lemma} \label{lem:aixiparam2}    \samepage
\ifDRAFT {\rm lem:aixiparam2}. \fi
Assume that $d \geq 1$.
Consider a sequence of scalars in $\F$:
\begin{equation}
 (\{a_i\}_{i=0}^d; \{x_i\}_{i=1}^d; \{\th^*_i\}_{i=0}^d).   \label{eq:TDDseq2b}
\end{equation}
The sequence \eqref{eq:TDDseq2b}
is a TD/D sequence over $\F$ if and only if
there exists a parameter array
\begin{equation}
(\{\th_i\}_{i=0}^d; \{\th^*_i\}_{i=0}^d; \{\vphi_i\}_{i=1}^d; \{\phi_i\}_{i=1}^d)
                                                  \label{eq:parray7b}
\end{equation}
over $\F$ such that
\begin{align}
a_0 &= \th_d + \frac{\phi_1} {\th^*_0 - \th^*_1},            \label{eq:a0paramb}
\\
a_i &= \th_{d-i} + \frac{ \phi_i } { \th^*_i - \th^*_{i-1} }
          + \frac{ \phi_{i+1} } {\th^*_i - \th^*_{i+1} }
      &&  (1 \leq i \leq d-1),                                         \label{eq:aiparamb}
\\
a_d &= \th_0 + \frac{\phi_d } {\th^*_d - \th^*_{d-1} },    \label{eq:adparamb}
\\
x_i &= \vphi_i \phi_i
  \frac{ \tau^*_{i-1}(\th^*_{i-1} ) \eta^*_{d-i}(\th^*_i) }
         { \tau^*_i (\th^*_i) \eta^*_{d-i+1} (\th^*_{i-1} ) }
      &&                  (1 \leq i \leq d).                    \label{eq:xiparamb}
\end{align}
In this case,
the TD/D sequence \eqref{eq:TDDseq2b} and
the parameter array \eqref{eq:parray7b} 
correspond.
\end{lemma}

\begin{proof}
By Lemma \ref{lem:correspond}, 
$(\{\th_{d-i}\}_{i=0}^d; \{\th^*_i\}_{i=0}^d; \{\phi_i\}_{i=1}^d; \{\vphi_i\}_{i=1}^d)$
is a parameter array over $\F$ corresponding to \eqref{eq:TDDseq2b}.
Now apply Lemma \ref{lem:aixiparam} to this parameter array.
\end{proof}

\begin{lemma}    \label{lem:TDDaffine}    \samepage
\ifDRAFT {\rm lem:TDDaffine}. \fi
Let $A,A^*$ denote a Leonard pair over $\F$ with TD/D sequence
\[
 (\{a_i\}_{i=0}^d; \{x_i\}_{i=1}^d;  \{\th^*_i\}_{i=0}^d).
\]
Then for scalars $\xi$, $\zeta$, $\xi^*$, $\zeta^*$ in $\F$ with $\xi \xi^* \neq 0$,
the sequence
\[
  ( \{\xi a_i + \zeta\}_{i=0}^d; \{\xi^2 x_i\}_{i=1}^d;   \{\xi^* \th^*_i + \zeta^*\}_{i=0}^d)
\]
is a TD/D sequence of the Leonard pair $\xi A + \zeta I, \xi^* A^* + \zeta^* I$.
\end{lemma}

\begin{proof}
By Lemmas \ref{lem:affineLPparam} and \ref{lem:aixiparam}.
\end{proof}

For the rest of this section, we use the following notation.
Let $A,A^*$ denote a Leonard pair on $V$ with a TD/D sequence
$(\{a_i\}_{i=0}^d; \{x_i\}_{i=1}^d; \{\th^*_i\}_{i=0}^d)$.
For $0 \leq i \leq d$ let $E^*_i$ denote the primitive idempotent of $A^*$
for $\th^*_i$.

\begin{lemma}    {\rm (See \cite[Definition 7.1, Lemma 7.5]{T:qRacah}.) }
\label{lem:aiEsi}     \samepage
\ifDRAFT {\rm lem:aiEsi}. \fi
For $0 \leq i \leq d$ the following hold:
\begin{itemize}
\item[\rm (i)]
$E^*_i A E^*_i = a_i E^*_i$;
\item[\rm (ii)]
$a_i = \text{\rm tr}(A E^*_i)$.
\end{itemize}
\end{lemma}

\begin{lemma}      {\rm (See \cite[Definition 7.1, Lemma 7.5]{T:qRacah}.) }
\label{lem:xiEsi}    \samepage
\ifDRAFT {\rm lem:xiEsi}. \fi
For $1 \leq i \leq d$ the following hold:
\begin{itemize}
\item[\rm (i)]
$E^*_i A E^*_{i-1} A E^*_i = x_i E^*_i$;
\item[\rm (ii)]
$x_i = \text{\rm tr}(E^*_i A E^*_{i-1} A)$.
\end{itemize}
\end{lemma}

\section{Leonard pairs in normalized TD/D form}
\label{sec:TDDform}
\ifDRAFT {\rm sec:TDDform}. \fi

Let $A,A^*$ denote a Leonard pair over $\F$.
In this section we describe a normalization for $A,A^*$,
called the normalized TD/D form.
We show that the normalized TD/D forms of $A,A^*$ are in bijection
with the TD/D sequences of  $A,A^*$.

\begin{defi}    \label{def:normalizedTDD}    \samepage
\ifDRAFT {\rm def:normalizedTDD}. \fi
A Leonard pair $A, A^*$ over $\F$ with diameter $d$ is said to have {\em normalized TD/D form}
whenever the following (i)--(iii) hold:
\begin{itemize}
\item[\rm (i)]
the underlying vector space of $A,A^*$ is $\F^{d+1}$;
\item[\rm (ii)]
$A$ is a normalized irreducible tridiagonal matrix in $\Matd$;
\item[\rm (iii)]
$A^*$ is a diagonal matrix in $\Matd$.
\end{itemize}
\end{defi}

\begin{lemma}     \label{lem:TDDstandard}    \samepage
\ifDRAFT {\rm lem:TDDstandard}. \fi
Consider a Leonard pair $A,A^*$ over $\F$ that is in normalized TD/D form:
\begin{align*}
A &=
\begin{pmatrix}
 a_0 & x_1 & & & & \text{\bf 0}  \\
 1  & a_1  & x_2 \\
    &  1  &  \cdot & \cdot \\
    &     &   \cdot       & \cdot & \cdot \\
     & & & \cdot & \cdot & x_d  \\
\text{\bf 0} & & & & 1 & a_d
\end{pmatrix},
 &
A^* &= \text{\rm diag}(\th^*_0, \th^*_1, \ldots, \th^*_d).
\end{align*}
Then the following hold.
\begin{itemize}
\item[\rm (i)]
The sequence
$(\{a_i\}_{i=0}^d; \{x_i\}_{i=1}^d; \{\th^*_i\}_{i=0}^d)$
is a TD/D sequence of $A,A^*$.
\item[\rm (ii)]
For $0 \leq i \leq d$ let $E^*_i$ denote  the primitive idempotent of $A^*$ 
for $\th^*_i$.
Then $E^*_i$ has $(i,i)$-entry $1$ and all other entries $0$.
\end{itemize}
\end{lemma}

\begin{proof}
(i)
Note by the paragraph below Lemma \ref{lem:LPmultfree} that
$\{\th^*_i\}_{i=0}^d$ is a standard ordering of the eigenvalues of $A^*$.
Now the result follows from Definition \ref{def:TDDseq}.

(ii)
Clear.
\end{proof}

\begin{lemma}     \label{lem:TDD3}    \samepage
\ifDRAFT {\rm lem:TDD3}. \fi
Let $A,A^*$ denote a Leonard pair on $V$ with TD/D sequence
\[
(\{a_i\}_{i=0}^d; \{x_i\}_{i=1}^d; \{\th^*_i\}_{i=0}^d).
\]
Then the matrices
\begin{align*}
&
\begin{pmatrix}
 a_0 & x_1 & & & & \text{\bf 0}  \\
 1  & a_1  & x_2 \\
    &  1  &  \cdot & \cdot \\
    &     &   \cdot       & \cdot & \cdot \\
     & & & \cdot & \cdot & x_d  \\
\text{\bf 0} & & & & 1 & a_d
\end{pmatrix},
 && 
\text{\rm diag}(\th^*_0, \th^*_1, \ldots, \th^*_d) 
\end{align*}
form a Leonard pair in normalized TD/D form that is isomorphic to $A,A^*$.
\end{lemma}

\begin{proof}
By Proposition \ref{prop:TDDseq} and Definition \ref{def:TDDseq}.
\end{proof}

Next we give a variation on Lemma \ref{lem:TDD3}.
Consider a sequence of scalars in $\F$:
\[
  (\{a_i\}_{i=0}^d; \{x_i\}_{i=1}^d; \{\th^*_i\}_{i=0}^d).
\]
Define  matrices $A$, $A^* \in \Matd$ by
\begin{align}
A &=
\begin{pmatrix}
 a_0 & x_1 & & & & \text{\bf 0}  \\
 1  & a_1  & x_2 \\
    &  1  &  \cdot & \cdot \\
    &     &   \cdot       & \cdot & \cdot \\
     & & & \cdot & \cdot & x_d  \\
\text{\bf 0} & & & & 1 & a_d
\end{pmatrix},
 & 
A^* &= \text{\rm diag}(\th^*_0, \th^*_1, \ldots, \th^*_d).     \label{eq:AAsTDD3}
\end{align}

\begin{prop}    \label{prop:TDD2}    \samepage
\ifDRAFT {\rm prop:TDD2}. \fi
With the above notation, the following are equivalent:
\begin{itemize}
\item[\rm (i)]
$(\{a_i\}_{i=0}^d; \{x_i\}_{i=1}^d; \{\th^*_i\}_{i=0}^d)$ is a TD/D sequence over $\F$;
\item[\rm (ii)]
the matrices $A,A^*$ form a Leonard pair over $\F$.
\end{itemize}
Suppose {\rm (i)} and {\rm (ii)} hold.
Then the Leonard pair $A,A^*$ is in normalized TD/D form,
and $(\{a_i\}_{i=0}^d; \{x_i\}_{i=1}^d; \{\th^*_i\}_{i=0}^d)$ is a TD/D sequence
of $A,A^*$.
\end{prop}

\begin{proof}
By Definitions \ref{def:TDDseq} and \ref{def:normalizedTDD}.
\end{proof}

\section{Bipartite and essentially bipartite Leonard pairs}
\label{sec:bip} 
\ifDRAFT {\rm sec:bip}. \fi

In this section, we first recall  the bipartite and essentially bipartite
conditions on a Leonard pair.
We then characterize these conditions in terms of the parameter array.

Throughout this section, let $A,A^*$ denote a Leonard pair over $\F$
with parameter array
\[
 (\{\th_i\}_{i=0}^d; \{\th^*_i\}_{i=0}^d; \{\vphi_i\}_{i=1}^d; \{\phi_i\}_{i=1}^d)
\]
and the corresponding TD/D sequence 
\[
 (\{a_i\}_{i=0}^d; \{x_i\}_{i=1}^d; \{\th^*_i\}_{i=0}^d).
\]

\begin{defi}    {\rm (See \cite[Section 1]{NT:balanced}.) }
\label{def:bip}    \samepage
\ifDRAFT {\rm def:bip}. \fi
\begin{itemize}
\item[\rm (i)]
The Leonard pair $A,A^*$ is said to be {\em bipartite} whenever
$a_i = 0$ for $0 \leq i \leq d$.
\item[\rm (ii)]
The Leonard pair $A,A^*$ is said to be {\em essentially bipartite} whenever
$a_i$ is independent of $i$ for $0 \leq i \leq d$.
\end{itemize}
\end{defi}

\begin{note}    \label{note:essbip}   \samepage
\ifDRAFT {\rm note:essbip}. \fi
A bipartite Leonard pair is essentially bipartite.
\end{note}

\begin{note}   \label{note:d0}    \samepage
\ifDRAFT {\rm note:d0}. \fi
Assume that $d=0$. 
Then any ordered pair $A,A^*$ of elements in $\End$ is an essentially bipartite
Leonard pair.
This Leonard pair is bipartite if and only if $A=0$.
\end{note}

\begin{lemma}    \label{lem:bipaffine}    \samepage
\ifDRAFT {\rm lem:bipaffine}. \fi
The following hold.
\begin{itemize}
\item[\rm (i)]
Assume that $A,A^*$ is bipartite.
Then for $\zeta \in \F$ the Leonard pair $A + \zeta I , A^*$
is essentially bipartite.
\item[\rm (ii)]
Assume that $A,A^*$ is essentially bipartite.
Then there exists a unique $\zeta \in \F$
such that $A-\zeta I, A^*$ is bipartite.
The scalar $\zeta$ is equal to the common value of $\{a_i\}_{i=0}^d$.
\end{itemize}
\end{lemma}

\begin{proof}
By Lemma \ref{lem:TDDaffine} and Definition \ref{def:bip}.
\end{proof}

\begin{lemma}    \label{lem:bip2}    \samepage
\ifDRAFT {\rm lem:bip2}. \fi
Assume that $A,A^*$ is bipartite.
Then for scalars $\xi$, $\xi^*$, $\zeta^*$ in $\F$ with $\xi \xi^* \neq 0$,
the Leonard pair $\xi A, \xi^* A^* + \zeta^* I$ is bipartite.
\end{lemma}

\begin{proof}
By Lemma \ref{lem:TDDaffine} and Definition \ref{def:bip}.
\end{proof}

\begin{lemma}    \label{lem:essbip2}    \samepage
\ifDRAFT {\rm lem:essbip2}. \fi
Assume that $A,A^*$ is essentially bipartite.
Then for scalars $\xi$, $\zeta$, $\xi^*$, $\zeta^*$ in $\F$ with $\xi \xi^* \neq 0$,
the Leonard pair $\xi A + \zeta I, \xi^* A^* + \zeta^* I$ is essentially bipartite.
\end{lemma}

\begin{proof}
By Lemma \ref{lem:TDDaffine} and Definition \ref{def:bip}.
\end{proof}

\begin{lemma}    {\rm (See \cite[Theorem 1.5]{NT:balanced}.)   }
\label{lem:balanced}    \samepage
\ifDRAFT {\rm lem:balanced}. \fi
The following are equivalent:
\begin{itemize}
\item[\rm (i)]
$A,A^*$ is essentially bipartite;
\item[\rm (ii)]
$\th_i + \th_{d-i}$ is independent of $i$ for $0 \leq i \leq d$,
and $\vphi_i + \phi_i = 0$ for $1 \leq i \leq d$.
\end{itemize}
Suppose {\rm (i), (ii)} hold.
Then $\th_i + \th_{d-i} = 2 \alpha$ for $0 \leq i \leq d$,
where $\alpha$ is the common value of $\{a_i\}_{i=0}^d$.
\end{lemma}

\begin{lemma}    \label{lem:bipartite}    \samepage
\ifDRAFT {\rm lem:bipartite}. \fi
Assume that $d \geq 1$.
Then the following are equivalent:
\begin{itemize}
\item[\rm (i)]
$A,A^*$ is bipartite;
\item[\rm (ii)]
$\th_i + \th_{d-i}=0$ for $0 \leq i \leq d$,
and $\vphi_i + \phi_i = 0$ for $1 \leq i \leq d$.
\end{itemize}
\end{lemma}

\begin{proof}
(i) $\Rightarrow$ (ii)
By Lemma \ref{lem:balanced}.

(ii) $\Rightarrow$ (i)
We first show that $\text{\rm Char}(\F) \neq 2$.
By the assumption, $\th_0 + \th_d = 0$.
By $d \geq 1$ and Lemma \ref{lem:classify}(i), $\th_0 \neq \th_d$.
By these comments $\th_d \neq - \th_d$.
Thus $\text{\rm Char}(\F) \neq 2$.
By the assumption $\th_i + \th_{d-i} = 0$ for $0 \leq i \leq d$.
By this and the last assertion of Lemma \ref{lem:balanced},
$2 a_i = 0$ for $0 \leq i \leq d$.
By this and $\text{\rm Char}(\F) \neq 2$ we get $a_i = 0$ for $0 \leq i \leq d$.
\end{proof}

\begin{lemma}     \label{lem:CharF2}    \samepage
\ifDRAFT {\rm lem:CharF2}. \fi
Assume that $d \geq 1$ and $A,A^*$ is essentially bipartite.
Then  $\text{\rm Char}(\F) \neq 2$.
\end{lemma}

\begin{proof}
By way of contradiction, assume that $\text{\rm Char}(\F)=2$.
By Lemma \ref{lem:balanced}, 
$\th_0 + \th_d = 0$.
So $\th_0 = \th_d$, contradicting Lemma \ref{lem:classify}(i).
\end{proof}

\begin{defi}    \label{def:bipparray}     \samepage
\ifDRAFT {\rm def:bipparray}. \fi
A parameter array over $\F$ is said to be {\em bipartite}
(resp.\ {\em essentially bipartite}) whenever
the corresponding Leonard pair is bipartite
(resp.\ essentially bipartite).
\end{defi}

\section{The element $F$ for a Leonard pair}
\label{sec:F}
\ifDRAFT {\rm sec:F}. \fi

Throughout this section, let $A,A^*$ denote a Leonard pair over $\F$
with TD/D sequence 
\begin{equation}
 (\{a_i\}_{i=0}^d; \{x_i\}_{i=1}^d; \{\th^*_i\}_{i=0}^d).           \label{eq:TDDseq3}
\end{equation}
We introduce the flat part $F$ of $A$.
We describe the bipartite condition and essentially bipartite condition
in terms of $F$.
For $0 \leq i \leq d$ let $E^*_i$ denote the primitive idempotent of $A^*$
associated with $\th^*_i$.

\begin{defi}    \label{def:F}    \samepage
\ifDRAFT {\rm def:F}. \fi
Define
\[
   F = \sum_{i=0}^d E^*_i A E^*_i.
\]
We call $F$ the {\em flat part of $A$}.
\end{defi}

\begin{note}    \label{note:defF}    \samepage
\ifDRAFT {\rm note:defF}. \fi
The element $F$ is independent of 
the choice of the TD/D sequence \eqref{eq:TDDseq3} for $A,A^*$.
\end{note}

\begin{lemma}    \label{lem:Fai}    \samepage
\ifDRAFT {\rm lem:Fai}. \fi
Referring to Definition \ref{def:F},
$F = \sum_{i=0}^d a_i E^*_i$.
\end{lemma}

\begin{proof}
By Lemma \ref{lem:aiEsi}(i) and Definition \ref{def:F}.
\end{proof}

\begin{lemma}    \label{lem:Fai2}    \samepage
\ifDRAFT {\rm lem:Fai2}. \fi
Assume that $A,A^*$ is in normalized TD/D form:
\begin{align*}
A &=
\begin{pmatrix}
 a_0 & x_1 & & & & \text{\bf 0}  \\
 1  & a_1  & x_2 \\
    &  1  &  \cdot & \cdot \\
    &     &   \cdot       & \cdot & \cdot \\
     & & & \cdot & \cdot & x_d  \\
\text{\bf 0} & & & & 1 & a_d
\end{pmatrix},
 &
A^* &= \text{\rm diag}(\th^*_0, \th^*_1, \ldots, \th^*_d).
\end{align*}
Then
\begin{align}
F &= \text{\rm diag}(a_0,a_1, \ldots, a_d),
&
A-F &=
\begin{pmatrix}
 0 & x_1 & & & & \text{\bf 0}  \\
 1  & 0  & x_2 \\
    &  1  &  \cdot & \cdot \\
    &     &   \cdot       & \cdot & \cdot \\
     & & & \cdot & \cdot & x_d  \\
\text{\bf 0} & & & & 1 & 0
\end{pmatrix}.                              \label{eq:FA-F}
\end{align}
\end{lemma}

\begin{proof}
By Lemmas \ref{lem:TDDstandard}(ii) and \ref{lem:Fai}.
\end{proof}

\begin{lemma}    \label{lem:F2}    \samepage
\ifDRAFT {\rm lem:F2}. \fi
Let $F$ denote the flat part of $A$.
Then $F$ commutes with $A^*$.
Moreover, $F \in \b{A^*}$.
\end{lemma}

\begin{proof}
By Lemma \ref{lem:As2} and Definition \ref{def:F}.
\end{proof}

\begin{lemma}    \label{lem:Faffine}    \samepage
\ifDRAFT {\rm lem:Faffine}. \fi
Let $F$ denote the flat part of $A$.
For scalars $\xi$, $\zeta$, $\xi^*$, $\zeta^*$ in $\F$ with $\xi \xi^* \neq 0$,
consider the Leonard pair
$\xi A + \zeta I, \xi^* A^* + \zeta^* I$.
For this Leonard pair, the flat part of $\xi A + \zeta I$ is
$\xi F + \zeta I$.
\end{lemma}

\begin{proof}
By Definition \ref{def:F}.
\end{proof}

\begin{lemma}     \label{lem:bip}    \samepage
\ifDRAFT {\rm lem:bip}. \fi
Let $F$ denote the flat part of $A$.
Then the following hold:
\begin{itemize}
\item[\rm (i)]
$A,A^*$ is bipartite if and only if $F=0$;
\item[\rm (ii)]
$A,A^*$ is essentially bipartite if and only if $F$ is a scalar multiple of $I$.
\end{itemize}
\end{lemma}

\begin{proof}
By Definition \ref{def:bip} and Lemma \ref{lem:Fai}.
\end{proof}

\section{Near-bipartite Leonard pairs}
\label{sec:nbip}
\ifDRAFT {\rm sec:nbip}. \fi

In this section we introduce the near-bipartite condition
on a Leonard pair.
To motivate this condition, we make an observation. 
Let $\mathbb{O}$ denote the sequence $\{a_i\}_{i=0}^d$ of scalars in $\F$
such that $a_i = 0$ for $0 \leq i \leq d$.

\begin{lemma}    \label{lem:F3}    \samepage
\ifDRAFT {\rm lem:F3}. \fi
Let $A,A^*$ denote a Leonard pair on $V$ with
TD/D sequence
\[
 (\{a_i\}_{i=0}^d; \{x_i\}_{i=1}^d; \{\th^*_i\}_{i=0}^d).  
\]
Let $F$ denote the flat part of $A$, 
and assume that $A-F, A^*$ is a Leonard pair on $V$.
Then $A-F, A^*$ has a TD/D sequence
\begin{equation}
  (\mathbb{O}; \{x_i\}_{i=1}^d; \{\th^*_i\}_{i=0}^d).         \label{eq:TDDseq0}
\end{equation}
Moreover, the Leonard pair $A-F, A^*$ is bipartite.
\end{lemma}

\begin{proof}
By Lemma \ref{lem:TDD3} we may assume that 
$A,A^*$ is in normalized TD/D form,
and $A,A^*$ are as in \eqref{eq:AAsTDD3}.
By Lemma \ref{lem:Fai2}, 
$F$ and $A-F$ are as in \eqref{eq:FA-F}.
By this and Proposition \ref{prop:TDD2}, $A-F,A^*$ has a TD/D
sequence \eqref{eq:TDDseq0}.
By this and Definition \ref{def:bip}, the Leonard pair
$A-F, A^*$ is bipartite.
\end{proof}

\begin{defi}    \label{def:nearbip}    \samepage
\ifDRAFT {\rm def:nearbip}. \fi
Let $A,A^*$ denote a Leonard pair on $V$,
and let $F$ denote the flat part of $A$.
Then $A,A^*$ is said to be {\em near-bipartite} whenever
$A-F,A^*$ is a Leonard pair on $V$.
In this case, we call the Leonard pair $A-F, A^*$ 
the {\em bipartite contraction of $A,A^*$}.
\end{defi}

\begin{note}    \label{note:nearess}    \samepage
\ifDRAFT {\rm note:nearess}. \fi
An essentially bipartite Leonard pair is near-bipartite.
\end{note}

\begin{note}     \label{note:d0b}    \samepage
\ifDRAFT {\rm note:d0b}. \fi
Assume that $d=0$.
Then any ordered pair of elements in $\End$ is a
near-bipartite Leonard pair.
\end{note}

\begin{defi}    \label{def:bipexpansion}    \samepage
\ifDRAFT {\rm def:bipexpansion}. \fi
Let $B,A^*$ denote a bipartite Leonard pair on $V$.
By a {\em near-bipartite expansion of $B,A^*$} we mean
a near-bipartite Leonard pair $A,A^*$ on $V$ whose bipartite contraction is equal to
$B,A^*$.
\end{defi}

In the next result we clarify Definitions \ref{def:nearbip} and \ref{def:bipexpansion}.

\begin{lemma}     \label{lem:nbippre}    \samepage
\ifDRAFT {\rm lem:nbippre}. \fi
Let $A,A^*$ denote a Leonard pair on $V$,
and let $F$ denote the flat part of $A$.
Let $B,A^*$ denote a bipartite Leonard pair on $V$.
Then the following are equivalent:
\begin{itemize}
\item[\rm (i)]
$B,A^*$ is the bipartite contraction of $A,A^*$;
\item[\rm (ii)]
$A,A^*$ is a near-bipartite expansion of $B,A^*$;
\item[\rm (iii)]
$A-F = B$.
\end{itemize}
\end{lemma}

\begin{proof}
(i) $\Leftrightarrow$ (ii)
By Definitions \ref{def:nearbip} and \ref{def:bipexpansion}.

(i) $\Leftrightarrow$ (iii)
By Definition \ref{def:nearbip}.
\end{proof}

In the next result we give a variation on Lemma \ref{lem:nbippre} that
does not refer to $F$.

\begin{lemma}     \label{lem:nbip}    \samepage
\ifDRAFT {\rm lem:nbip}. \fi
Let $A,A^*$ denote a Leonard pair on $V$,
and let $B,A^*$ denote a bipartite Leonard pair on $V$.
Then the following are equivalent:
\begin{itemize}
\item[\rm (i)]
$B,A^*$ is the bipartite contraction of $A,A^*$;
\item[\rm (ii)]
$A,A^*$ is a near-bipartite expansion of $B,A^*$;
\item[\rm (iii)]
$[A,A^*] = [B,A^*]$.
\end{itemize}
\end{lemma}

\begin{proof}
(i) $\Leftrightarrow$ (ii)
By Lemma \ref{lem:nbippre}.

(i), (ii) $\Rightarrow$ (iii)
By Lemma \ref{lem:nbippre}, $A-F = B$.
By Lemma \ref{lem:F2}, $[F,A^*]=0$.
By these comments, $[A,A^*] = [B,A^*]$.

(iii) $\Rightarrow$ (i), (ii)
For the Leonard pair $A,A^*$
let $\{\th^*_i\}_{i=0}^d$ denote a dual eigenvalue sequence of $A,A^*$.
For $0 \leq i \leq d$ let $E^*_i$ denote the primitive idempotent of $A^*$
associated with $\th^*_i$.
By Definition \ref{def:F},
the flat part of $A$ is $F = \sum_{i=0}^d E^*_i A E^*_i$.
By Lemma \ref{lem:bip}(i) and since the Leonard pair $B,A^*$ is bipartite,
we have $E^*_i B E^*_i = 0$ for $0 \leq i \leq d$.
By assumption $[A,A^*] = [B,A^*]$ so $[A-B, A^*]=0 $.
By this and Lemma \ref{lem:As2}, we have $A-B \in \b{ A^* }$.
The elements $\{E^*_i\}_{i=0}^d$ form a basis of $\b{ A^* }$,
so there exist scalars $\{\alpha_i\}_{i=0}^d$ in $\F$ such that
$A-B = \sum_{i=0}^d \alpha_i E^*_i$.
We may now argue
\[
F= \sum_{i=0}^d E^*_i A E^*_i
  = \sum_{i=0}^d E^*_i (A-B) E^*_i
 = \sum_{i=0}^d \alpha_i E^*_i 
 = A-B.
\]
By these comments and Lemma \ref{lem:nbippre} we obtain (i), (ii).
\end{proof}

\begin{lemma}    \label{lem:nearbip1}    \samepage
\ifDRAFT {\rm lem:nearbip1}. \fi
Let $A,A^*$ denote a near-bipartite Leonard pair on $V$,
with bipartite contraction $B,A^*$.
Let $\xi$, $\zeta$, $\xi^*$, $\zeta^*$ denote scalars in $\F$ with $\xi \xi^* \neq 0$. 
Then the Leonard pair 
\begin{equation}
  \xi A + \zeta I, \, \xi^* A^* + \zeta^* I                  \label{eq:affine}
\end{equation}
is near-bipartite, with bipartite contraction
$\xi B, \xi^* A^* + \zeta^* I$.
\end{lemma}

\begin{proof}
Let $F$ denote the flat part of $A$.
We have $B = A-F$ by Lemma \ref{lem:nbippre}.
For the Leonard pair \eqref{eq:affine},
the flat part of $\xi A + \zeta I$ is $\xi F + \zeta I$
by Lemma \ref{lem:Faffine}.
We have
\[
   \xi A + \zeta I - (\xi F + \zeta I) = \xi (A-F) = \xi B.
\]
Note that $\xi B, \xi^* A^* + \zeta^* I$ is a bipartite Leonard pair by
Lemma \ref{lem:bip2} and since $B,A^*$ is a bipartite Leonard pair.
By these comments
\[
   \xi A + \zeta I - (\xi F + \zeta I), \; \xi^* A^* + \zeta^* I
\]
is a Leonard pair.
By these comments and Definition \ref{def:nearbip}, the pair \eqref{eq:affine} is near-bipartite,
with bipartite contraction
$\xi B, \xi^* A^* + \zeta^* I$.
\end{proof}

\begin{lemma}    \label{lem:nbipTDD1}    \samepage
\ifDRAFT {\rm lem:nbipTDD1}. \fi
Let $A,A^*$ denote a near-bipartite Leonard pair over $\F$
that is in normalized TD/D form.
Then the bipartite contraction of $A,A^*$ is in normalized TD/D form.
\end{lemma}

\begin{proof}
By Definition \ref{def:normalizedTDD} and Lemmas \ref{lem:Fai2}, \ref{lem:F3}.
\end{proof}

\begin{lemma}    \label{lem:nbipTDD2}    \samepage
\ifDRAFT {\rm lem:nbipTDD2}. \fi
Let $B,A^*$ denote a bipartite Leonard pair over $\F$ that
is in normalized TD/D form.
Then every near-bipartite expansion of $B,A^*$ is in normalized TD/D form.
\end{lemma}

\begin{proof}
By Definition \ref{def:normalizedTDD} and Lemmas \ref{lem:Fai2}, \ref{lem:F3}.
\end{proof}

\begin{lemma}   \label{lem:nbipTDD3}    \samepage
\ifDRAFT {\rm lem:nbipTDD3}. \fi
Referring to Lemma \ref{lem:nbip},
assume that both $A,A^*$ and $B,A^*$ are in normalized TD/D form.
Let 
\begin{align*}
& (\{\th_i\}_{i=0}^d; \{\th^*_i\}_{i=0}^d; \{\vphi_i\}_{i=1}^d; \{\phi_i\}_{i=1}^d),
&&
 (\{\th'_i\}_{i=0}^d; \{\th^*_i\}_{i=0}^d; \{\vphi'_i\}_{i=1}^d; \{\phi'_i\}_{i=1}^d)
\end{align*}
denote a parameter array of $A,A^*$ and $B,A^*$, respectively,
and let
\begin{align*}
& (\{a_i\}_{i=0}^d; \{x_i\}_{i=1}^d; \{\th^*_i\}_{i=0}^d),
&&
 (\mathbb{O}; \{x'_i\}_{i=1}^d; \{\th^*_i\}_{i=0}^d)
\end{align*}
denote the corresponding TD/D sequence of $A,A^*$ and $B,A^*$, respectively.
Then {\rm (i)--(iii)} hold if and only if
\begin{align}
  x_i &= x'_i   &&   (1 \leq i \leq d)     \label{eq:xixdi}
\end{align}
if and only if
\begin{align}
  \vphi_i \phi_i &= \vphi'_i \phi'_i  &&  (1 \leq i \leq d).    \label{eq:vphiphivphidphid}
\end{align}
\end{lemma}

\begin{proof}
By Lemma \ref{lem:Fai2},
$A^* = \text{\rm diag}(\th^*_0, \th^*_1, \ldots, \th^*_d)$ and
$F = \text{\rm diag}(a_0, a_1, \ldots, a_d)$.
Also by Lemma \ref{lem:Fai2},
\begin{align*}
A &=
\begin{pmatrix}
 a_0 & x_1 & & & & \text{\bf 0}  \\
 1  & a_1  & x_2 \\
    &  1  &  \cdot & \cdot \\
    &     &   \cdot       & \cdot & \cdot \\
     & & & \cdot & \cdot & x_d  \\
\text{\bf 0} & & & & 1 & a_d
\end{pmatrix},
 & 
B &=
\begin{pmatrix}
 0 & x'_1 & & & & \text{\bf 0}  \\
 1  & 0  & x'_2 \\
    &  1  &  \cdot & \cdot \\
    &     &   \cdot       & \cdot & \cdot \\
     & & & \cdot & \cdot & x'_d  \\
\text{\bf 0} & & & & 1 & 0
\end{pmatrix}.
\end{align*}
By this and Definition \ref{def:nearbip},
the equivalent conditions (i)--(iii) in Lemma \ref{lem:nbip} hold 
if and only if $A-F = B$
if and only if \eqref{eq:xixdi} holds.
By \eqref{eq:xiparam} the condition \eqref{eq:xixdi} is equivalent to 
the condition \eqref{eq:vphiphivphidphid}.
\end{proof}

\begin{lemma}    \label{lem:nbipTDD3b}    \samepage
\ifDRAFT {\rm lem:nbipTDD3b}. \fi
Let $A,A^*$ denote a Leonard pair over $\F$ with parameter array
\begin{equation}
 (\{\th_i\}_{i=0}^d; \{\th^*_i\}_{i=0}^d; \{\vphi_i\}_{i=1}^d; \{\phi_i\}_{i=1}^d).   \label{eq:parray000x}
\end{equation}
Let
\begin{equation}
   (\{\th'_i\}_{i=0}^d; \{\th^*_i\}_{i=0}^d; \{\vphi'_i\}_{i=1}^d; \{\phi'_i\}_{i=1}^d)   \label{eq:parray00ax}
\end{equation}
denote a bipartite parameter array over $\F$.
Then the following are equivalent:
\begin{itemize}
\item[\rm (i)]
$\vphi_i \phi_i =\vphi'_i \phi'_i$ for $1 \leq i \leq d$;
\item[\rm (ii)]
$A,A^*$ is near-bipartite, and its bipartite contraction has
parameter array \eqref{eq:parray00ax}.
\end{itemize}
\end{lemma}

\begin{proof}
By Lemma \ref{lem:TDD3}  we may assume that $A,A^*$ is in normalized
TD/D form.
Let $B,A^*$ denote a Leonard pair in normalized TD/D form that has parameter array
\eqref{eq:parray00ax}.
Note that $B,A^*$ is bipartite.
Now the result follows by Lemma \ref{lem:nbipTDD3}.
\end{proof}

\begin{lemma}    \label{lem:nbipTDD3c}    \samepage
\ifDRAFT {\rm lem:nbipTDD3c}. \fi
Let $B,A^*$ denote a bipartite Leonard pair over $\F$ with parameter array
\begin{equation}
   (\{\th'_i\}_{i=0}^d; \{\th^*_i\}_{i=0}^d; \{\vphi'_i\}_{i=1}^d; \{\phi'_i\}_{i=1}^d).   \label{eq:parray00a}
\end{equation}
Let
\begin{equation}
 (\{\th_i\}_{i=0}^d; \{\th^*_i\}_{i=0}^d; \{\vphi_i\}_{i=1}^d; \{\phi_i\}_{i=1}^d)   \label{eq:parray000}
\end{equation}
denote a parameter array over $\F$.
Then the following are equivalent:
\begin{itemize}
\item[\rm (i)]
$\vphi_i \phi_i = \vphi'_i \phi'_i$ for $1 \leq i \leq d$;
\item[\rm (ii)]
there exists a near-bipartite expansion of $B,A^*$ that has
parameter array \eqref{eq:parray000}.
\end{itemize}
\end{lemma}

\begin{proof}
By Lemma \ref{lem:TDD3}  we may assume that $B,A^*$ is in normalized
TD/D form.

(i) $\Rightarrow$ (ii)
Define scalars $\{a_i\}_{i=0}^d$ by \eqref{eq:a0param}--\eqref{eq:adparam}.
Define scalars $\{x_i\}_{i=1}^d$ by \eqref{eq:xiparam}.
Define $A \in \Matd$ as in \eqref{eq:AAsTDD3}.
By Lemma \ref{lem:TDD3} the pair $A,A^*$ is a Leonard pair that has
parameter array \eqref{eq:parray000}.
By Lemma \ref{lem:nbipTDD3}, $A,A^*$ is a near-bipartite expansion of $B,A^*$.

(ii) $\Rightarrow$ (i)
By Lemma \ref{lem:nbipTDD3}.
\end{proof}

In the present paper we have three main goals.
Assuming $\F$ is algebraically closed,
\begin{itemize}
\item[\rm (i)]
we classify up to isomorphism the near-bipartite Leonard pairs over $\F$;
\item[\rm (ii)]
for each near-bipartite Leonard pair over $\F$ we describe its bipartite contraction;
\item[\rm (iii)] for each bipartite Leonard pair over $\F$ we describe its near-bipartite expansions.
\end{itemize}

\begin{lemma}     \label{lem:CharF2b}    \samepage
\ifDRAFT {\rm lem:CharF2b}. \fi
Assume that $d \geq 1$ and $\text{\rm Char}(\F) = 2$.
Then there does not exist a near-bipartite Leonard pair on $V$.
\end{lemma}

\begin{proof}
By Definition \ref{def:nearbip} and Lemma \ref{lem:CharF2}.
\end{proof}

In view of  Lemmas \ref{lem:CharF2},  \ref{lem:CharF2b},
for the rest of this paper we assume
\[
   \text{\rm Char}(\F) \neq 2.
\]
As a warmup we will treat  $d=1$,  and for this case we assume that $\F$ is arbitrary. 
Later, for $d\geq 2$ we will assume that $\F$ is algebraically closed.

\section{Near-bipartite Leonard pairs with diameter one}
\label{sec:d1}
\ifDRAFT {\rm sec:d1}. \fi

In this section, we classify up to isomorphism the near-bipartite Leonard pairs with diameter one.
Moreover, for each near-bipartite Leonard pair with diameter one, we describe
its bipartite contraction.
Also, for each bipartite Leonard pair with diameter one,
we describe its near-bipartite expansions.

In view of Proposition \ref{prop:TDD2}, we consider 
the following matrices in $\Mat_2(\F)$:
\begin{align*}
A &= \begin{pmatrix}
       a_0 & x_1  \\
       1   & a_1
     \end{pmatrix},
&
B &= \begin{pmatrix}
        0 & x_1 \\
        1 & 0
       \end{pmatrix},
&
A^* &= \begin{pmatrix}
           \th^*_0 & 0   \\
            0 & \th^*_1
         \end{pmatrix}.
\end{align*}
Let $f_A(\lambda)$ denote the characteristic polynomial of $A$.
By linear algebra,
\[
f_A (\lambda) = \lambda^2 - (a_0 + a_1)\lambda + a_0 a_1 - x_1.
\]
Let $\th_0$, $\th_1$ denote the roots of $f_A(\lambda)$.
We remark that $\th_0$, $\th_1$ are contained in the algebraic closure of $\F$,
and possibly not in $\F$.
We have
\begin{equation}
  (\th_0 - \th_1)^2 = (a_0 - a_1)^2 + 4 x_1.        \label{eq:d1key}
\end{equation}

\begin{lemma}    \label{lem:d1key}    \samepage
\ifDRAFT {\rm lem:d1key}. \fi
The pair $A,A^*$ is a Leonard pair over $\F$
if and  only if the following {\rm (i)} and {\rm (ii)} hold:
\begin{itemize}
\item[\rm (i)]
$\th_0$, $\th_1 \in \F$;
\item[\rm (ii)]
$x_1 \neq 0$, 
$\th^*_0 \neq \th^*_1$, and 
\begin{equation}
(a_0 - a_1)^2 + 4 x_1 \neq 0.            \label{eq:d1cond}
\end{equation}
\end{itemize}
\end{lemma}

\begin{proof}
First assume that $A,A^*$ is a Leonard pair over $\F$.
Clearly (i) holds.
By Lemma \ref{lem:irred}, $x_1 \neq 0$.
By Lemma \ref{lem:LPmultfree},
$\th^*_0 \neq \th^*_1$.
Also by Lemma \ref{lem:LPmultfree}, $\th_0 \neq \th_1$.
By this and \eqref{eq:d1key} we get \eqref{eq:d1cond}.
Thus (ii) holds.
Next assume that (i) and (ii) hold.
We have $\th_0 \neq \th_1$ by \eqref{eq:d1key} and \eqref{eq:d1cond}, so $A$ is diagonalizable.
The matrices $A$ and $A^*$ do not have an eigenspace in common,
since $A$ is not upper or lower triangular.
Since $A$ is diagonalizable, there exists an invertible matrix $P \in \Mat_2(\F)$
such that $P A P^{-1}$ is diagonal.
The matrix $P A^* P^{-1}$ has off diagonal entries nonzero,
because $A$ and $A^*$ have no eigenspace in common.
Therefore $P A^* P^{-1}$ is irreducible tridiagonal.
By these comments,
 $A,A^*$ is a Leonard pair over $\F$.
\end{proof}

Let $f_B(\lambda)$ denote the characteristic polynomial of $B$.
We have $f_B(\lambda) = \lambda^2 - x_1$. The roots of $f_B(\lambda)$
are $\pm \sqrt{x_1}$.
We remark that $\sqrt{x_1}$ is contained in the algebraic closure of $\F$,
and possibly not in $\F$.
The following result classifies the bipartite Leonard pairs with diameter one.

\begin{prop}    \label{prop:d1pre}    \samepage
\ifDRAFT {\rm prop:d1pre}. \fi
The pair $B,A^*$ is a Leonard pair over $\F$ if and only if
the following {\rm (i)} and {\rm (ii)} hold:
\begin{itemize}
\item[\rm (i)]
$\sqrt{x_1} \in \F$;
\item[\rm (ii)]
$x_1 \neq 0$ and $\th^*_0 \neq \th^*_1$.
\end{itemize}
\end{prop}

\begin{proof}
By Lemma \ref{lem:d1key}.
\end{proof}

The following result classifies the near-bipartite Leonard pairs with diameter one,
and describes its bipartite contraction.

\begin{prop}    \label{prop:d1}    \samepage
\ifDRAFT {\rm prop:d1}. \fi
Assume that the pair $A,A^*$ is a Leonard pair over $\F$.
Then the following {\rm (i)} and {\rm (ii)} are equivalent.
\begin{itemize}
\item[\rm (i)]
$A,A^*$ is near-bipartite;
\item[\rm (ii)]
$B,A^*$ is a Leonard pair over $\F$.
\end{itemize}
Assume that {\rm (i)} and {\rm (ii)} hold.
Then $B,A^*$ is the bipartite contraction of $A,A^*$.
\end{prop}

\begin{proof}
Clear.
\end{proof}

The following result describes the near-bipartite expansions of a given
bipartite Leonard pair over $\F$ with diameter one.

\begin{prop}    \label{prop:d1converse}    \samepage
\ifDRAFT {\rm prop:d1converse}. \fi
Assume that the pair $B,A^*$ is a Leonard pair over $\F$.
Then the following {\rm (i)--(iii)} are equivalent:
\begin{itemize}
\item[\rm (i)]
$A,A^*$ is a near-bipartite expansion of $B,A^*$;
\item[\rm (ii)]
$A,A^*$ is a Leonard pair over $\F$;
\item[\rm (iii)]
$\th_0$, $\th_1 \in \F$ and $a_0$, $a_1$, $x_1$ satisfy
\eqref{eq:d1cond}. 
\end{itemize}
\end{prop}

\begin{proof}
The equivalence of (i) and (ii) follows from Lemmas \ref{lem:Fai2} and \ref{lem:nbippre}.
The equivalence of (ii) and (iii) follows from Lemma \ref{lem:d1key} 
and Proposition \ref{prop:d1pre}.
\end{proof}

Next we consider what the previous results in this section become
if $\F$ is algebraically closed.

\begin{lemma}    \label{lem:d1keyb}    \samepage
\ifDRAFT {\rm lem:d1keyb}. \fi
Assume that $\F$ is algebraically closed.
Then the pair $A,A^*$ is a Leonard pair over $\F$
if and  only if 
$x_1 \neq 0$, $\th^*_0 \neq \th^*_1$, and 
\[
(a_0 - a_1)^2 + 4 x_1 \neq 0. 
\]
\end{lemma}

\begin{prop}    \label{prop:d1preb}    \samepage
\ifDRAFT {\rm prop:d1preb}. \fi
Assume that $\F$ is algebraically closed.
Then the pair $B,A^*$ is a Leonard pair over $\F$ if and only if
$x_1 \neq 0$ and $\th^*_0 \neq \th^*_1$.
\end{prop}

\begin{prop}    \label{prop:d1b}    \samepage
\ifDRAFT {\rm prop:d1b}. \fi
Assume that $\F$ is algebraically closed, and the pair $A,A^*$ is a Leonard pair over $\F$.
Then $A,A^*$ is near-bipartite if and only if $B,A^*$ is a Leonard pair
over $\F$.
In this case, $B,A^*$ is the bipartite contraction of $A,A^*$.
\end{prop}

\begin{prop}    \label{prop:d1converseb}    \samepage
\ifDRAFT {\rm prop:d1converseb}. \fi
Assume that $\F$ is algebraically closed, and
the pair $B,A^*$ is a Leonard pair over $\F$.
Then $A,A^*$ is a near-bipartite expansion of $B,A^*$ if and only if
$A,A^*$ is a Leonard pair over $\F$
if and only if $a_0$, $a_1$, $x_1$ satisfy \eqref{eq:d1cond}. 
\end{prop}

For the rest of this paper we assume that
$\F$ is algebraically closed.

\section{Leonard pairs with diameter two in TD/D form}
\label{sec:LPd2}
\ifDRAFT {\rm sec:LPd2}. \fi

In the previous section we described the near-bipartite Leonard pairs
with diameter one.
In the present section and the next,
we describe the near-bipartite Leonard pairs with diameter two.
Consider the following matrices in $\Mat_3(\F)$:
\begin{align*}
A &=
 \begin{pmatrix}
  a_0 & x_1  & 0\\
  1    & a_1 & x_2 \\
  0 & 1 & a_2
 \end{pmatrix},
&
A^* &=
 \begin{pmatrix}
  \th^*_0 & 0 & 0 \\
  0 & \th^*_1 & 0 \\
  0 & 0 & \th^*_2
 \end{pmatrix}.
\end{align*}
In this section we prove the following result.

\begin{prop}    \label{prop:d2key}    \samepage
\ifDRAFT {\rm prop:d2key}. \fi
The pair $A,A^*$ is Leonard pair over $\F$ if and only if
{\small
\begin{align}
& \qquad\qquad
 \th^*_i \neq \th^*_j   \quad (0 \leq i < j \leq 2),  
\qquad\qquad\qquad
x_i \neq 0   \quad (1 \leq i \leq 2),                              \label{eq:d2distinct}
\\
& \frac{x_1}{\th^*_2 - \th^*_1} + \frac{x_2}{\th^*_0 - \th^*_1}
 =
 \frac{a_0 - a_2} { (\th^*_0 - \th^*_2)^2}
  \bigg(
   a_0 (\th^*_0 - \th^*_1) + a_1 (\th^*_2 - \th^*_0)  + a_2 (\th^*_1 -\th^*_2) 
  \bigg),                                                               \label{eq:d2equat}
\\
& 0 \neq
 \frac{x_1} { (\th^*_2 - \th^*_1)^2} + \frac{x_2} { (\th^*_0 - \th^*_1)^2 }
  +  \frac {(a_0 - a_2)^2} { (\th^*_0 - \th^*_2)^2 },                          \label{eq:d2nonzero1}
\\
& 0 \neq  
\frac{ x_1 } { \th^*_2 - \th^*_1 } - \frac{ x_2 } { \th^*_0 - \th^*_1 }
 + \frac{ (a_0 - a_2)^2 } { 2 (\th^*_2 - \th^*_0)} 
  + \frac{ \big( a_0 (\th^*_0-\th^*_1) + a_1 (\th^*_2 - \th^*_0) + a_2 (\th^*_1 - \th^*_2)\big)^2}
           { 2 (\th^*_2 - \th^*_0)^3 }.
    \label{eq:d2nonzero2}
\end{align}
}
\end{prop}

\begin{lemma}    \label{lem:d2th1}    \samepage
\ifDRAFT {\rm lem:d2th1}. \fi
Assume that $A,A^*$ is a Leonard pair over $\F$,
and $\{\th_i\}_{i=0}^2$ is a standard ordering of the eigenvalues of $A$.
Then
\begin{equation}
\th_1 = \frac{ a_0 (\th^*_0 - \th^*_1) + a_2 (\th^*_1 - \th^*_2) }
                  { \th^*_0 - \th^*_2 }.                    \label{eq:d2th1}
\end{equation}
\end{lemma}

\begin{proof}
By Lemma \ref{lem:classify}(iii),
\[
 \vphi_2 = \phi_1 + (\th^*_2 - \th^*_0)(\th_1 - \th_2).    \label{eq:d2vphi2}
\]
In this equation, eliminate $\vphi_2$ and $\phi_1$ 
using \eqref{eq:adparam} and \eqref{eq:a0paramb},
and solve the result in $\th_1$.
\end{proof}

\begin{lemma}    \label{lem:d2equat}    \samepage
\ifDRAFT {\rm lem:d2equat}. \fi
Assume that $A,A^*$ is a Leonard pair over $\F$.
Then \eqref{eq:d2equat} holds.
\end{lemma}

\begin{proof}
By \eqref{eq:xiparam},
\begin{align*}
x_1 &= - \vphi_1 \phi_1
          \frac{ \th^*_1 - \th^*_2 } { (\th^*_0 - \th^*_1)^2 (\th^*_0 - \th^*_2) },
&
x_2 &= - \vphi_2 \phi_2
         \frac{ \th^*_0 - \th^*_1 } { (\th^*_0 - \th^*_2) (\th^*_1- \th^*_2)^2 }.
\end{align*}
In these equations, eliminate $\vphi_1$, $\vphi_2$, $\phi_1$, $\phi_2$
using \eqref{eq:a0param}, \eqref{eq:adparam}, \eqref{eq:a0paramb}, \eqref{eq:adparamb}.
This gives 
\begin{align*}
x_1 &= - \frac{(a_0 - \th_0) (a_0 - \th_2) (\th^*_1 - \th^*_2) }
                   { \th^*_0 - \th^*_2},  
&
x_2 &= - \frac{(a_2 - \th_0)(a_2 - \th_2) (\th^*_0 - \th^*_1) }
                 { \th^*_0 - \th^*_2 }. 
\end{align*}
So
\begin{equation}
\frac{x_1}{\th^*_1-\th^*_2} - \frac{x_2}{\th^*_0-\th^*_1}
= - \frac{(a_0 - \th_0)(a_0 - \th_2) - (a_2 - \th_0)(a_2 - \th_2) }
             {\th^*_0 - \th^*_2}.                             \label{eq:d2x1x2aux}
\end{equation}
Considering the trace of $A$,
\[
 \th_0 + \th_1 + \th_2 = a_0 + a_1 + a_2.
\]
Using this and \eqref{eq:d2th1}, one finds that
\begin{align*}
&(a_0 - \th_0)(a_0 - \th_2) - (a_2 - \th_0)(a_2 - \th_2)
\\
& \qquad\qquad
 = \frac{ (a_0 - a_2) \big( a_0 (\th^*_0 - \th^*_1)+a_1(\th^*_2-\th^*_0)+a_2(\th^*_1-\th^*_2) \big) }
          { \th^*_0 - \th^*_2 }.          
\end{align*}
By this and \eqref{eq:d2x1x2aux} we get \eqref{eq:d2equat}.
\end{proof}

\begin{lemma}    \label{lem:d2th0th2}    \samepage
\ifDRAFT {\rm lem:d2th0th2}. \fi
Assume that $A,A^*$ is a Leonard pair over $\F$.
Let $\{\th_i\}_{i=0}^2$ denote a standard ordering of the eigenvalues of $A$.
Then
\begin{align}
\th_0 + \th_2 
 &= \frac{ a_0 (\th^*_1 - \th^*_2) + a_1 (\th^*_0 - \th^*_2) + a_2 (\th^*_0 - \th^*_1) }
          { \th^*_0 - \th^*_2},                                 \label{eq:d2th0+th2}
\\
\th_0 \th_2 &=
\frac{a_1 \big( a_0(\th^*_1-\th^*_2) + a_2(\th^*_0 - \th^*_1) \big) } {\th^*_0 - \th^*_2}
+ \frac{ (a_0 - a_2)^2(\th^*_1 - \th^*_0)(\th^*_1-\th^*_2) } {(\th^*_0-\th^*_2)^2}
- x_1 - x_2.                                   \label{eq:d2th0th2}
\end{align}
\end{lemma}

\begin{proof}
Let $f(\lambda)$ denote the characteristic polynomial of $A$.
Then
\begin{align*}
f(\lambda) = \lambda^3 - (a_0 + a_1 + a_2) \lambda^2
 &+ (a_0 a_1 + a_0 a_2 + a_1 a_2 - x_1 - x_2) \lambda
\\ &
 + a_0 x_2 + a_2 x_1 - a_0 a_1 a_2.
\end{align*}
We have
$f(\lambda) = (\lambda-\th_0)(\lambda-\th_1)(\lambda-\th_2)$.
Comparing the previous two equations, we obtain
\begin{align}
 \th_0 + \th_2 &= a_0 + a_1 + a_2 - \th_1,    \label{eq:d2beta}
\\
 \th_0 \th_2 &= a_0 a_1 + a_0 a_2 + a_1 a_2 - x_1 - x_2 - \th_1(\th_0 + \th_2).  \label{eq:d2gamma}
\end{align}
In \eqref{eq:d2beta}, eliminate $\th_1$ using \eqref{eq:d2th1} to get \eqref{eq:d2th0+th2}.
In \eqref{eq:d2gamma}, eliminate $\th_1$ using \eqref{eq:d2th1}, and simplify
the result using \eqref{eq:d2equat} to get \eqref{eq:d2th0th2}.
\end{proof}

\begin{lemma}    \label{lem:d2difference}    \samepage
\ifDRAFT {\rm lem:d2difference}. \fi
Assume that $A,A^*$ is a Leonard pair over $\F$.
Let $\{\th_i\}_{i=0}^2$ denote a standard ordering of the
eigenvalues of $A$.
Then
{\small
\begin{align*}
& \frac{ (\th_1 - \th_0)(\th_1 - \th_2)}
       { (\th^*_1 - \th^*_0)(\th^*_1 - \th^*_2) }
  = \frac{x_1} { (\th^*_2 - \th^*_1)^2} + \frac{x_2} { (\th^*_0 - \th^*_1)^2 }
  +  \frac {(a_0 - a_2)^2} { (\th^*_0 - \th^*_2)^2 },
\\
& \frac{ (\th_0 - \th_2)^2 } { 2(\th^*_2 - \th^*_0) }
  = \frac{ x_1 } { \th^*_2 - \th^*_1 } - \frac{ x_2 } { \th^*_0 - \th^*_1 }
 + \frac{ (a_0 - a_2)^2 } { 2 (\th^*_2 - \th^*_0)} 
  + \frac{ \big( a_0 (\th^*_0-\th^*_1) + a_1 (\th^*_2 - \th^*_0) + a_2 (\th^*_1 - \th^*_2)\big)^2}
           { 2 (\th^*_2 - \th^*_0)^3 }.
\end{align*}
}
\end{lemma}

\begin{proof}
We have
\begin{align*}
(\th_1 - \th_0)(\th_1 - \th_2)
 &= \th_1^2 - (\th_0 + \th_2) \th_1 + \th_0 \th_2,
\\
(\th_0 - \th_2)^2
 &= (\th_0 + \th_2)^2 - 4 \th_0 \th_2.
\end{align*}
In the right-hand sides of these equations, eliminate $\th_1$, $\th_0 + \th_2$, $\th_0 \th_2$
using \eqref{eq:d2th1}, \eqref{eq:d2th0+th2}, \eqref{eq:d2th0th2}, 
and simplify the result using \eqref{eq:d2equat}.
\end{proof}

\begin{proofof}{Proposition \ref{prop:d2key}}
First assume that $A,A^*$ is a Leonard pair over $\F$.
Clearly \eqref{eq:d2distinct} holds.
By Lemma \ref{lem:d2equat},  \eqref{eq:d2equat} holds.
The scalars $\{\th_i\}_{i=0}^2$ are mutually distinct.
By this and Lemma \ref{lem:d2difference} we get \eqref{eq:d2nonzero1}
and \eqref{eq:d2nonzero2}.

Next assume that \eqref{eq:d2distinct}--\eqref{eq:d2nonzero2} hold.
Define scalars $\{\th_i\}_{i=0}^2$ that satisfy \eqref{eq:d2th1},
\eqref{eq:d2th0+th2}, \eqref{eq:d2th0th2}.
Define scalars 
\begin{align}
\vphi_1 &= (a_0 - \th_0)(\th^*_0 - \th^*_1),
&
\vphi_2 &= (a_2 - \th_2)(\th^*_2 - \th^*_1),           \label{eq:d2defvphi}
\\
\phi_1 &= (a_0 - \th_2)(\th^*_0 - \th^*_1),
&
\phi_2 &= (a_2 - \th_0)(\th^*_2 - \th^*_1).            \label{eq:d2defphi}
\end{align}
Our first goal is to show that the sequence
\begin{equation}
(\{\th_i\}_{i=0}^2; \{\th^*_i\}_{i=0}^2; \{\vphi_i\}_{i=1}^2; \{\phi_i\}_{i=1}^2)
                    \label{eq:d2parray2}
\end{equation}
is a parameter array over $\F$.
We have
\[
 (a_0 - \th_0)(a_0 - \th_2) = a_0^2 - (\th_0 + \th_2) a_0 + \th_0 \th_2.
\]
In this equation, eliminate the factors $\th_0 + \th_2$ and $\th_0 \th_2$
using \eqref{eq:d2th0+th2} and \eqref{eq:d2th0th2}.
In the result, eliminate $x_2$ using \eqref{eq:d2equat} to get
\begin{equation}
(a_0 - \th_0)(a_0 - \th_2)
 = - \frac{ \th^*_0 - \th^*_2} {\th^*_1 - \th^*_2 } x_1.            \label{eq:d2a0}
\end{equation}
In a similar way we get
\begin{equation}
(a_2 - \th_0)(a_2 - \th_2)
 = - \frac{\th^*_0 - \th^*_2 } { \th^*_0 - \th^*_1} x_2.            \label{eq:d2a2}
\end{equation}
We show that
\begin{align}
x_1 &= - \vphi_1 \phi_1  
          \frac{\th^*_1 - \th^*_2}
                 { (\th^*_0 - \th^*_1)^2 (\th^*_0 - \th^*_2) },     \label{eq:d2x1}
\\
x_2 &= - \vphi_2 \phi_2
         \frac{ \th^*_0 - \th^*_1}
                { (\th^*_0 - \th^*_2) (\th^*_1 - \th^*_2)^2}.      \label{eq:d2x2}
\end{align}
In the right-hand side of \eqref{eq:d2x1},
eliminate $\vphi_1$, $\phi_1$ using \eqref{eq:d2defvphi}, \eqref{eq:d2defphi},
and simplify the result using \eqref{eq:d2a0} to get \eqref{eq:d2x1}.
The line \eqref{eq:d2x2} is similarly obtained.
We now verify conditions (i)--(v) in Lemma \ref{lem:classify}.
We are assuming $d=2$.
We first verify that condition (i) holds.
The scalars $\{\th^*_i\}_{i=0}^2$ are mutually distinct by \eqref{eq:d2distinct}.
As in the proof of Lemma \ref{lem:d2difference}, we get the
two equations in Lemma \ref{lem:d2difference}.
By this and  \eqref{eq:d2nonzero1}, \eqref{eq:d2nonzero2} we find that
the scalars $\{\th_i\}_{i=0}^2$ are mutually distinct.
By these comments, condition (i) holds.
Condition (ii) holds by \eqref{eq:d2x1} and \eqref{eq:d2x2}.
Conditions (iii) and (iv) hold by \eqref{eq:d2th1},
\eqref{eq:d2defvphi}, \eqref{eq:d2defphi}.
Condition (v) vacuously holds.
We have shown that \eqref{eq:d2parray2} is a parameter array over $\F$.
Our next goal is to show that the sequence
\begin{equation}
  (\{a_i\}_{i=0}^2; \{x_i\}_{i=1}^2; \{\th^*_i\}_{i=0}^2)    \label{eq:d2TDDseq}
\end{equation}
is a TD/D sequence.
To do this we verify \eqref{eq:a0param}--\eqref{eq:xiparam}.
Lines \eqref{eq:a0param}, \eqref{eq:adparam} hold by 
\eqref{eq:d2defvphi} and \eqref{eq:d2defphi}.
Using \eqref{eq:d2defvphi} we find that
\[
  \frac{ \vphi_1 } {\th^*_1 - \th^*_0} + \frac{ \vphi_2 } { \th^*_1 - \th^*_2}
   = \th_0 + \th_2 - a_0 - a_2.
\]
In this line, use $\th_0 + \th_1 + \th_2 = a_0 + a_1 + a_2$ to get \eqref{eq:aiparam}.
Line \eqref{eq:xiparam} holds by \eqref{eq:d2x1} and \eqref{eq:d2x2}.
We have shown that the sequence \eqref{eq:d2TDDseq} is a TD/D sequence.
Now by Proposition \ref{prop:TDD2}, $A,A^*$ is a Leonard pair over $\F$.
\end{proofof}

\section{The classification of near-bipartite Leonard pairs with diameter two}
\label{sec:d2}
\ifDRAFT {\rm sec:d2}. \fi

In this section, we classify up to isomorphism the near-bipartite Leonard pairs with diameter two.
Moreover, for each near-bipartite Leonard pair with diameter two, we describe
its bipartite contraction.
Also, for each bipartite Leonard pair with diameter two,
we describe its near-bipartite expansions.

In view of Proposition \ref{prop:TDD2}, we consider 
the following matrices in $\Mat_3(\F)$:
\begin{align*}
A &=
 \begin{pmatrix}
  a_0 & x_1  & 0\\
  1    & a_1 & x_2 \\
  0 & 1 & a_2
 \end{pmatrix},
&
B &=
 \begin{pmatrix}
  0 & x_1  & 0\\
  1    & 0 & x_2 \\
  0 & 1 & 0
 \end{pmatrix},
&
A^* &=
 \begin{pmatrix}
  \th^*_0 & 0 & 0 \\
  0 & \th^*_1 & 0 \\
  0 & 0 & \th^*_2
 \end{pmatrix}.
\end{align*}
The following result classifies the bipartite Leonard pairs with diameter two.

\begin{prop}    \label{prop:d2pre}   \samepage
\ifDRAFT {\rm prop:d2pre}. \fi
The pair $B,A^*$ is a Leonard pair over $\F$ if and only if
\begin{align}
\th^*_i &\neq \th^*_j \quad (0 \leq i < j \leq 2),
&
x_i &\neq 0 \quad (1 \leq i \leq 2),                      \label{eq:d2distinctb}
\end{align}
\begin{equation}
\frac{x_1} { \th^*_1 - \th^*_2}
= \frac{x_2} {\th^*_0 - \th^*_1}.        \label{eq:d2cond}
\end{equation}
\end{prop}

\begin{proof}
Apply Proposition \ref{prop:d2key} with $a_i=0$ $(0 \leq i \leq 2)$.
\end{proof}

The following result classifies the near-bipartite Leonard pairs with diameter two,
and describes its bipartite contraction.

\begin{prop}   \label{prop:d2}    \samepage
\ifDRAFT {\rm prop:d2}. \fi
Assume that $A,A^*$ is a Leonard pair over $\F$.
Then the pair $A,A^*$ is near-bipartite if and only if 
\eqref{eq:d2cond} holds.
In this case, $B,A^*$ is the bipartite contraction of $A,A^*$.
\end{prop}

\begin{proof}
By Propositions \ref{prop:d2key} and \ref{prop:d2pre}.
\end{proof}

Our next goal is to determine all the near-bipartite expansions of
a given bipartite Leonard pair.
We will invoke Proposition \ref{prop:d2key}.
To do this, we consider
what conditions \eqref{eq:d2equat}--\eqref{eq:d2nonzero2}
become under the assumption that $B,A^*$ is a Leonard pair.
We first consider the condition \eqref{eq:d2equat}.

\begin{lemma}     \label{lem:d2equat2}    \samepage
\ifDRAFT {\rm lem:d2equat2}. \fi
Assume that $B,A^*$ is a Leonard pair over $\F$.
Then \eqref{eq:d2equat} holds if and only if
one of the following {\rm (i), (ii)} holds:
\begin{itemize}
\item[\rm (i)]
$a_0 = a_2$;
\item[\rm (ii)]
$a_0 \neq a_2$ and
\begin{equation}
a_0 (\th^*_0 - \th^*_1) + a_2 (\th^*_1 - \th^*_2) + a_1 (\th^*_2 - \th^*_0) = 0.    \label{eq:d2aux4}
\end{equation}
\end{itemize}
\end{lemma}

\begin{proof}
Note by Proposition \ref{prop:d2pre} that \eqref{eq:d2cond} holds, and so
in \eqref{eq:d2equat} the left-hand side is zero.
First assume that \eqref{eq:d2equat} holds.
In \eqref{eq:d2equat} the right-hand side is zero. 
Thus one of (i), (ii) holds.
Next assume that one of (i), (ii) holds.
Then in \eqref{eq:d2equat} the right-hand side is zero.
So \eqref{eq:d2equat} holds.
\end{proof}

Next we consider conditions \eqref{eq:d2nonzero1}
and \eqref{eq:d2nonzero2}.

\begin{lemma}    \label{lem:d2mfree}    \samepage
\ifDRAFT {\rm lem:d2mfree}. \fi
Assume that $B,A^*$ is a Leonard pair over $\F$.
Then \eqref{eq:d2nonzero1} and \eqref{eq:d2nonzero2} hold if and only if one of the following holds:
\begin{itemize}
\item[\rm (i)]
$a_0 = a_2$, and
\begin{equation}
(a_0 - a_1)^2 + 4 (x_1 + x_2) \neq 0 ;             \label{eq:d2condx1a}
\end{equation}
\item[\rm (ii)]
$a_0 \neq a_2$, and
\begin{align}
(a_0 - a_2)^2 + 4(x_1 + x_2) &\neq 0,                                  \label{eq:d2condx1b}                 
\\
(a_0 - a_2)^2 + \frac{  (x_1 + x_2)^3 } {x_1 x_2 } &\neq 0.                     \label{eq:d2condx1c} 
\end{align}
\end{itemize}
\end{lemma}

\begin{proof}
Note by Proposition \ref{prop:d2pre} that \eqref{eq:d2cond} holds.
Using \eqref{eq:d2cond}, one routinely finds that
\begin{equation}
(\th^*_0 - \th^*_2)
\left( \frac{x_1} { \th^*_1 - \th^*_2 } + \frac{x_2}{\th^*_0-\th^*_1} \right)
= 2 (x_1 + x_2),                          \label{eq:d2x1+x2}
\end{equation}
\begin{equation}
(\th^*_0 - \th^*_2)^2
\left( \frac{x_1}{(\th^*_1 - \th^*_2)^2} + \frac{x_2}{(\th^*_0 - \th^*_1)^2} \right)
=\frac{(x_1 + x_2)^3}{x_1 x_2}.              \label{eq:d2x1+x2no2}
\end{equation}                    
For each of the cases $a_0 = a_2$ and $a_0 \neq a_2$,
simplify \eqref{eq:d2nonzero1} and \eqref{eq:d2nonzero2}
using \eqref{eq:d2x1+x2}, \eqref{eq:d2x1+x2no2}.
For the case $a_0 \neq a_2$, also use \eqref{eq:d2aux4}.
\end{proof}

\begin{prop}    \label{prop:d2converse}    \samepage
\ifDRAFT {\rm prop:d2converse}. \fi
Assume that $B,A^*$ is a Leonard pair over $\F$.
Then $A,A^*$ is a near-bipartite expansion of $B,A^*$ if and only if
one of the following {\rm (i), (ii)} holds:
\begin{itemize}
\item[\rm (i)]
$a_0 = a_2$ and \eqref{eq:d2condx1a} holds;
\item[\rm (ii)]
$a_0 \neq a_2$ and \eqref{eq:d2aux4}, \eqref{eq:d2condx1b}, \eqref{eq:d2condx1c} hold.
\end{itemize}
\end{prop}

\begin{proof}
By Proposition \ref{prop:d2key} and Lemmas \ref{lem:d2equat2}, \ref{lem:d2mfree}.
\end{proof}

\section{The type of a Leonard pair}
\label{sec:type}
\ifDRAFT {\rm sec:type}. \fi

In Sections \ref{sec:d1} and \ref{sec:d2},
we classified the near-bipartite Leonard pairs of diameter one and two.
To classify the near-bipartite Leonard pairs of diameter at least $3$, 
we will divide the arguments into some cases.
To describe these cases, we recall from \cite{NT:balanced} the type of a Leonard pair.

For the rest of this paper, we assume that $d \geq 3$.

\begin{defi}     {\rm (See \cite[Section 4]{NT:balanced}.) }
\label{def:type}    \samepage
\ifDRAFT {\rm def:type}. \fi
Let $A,A^*$ denote a Leonard pair over $\F$ with diameter $d$.
Let $\beta$ denote the fundamental constant of $A,A^*$.
We define the {\em type} of $A,A^*$ as follows.
\[
\begin{array}{c|c}
\text{\rm type} & \text{\rm description}
\\ \hline
\text{\rm I} &  \beta \neq 2, \;\; \beta \neq -2          \rule{0mm}{2.5ex}
\\
\text{\rm II} &  \beta = 2         \rule{0mm}{2.5ex}
\\
\text{\rm III}^+ &    \beta = -2, 
                             \;\; \text{\rm $d$ even}                             \rule{0mm}{2.5ex}
\\
\text{\rm III}^- &    \beta = -2, 
\;\; \text{\rm $d$ odd}                                                     \rule{0mm}{2.5ex}
\end{array}
\]
\end{defi}

\begin{note}     \label{note:type}    \samepage
\ifDRAFT {\rm note:beta}. \fi
By Definition \ref{def:beta0},
the fundamental constant of a Leonard pair $A,A^*$ is determined by $A^*$.
By this and Definition \ref{def:nearbip},
any near-bipartite Leonard pair has the same type as its bipartite contraction.
\end{note}

\begin{defi}  \label{def:typearray}    \samepage
\ifDRAFT {\rm def:typearray}. \fi
By the {\em type of a parameter array} we mean the type of the
corresponding Leonard pair.
\end{defi}

\begin{lemma}    \label{lem:typeaffine}   \samepage
\ifDRAFT {\rm lem:typeaffine}. \fi
Let $A,A^*$ denote a Leonard pair over $\F$.
For scalars $\xi$, $\zeta$, $\xi^*$, $\zeta^*$ in $\F$ with $\xi \xi^* \neq 0$,
the Leonard pair $\xi A + \zeta I, \xi^* A^* + \zeta^* I$
has the same type as $A,A^*$.
\end{lemma}

\begin{proof}
By Lemma \ref{lem:affineLPparam} and  Definition \ref{def:beta0},
the Leonard pairs $A,A^*$ and $\xi A + \zeta I, \xi^* A^* + \zeta^* I$
have the same fundamental constant.
\end{proof}

By \cite[Theorem 9.6]{NT:balanced} a Leonard pair
of type III$^-$ is not bipartite, so not near-bipartite by Note \ref{note:type}.
In Sections \ref{sec:type1}--\ref{sec:type3+} we will describe
the parameter arrays of types  I, II, III$^+$.

\section{Parameter arrays of type I}
\label{sec:type1}
\ifDRAFT {\rm sec:type1}. \fi

In this section,
we describe the parameter arrays  of type I.
Throughout this section,
fix a nonzero $q \in \F$ such that $q^4 \neq 1$.

\begin{lemma}  {\em (See \cite[Lemma 16.1]{NT:K}.) }
\label{lem:type1param}   \samepage
\ifDRAFT {\rm lem:type1param}. \fi
For a sequence 
\begin{equation}
  (\delta, \mu, h, \delta^*, \mu^*, h^*, \tau)       \label{eq:type1pseq}
\end{equation}
of scalars in $\F$, define
\begin{align}
 \th_i &= \delta + \mu q^{2i-d} + h q^{d-2i},                  \label{eq:type1th}
\\
 \th^*_i &= \delta^* + \mu^* q^{2i-d} + h^* q^{d-2i}         \label{eq:type1ths}
\end{align}
for $0 \leq i \leq d$ and
\begin{align}
 \vphi_i &= (q^i-q^{-i})(q^{d-i+1}-q^{i-d-1})(\tau - \mu \mu^* q^{2i -d-1} - h h^* q^{d-2i+1}),
                                                                              \label{eq:type1vphi}
\\
 \phi_i &=  (q^i-q^{-i})(q^{d-i+1}-q^{i-d-1})(\tau - h \mu^* q^{2i -d-1} - \mu h^* q^{d-2i+1})  
                                               \label{eq:type1phi}
\end{align}
for $1 \leq i \leq d$.
Then the sequence
\begin{equation}
   (\{\th_i\}_{i=0}^d; \{\th^*_i\}_{i=0}^d; \{\vphi_i\}_{i=1}^d; \{\phi_i\}_{i=1}^d)        \label{eq:type1parray}
\end{equation}
is a parameter array over $\F$ that has type I and fundamental constant $\beta = q^2 + q^{-2}$,
provided that the inequalities in Lemma \ref{lem:classify}{\rm (i),(ii)} hold.
Conversely, assume that the sequence \eqref{eq:type1parray} is a parameter array over $\F$ that has type I
and fundamental constant $\beta = q^2 + q^{-2}$.
Then there exists a unique sequence \eqref{eq:type1pseq} of scalars in $\F$ that satisfies
\eqref{eq:type1th}--\eqref{eq:type1phi}.
\end{lemma}

\begin{lemma}    {\rm (See \cite[Lemma 16.4]{NT:K}.) }
\label{lem:type1cond}   \samepage
\ifDRAFT {\rm lem::type1cond}. \fi
Referring to Lemma  \ref{lem:type1param},
the inequalities in Lemma \ref{lem:classify}{\rm (i),(ii)} hold
if and only if
\begin{align}
 & q^{2i} \neq 1 &&  (1 \leq i \leq d),                       \label{eq:type1paramcond1}
\\
 & \mu \neq h q^{2i}  &&  (1-d \leq i \leq d-1),          \label{eq:type1paramcond2}
\\
 & \mu^* \neq h^* q^{2i}  && (1-d \leq i \leq d-1),      \label{eq:type1paramcond3}
\\
 & \tau \neq \mu \mu^* q^{2i-d-1} + h h^* q^{d-2i+1}  &&  (1 \leq i \leq d),  \label{eq:type1paramcond4}
\\
 & \tau \neq h \mu^* q^{2i-d-1} + \mu h^* q^{d-2i+1}  &&  (1 \leq i \leq d).   \label{eq:type1paramcond5}
\end{align}
\end{lemma}

\begin{defi}    \label{def:type1basic}    \samepage
\ifDRAFT {\rm def:type1basic}. \fi
A {\em primary $q$-data} is a sequence \eqref{eq:type1pseq} of scalars in $\F$ 
that satisfy \eqref{eq:type1paramcond1}--\eqref{eq:type1paramcond5}.
The primary $q$-data \eqref{eq:type1pseq} and the parameter array  \eqref{eq:type1parray}
are said to {\em correspond}.
\end{defi}

\begin{defi}    \label{def:type1basic2}    \samepage
\ifDRAFT {\rm def:type1basic2}. \fi
Let $A,A^*$ denote a Leonard pair of type I.
By {\em a primary $q$-data of $A,A^*$}
we mean the primary $q$-data that corresponds to a parameter array
of $A,A^*$.
\end{defi}

\begin{lemma}    \label{lem:type1paramDown}    \samepage
\ifDRAFT {\rm lem:type1paramDown}. \fi
For the parameter arrays in Lemma \ref{lem:parrayunique},
consider the corresponding primary $q$-data.
These primary $q$-data are related as follows:
\[
\begin{array}{lcl}
\qquad\qquad \text{\rm parameter array}  & &  \qquad \text{\rm primary $q$-data}
\\ \hline
(\{\th_i\}_{i=0}^d; \{\th^*_i\}_{i=0}^d; \{\vphi_i\}_{i=1}^d; \{\phi_i\}_{i=1}^d)  
& & \quad
(\delta, \mu, h, \delta^*, \mu^*, h^*, \tau)                             \rule{0mm}{3ex}
\\
 (\{\th_i\}_{i=0}^d; \{\th^*_{d-i}\}_{i=0}^d; \{\phi_{d-i+1}\}_{i=1}^d; \{\vphi_{d-i+1} \}_{i=1}^d)  
 & &  \quad
(\delta, \mu, h, \delta^*, h^*, \mu^*, \tau)                            \rule{0mm}{2.7ex}
\\
(\{\th_{d-i}\}_{i=0}^d; \{\th^*_i\}_{i=0}^d; \{\phi_i\}_{i=1}^d; \{\vphi_i\}_{i=1}^d)  
& & \quad
(\delta, h, \mu, \delta^*, \mu^*, h^*, \tau)                             \rule{0mm}{2.7ex}
\\
(\{\th_{d-i}\}_{i=0}^d; \{\th^*_{d-i}\}_{i=0}^d; \{\vphi_{d-i+1}\}_{i=1}^d; \{\phi_{d-i+1}\}_{i=1}^d)  
& &  \quad
(\delta, h, \mu, \delta^*, h^*, \mu^*, \tau)                             \rule{0mm}{2.7ex}   
\end{array}
\]
\end{lemma}

\begin{proof}
Use Lemma \ref{lem:type1param} and Definition \ref{def:type1basic}.
\end{proof}

\section{Parameter arrays of type II}
\label{sec:type2}
\ifDRAFT {\rm sec:type2}. \fi

In this section we describe the parameter arrays of type II.

\begin{lemma}   {\rm (See \cite[Lemma 19.1]{NT:K}.)}
\label{lem:type2param}    \samepage
\ifDRAFT {\rm lem:type2param}. \fi
For a sequence
\begin{equation}
   (\delta, \mu, h, \delta^*, \mu^*, h^*, \tau)                \label{eq:type2pseq}
\end{equation}
of scalars in $\F$,  define
\begin{align}
\th_i &= \delta + \mu (i-d/2)+ h i (d-i),                \label{eq:type2th}
\\
\th^*_i &= \delta^* + \mu^* (i-d/2) + h^* i (d-i)     \label{eq:type2ths}
\end{align}
for $0 \leq i \leq d$ and
\begin{align}
\vphi_i &= i (d-i+1) \big(\tau - \mu \mu^*/2 + 
  (h \mu^* + \mu h^*)(i-(d+1)/2) + h h^* (i-1)(d-i) \big),             \label{eq:type2vphi}
\\
\phi_i &= i (d-i+1) \big(\tau + \mu \mu^*/2+
         (h \mu^* - \mu h^*)(i-(d+1)/2)+h h^*(i-1)(d-i) \big)         \label{eq:type2phi}
\end{align}
for $1 \leq i \leq d$.
Then the sequence
\begin{equation}
   (\{\th_i\}_{i=0}^d; \{\th^*_i\}_{i=0}^d; \{\vphi_i\}_{i=1}^d; \{\phi_i\}_{i=1}^d)        \label{eq:type2parray}
\end{equation}
is a parameter array over $\F$ that has type II,
provided that the inequalities in Lemma \ref{lem:classify}{\rm (i),(ii)} hold.
Conversely, assume that the sequence \eqref{eq:type2parray} is a parameter 
array over $\F$ that has type II.
Then there exists a unique sequence \eqref{eq:type2pseq} of scalars in $\F$ that satisfies
\eqref{eq:type2th}--\eqref{eq:type2phi}.
\end{lemma}

\begin{lemma}    {\rm (See \cite[Lemma 19.4]{NT:K}.) }
\label{lem:type2cond}   \samepage
\ifDRAFT {\rm lem::type2cond}. \fi
Referring to Lemma \ref{lem:type2param}, 
the inequalities in Lemma \ref{lem:classify}{\rm (i),(ii)} hold
if and only if
\begin{align}
 & \text{\rm $\text{\rm Char}(\F)$ is equal to $0$ or greater than $d$},          \label{eq:type2paramcond1}
\\
 & \mu \neq h i  \qquad\qquad \;\;\,  (1-d \leq i \leq d-1),          \label{eq:type2paramcond2}
\\
 & \mu^* \neq h^* i \qquad\qquad (1-d \leq i \leq d-1),     \label{eq:type2paramcond3}
\\
 & \tau \neq \mu \mu^*/2 - 
       (h \mu^* + \mu h^*)(i-(d+1)/2) -  h h^* (i-1)(d-i)   & (1 \leq i \leq d),  \label{eq:type2paramcond4}
\\
& \tau \neq  - \mu \mu^*/2 - 
       (h \mu^* - \mu h^*)(i-(d+1)/2) -  h h^* (i-1)(d-i)  & (1 \leq i \leq d). \label{eq:type2paramcond5}
\end{align}
\end{lemma}

\begin{defi}    \label{def:type2basic}    \samepage
\ifDRAFT {\rm def:type2basic}. \fi
A {\em primary data of type II} is a sequence \eqref{eq:type2pseq} of scalars in $\F$ 
that satisfy \eqref{eq:type2paramcond1}--\eqref{eq:type2paramcond5}.
The primary data \eqref{eq:type2pseq} of type II and the parameter array \eqref{eq:type2parray}
are said to {\em correspond}.
\end{defi}

\begin{defi}   \label{def:type2basic2}    \samepage
\ifDRAFT {\rm def:type2basic2}. \fi
Let $A,A^*$ denote a Leonard pair of type II.
By a {\em primary data of $A,A^*$} we mean the primary data of type II
that corresponds to a parameter array of $A,A^*$.
\end{defi}

\begin{lemma}    \label{lem:type2paramDown}    \samepage
\ifDRAFT {\rm lem:type2paramDown}. \fi
For the parameter arrays in Lemma \ref{lem:parrayunique},
consider the corresponding primary data of type II.
These primary data are related as follows:
\[
\begin{array}{lcl}
\qquad\qquad \text{\rm parameter array}  & &  \qquad \text{\rm primary data}
\\ \hline
(\{\th_i\}_{i=0}^d; \{\th^*_i\}_{i=0}^d; \{\vphi_i\}_{i=1}^d; \{\phi_i\}_{i=1}^d)  
& & \quad
(\delta, \mu, h, \delta^*, \mu^*, h^*, \tau)                             \rule{0mm}{3ex}
\\
 (\{\th_i\}_{i=0}^d; \{\th^*_{d-i}\}_{i=0}^d; \{\phi_{d-i+1}\}_{i=1}^d; \{\vphi_{d-i+1} \}_{i=1}^d)  
 & &  \quad
  (\delta, \mu, h, \delta^*, - \mu^*, h^*, \tau)                            \rule{0mm}{2.7ex}
\\
(\{\th_{d-i}\}_{i=0}^d; \{\th^*_i\}_{i=0}^d; \{\phi_i\}_{i=1}^d; \{\vphi_i\}_{i=1}^d)  
& & \quad
(\delta, - \mu, h, \delta^*, \mu^*, h^*, \tau)                                \rule{0mm}{2.7ex}
\\
(\{\th_{d-i}\}_{i=0}^d; \{\th^*_{d-i}\}_{i=0}^d; \{\vphi_{d-i+1}\}_{i=1}^d; \{\phi_{d-i+1}\}_{i=1}^d)  
& &  \quad
 (\delta, - \mu, h, \delta^*, - \mu^*, h^*, \tau)                            \rule{0mm}{2.7ex}   
\end{array}
\]
\end{lemma}

\begin{proof}
Use Lemma \ref{lem:type2param} and Definition \ref{def:type2basic}.
\end{proof}

\section{Parameter arrays of type III$^+$}
\label{sec:type3+}
\ifDRAFT {\rm sec:type3+}.  \fi

In this section we describe the parameter arrays of type III$^+$.

\begin{lemma}    {\rm (See \cite[Lemma 22.1]{NT:K}.) }
\label{lem:type3+param}    \samepage
\ifDRAFT {\rm lem:type3+param}. \fi
Assume that $d$ is even.
For a sequence
\begin{equation}
   (\delta, s, h, \delta^*,  s^*, h^*, \tau)                \label{eq:type3+pseq}
\end{equation}
of scalars in $\F$,  define
\begin{align}
\theta_i &=
  \begin{cases}
     \delta+s+h(i-d/2)  & \text{\rm if $i$ is even}, \\
     \delta-s -h(i-d/2) & \text{\rm if $i$ is odd},
  \end{cases}        
            \label{eq:type3+th}
\\
\theta^*_i &=
   \begin{cases}
     \delta^* +s^* +h^*(i-d/2)  &   \text{\rm if $i$ is even}, \\
     \delta^* - s^* -h^*(i-d/2) &  \text{\rm if $i$ is odd}
   \end{cases}  
      \label{eq:type3+ths}
\end{align}
for $0 \leq i \leq d$ and
\begin{align}
\varphi_i &=
   \begin{cases}
      i \big(\tau-sh^*-s^*h-hh^*(i-(d+1)/2) \big) &   \text{\rm if $i$ is even}, \\
      (d-i+1) \big( \tau+sh^*+s^*h+hh^*(i-(d+1)/2) \big)  & \text{\rm if $i$ is odd},    
   \end{cases}                                                                 \label{eq:type3+vphi}
\\
\phi_i &=
   \begin{cases}
      i \big( \tau-sh^*+s^*h+hh^*(i-(d+1)/2) \big) &   \text{\rm if $i$ is even}, \\
      (d-i+1) \big( \tau+sh^*-s^*h - hh^*(i-(d+1)/2) \big)  & \text{\rm if $i$ is odd}
   \end{cases}                                                                \label{eq:type3+phi}
\end{align}
for $1 \leq i \leq d$.
Then the sequence
\begin{equation}
   (\{\th_i\}_{i=0}^d; \{\th^*_i\}_{i=0}^d; \{\vphi_i\}_{i=1}^d; \{\phi_i\}_{i=1}^d)        \label{eq:type3+parray}
\end{equation}
is a parameter array over $\F$ that has type III$^+$,
provided that the inequalities in Lemma \ref{lem:classify}{\rm (i),(ii)} hold.
Conversely, assume that the sequence \eqref{eq:type3+parray} is a parameter 
array over $\F$ that has type III$^+$.
Then there exists a unique sequence \eqref{eq:type3+pseq} of scalars in $\F$ that satisfies
\eqref{eq:type3+th}--\eqref{eq:type3+phi}.
\end{lemma}

\begin{lemma}    {\rm (See \cite[Lemma 22.4]{NT:K}.) }
\label{lem:type3+cond}   \samepage
\ifDRAFT {\rm lem::type3+cond}. \fi
Referring to Lemma \ref{lem:type3+param}, 
the inequalities in Lemma \ref{lem:classify}{\rm (i),(ii)} hold
if and only if
\begin{align}
 & \text{\rm $\text{\rm Char}(\F)$ is equal to $0$ or greater than $d/2$},          \label{eq:type3+paramcond1}
\\
 & h \neq 0,   \qquad h^* \neq 0,                                      \label{eq:type3+paramcond1b}
\\
 & 2 s \neq i h  \qquad\;\;\, \text{\rm if $i$ is odd} \qquad  (1-d \leq i \leq d-1),          \label{eq:type3+paramcond2}
\\
 & 2 s^* \neq i h^*    \qquad \text{\rm if $i$ is odd} \qquad  (1-d \leq i \leq d-1),     \label{eq:type3+paramcond3}
\\
 & \tau \neq
  \begin{cases}
   s h^* + s^* h + h h^* \big( i - (d+1)/2 \big)   &  \text{\rm if $i$ is even},
  \\
  - s h^* - s^* h - h h^* \big( i - (d+1)/2 \big)  &  \text{\rm if $i$ is odd} 
 \end{cases}
 \qquad \, (1 \leq i \leq d),       \label{eq:type3+paramcond4a}
\\
 & \tau \neq 
  \begin{cases}
    s h^* - s^* h - h h^* \big( i - (d+1)/2 \big)    & \text{\rm if $i$ is even}, 
  \\
  - s h^* + s^* h + h h^* \big( i - (d+1)/2 \big)   &  \text{\rm if $i$ is odd}
  \end{cases}
    \qquad \, (1 \leq i \leq d).       \label{eq:type3+paramcond4b}
\end{align}
\end{lemma}

\begin{defi}    \label{def:type3+basic}    \samepage
\ifDRAFT {\rm def:type3+basic}. \fi
A {\em primary data of type III$^+$} is
a sequence \eqref{eq:type3+pseq} of scalars in $\F$  that satisfy
\eqref{eq:type3+paramcond1}--\eqref{eq:type3+paramcond4b}.
The primary data \eqref{eq:type3+pseq} of type III$^+$ and the parameter array
\eqref{eq:type3+parray} are said to {\em correspond}.
\end{defi}

\begin{defi}   \label{def:type3+basic2}    \samepage
\ifDRAFT {\rm def:type3+basic2}. \fi
Let $A,A^*$ denote a Leonard pair of type III$^+$.
By a {\em primary data of $A,A^*$} we mean the primary data of type III$^+$
that corresponds to a parameter array of $A,A^*$.
\end{defi}

\begin{lemma}    \label{lem:type3+paramDown}    \samepage
\ifDRAFT {\rm lem:type3+paramDown}. \fi
For the parameter arrays in Lemma \ref{lem:parrayunique},
consider the corresponding primary data of type III$^+$.
These primary data are related as follows:
\[
\begin{array}{lcl}
\qquad\qquad \text{\rm parameter array}  & &  \qquad \text{\rm primary data}
\\ \hline
(\{\th_i\}_{i=0}^d; \{\th^*_i\}_{i=0}^d; \{\vphi_i\}_{i=1}^d; \{\phi_i\}_{i=1}^d)  
& & \quad
   (\delta, s, h,  \delta^*, s^*,  h^*, \tau)                         \rule{0mm}{3ex}
\\
 (\{\th_i\}_{i=0}^d; \{\th^*_{d-i}\}_{i=0}^d; \{\phi_{d-i+1}\}_{i=1}^d; \{\vphi_{d-i+1} \}_{i=1}^d)  
 & &  \quad
   (\delta, s, h,  \delta^*, s^*,  - h^*, \tau)                            \rule{0mm}{2.7ex}
\\
(\{\th_{d-i}\}_{i=0}^d; \{\th^*_i\}_{i=0}^d; \{\phi_i\}_{i=1}^d; \{\vphi_i\}_{i=1}^d)  
& & \quad
  (\delta, s, - h,  \delta^*, s^*,  h^*, \tau)                                \rule{0mm}{2.7ex}
\\
(\{\th_{d-i}\}_{i=0}^d; \{\th^*_{d-i}\}_{i=0}^d; \{\vphi_{d-i+1}\}_{i=1}^d; \{\phi_{d-i+1}\}_{i=1}^d)  
& &  \quad
  (\delta, s, - h,  \delta^*, s^*, - h^*, \tau)                               \rule{0mm}{2.7ex}   
\end{array}
\]
\end{lemma}

\begin{proof}
Use Lemma \ref{lem:type3+param} and Definition \ref{def:type3+basic}.
\end{proof}

\section{About the condition $\vphi_i\phi_i = \vphi'_i \phi'_i$}
\label{sec:vphiphi} 
\ifDRAFT {\rm sec:vphiphi}. \fi

In the proof of our main results,
Lemma \ref{lem:nbipTDD3} will play an essential role.
In that lemma we see the condition \eqref{eq:vphiphivphidphid}.
In this section we recall from \cite{NT:K} the necessary and sufficient conditions for
\eqref{eq:vphiphivphidphid} in terms of the primary data.

Let $A,A^*$ and $B,A^*$ denote Leonard pairs over $\F$.
Let
\begin{align}
& (\{ \th_i\}_{i=0}^d; \{\th^*_i\}_{i=0}^d; \{\vphi_i\}_{i=1}^d; \{\phi_i\}_{i=1}^d), 
&&  (\{ \th'_i\}_{i=0}^d; \{\th^*_i\}_{i=0}^d; \{\vphi'_i\}_{i=1}^d; \{\phi'_i\}_{i=1}^d) 
                                      \label{eq:parrays4b}
\end{align}
denote a parameter array of $A,A^*$ and $B,A^*$, respectively.

First consider the case of type I.

\begin{lemma}     {\rm (See \cite[Theorem 17.1]{NT:K}.) }
\label{lem:comptype1}    \samepage
\ifDRAFT {\rm lem:comptype1}. \fi
Assume that $A,A^*$ has type I.
For the parameter arrays in \eqref{eq:parrays4b},
let
\begin{align*}
&  (\delta, \mu, h, \delta^*, \mu^*, h^*, \tau),
&&  (\delta', \mu', h', \delta^*, \mu^*, h^*, \tau')
\end{align*}
denote the corresponding primary $q$-data.
Then $\vphi_i \phi_i = \vphi'_i \phi'_i$ $(1 \leq i \leq d)$
if and only if
\begin{align}
 \mu h &= \mu' h',               \label{eq:type1cond1}
\\
 \tau (\mu+h) &= \tau' (\mu' + h'),   \label{eq:type1cond2}
\\
\tau^2 + (\mu+h)^2 \mu^* h^* 
 &= \tau^{\prime 2} + (\mu' + h')^2 \mu^* h^*.     \label{eq:type1cond3}
\end{align}
\end{lemma}

Next consider the case of type II.

\begin{lemma}    {\rm (See \cite[Theorem 20.1]{NT:K}.)}
\label{lem:comptype2}    \samepage
\ifDRAFT {\rm lem:comptype2}. \fi
Assume that $A,A^*$ has type II.
For the parameter arrays in \eqref{eq:parrays4b},
let
\begin{align*}
& (\delta, \mu, h, \delta^*, \mu^*, h^*, \tau),  
&& 
 (\delta', \mu', h', \delta^*, \mu^*, h^*, \tau')
\end{align*}
denote the corresponding primary data.
Then $\vphi_i \phi_i = \vphi'_i \phi'_i$ $(1 \leq i \leq d)$ if and only if
\begin{align}
  h^2 &= h^{\prime 2},                          \label{eq:type2cond1}
\\
 2 h \tau + \mu^2 h^* &= 2 h' \tau' + \mu^{\prime 2} h^*,    \label{eq:type2cond2}
\\ 
 4 \tau^2 - \mu^2 \big(  \mu^{* 2} + (d-1)^2 h^{* 2} \big)
&=
  4 \tau^{\prime 2} - \mu^{\prime 2} \big(  \mu^{* 2} + (d-1)^2 h^{* 2} \big).   \label{eq:type2cond3}
\end{align}
\end{lemma}

Next consider the case of type III$^+$.

\begin{lemma}       {\rm (See \cite[Theorem 23.1]{NT:K}.)}
\label{lem:comptype3+}    \samepage
\ifDRAFT {\rm lem:comptype3+}. \fi
Assume that $A,A^*$ has type III$^+$.
For the parameter arrays in \eqref{eq:parrays4b},
let
\begin{align*}
& (\delta, s, h, \delta^*, s^*, h^*, \tau),  
&& 
 (\delta', s', h', \delta^*, s^*, h^*, \tau')
\end{align*}
denote the corresponding primary data.
Then $\vphi_i \phi_i = \vphi'_i \phi'_i$ $(1 \leq i \leq d)$ if and only if
\begin{align}
  h^2 &= h^{\prime 2},                 \label{eq:type3+cond1}
\\
 (\tau + s h^*)^2 &= (\tau' + s' h^*)^2,   \label{eq:type3+cond2}
\\ 
 (\tau- s h^*)^2 &= (\tau' - s' h^*)^2.    \label{eq:type3+cond3}
\end{align}
\end{lemma}

\section{Some characterizations of the essentially bipartite Leonard pairs}
\label{sec:bipessbip}
\ifDRAFT {\rm sec:bipessbip}. \fi

In this section, we characterize the essentially bipartite Leonard pairs
in terms of the primary data,
the parameter array,
and the TD/D sequence.

Throughout this section, let $A,A^*$ denote a Leonard pair over $\F$,
with parameter array
\begin{equation}
  (\{\th_i\}_{i=0}^d; \{\th^*_i\}_{i=0}^d; \{\vphi_i\}_{i=1}^d; \{\phi_i\}_{i=1}^d)  \label{eq:parray23b}
\end{equation}
and corresponding TD/D sequence
\[
 (\{a_i\}_{i=0}^d; \{x_i\}_{i=1}^d; \{\th^*_i\}_{i=0}^d).
\]

We first consider the case of type I.

\begin{lemma}    {\rm (See \cite[Theorem 6.11]{NT:balanced}.)}
 \label{lem:essbiptype1}    \samepage
\ifDRAFT {\rm lem:essbiptype1}. \fi
Fix a nonzero $q \in \F$,
and assume that $A,A^*$ has type I with fundamental constant $\beta = q^2 + q^{-2}$.
Let
\[
 (\delta, \mu, h, \delta^*, \mu^*, h^*, \tau)
\]
denote the primary $q$-data corresponding to the parameter array
\eqref{eq:parray23b}.
Then the following are equivalent:
\begin{itemize}
\item[\rm (i)]
$A,A^*$ is essentially bipartite;
\item[\rm (ii)]
$\mu + h = 0$ and $\tau =0$.
\end{itemize}
\end{lemma}

\begin{lemma}     \label{lem:essbiptype1b}     \samepage
\ifDRAFT {\rm lem:essbiptype1b}. \fi
Referring to Lemma \ref{lem:essbiptype1},
$A,A^*$ is essentially bipartite if and only if
\begin{align*}
\th_i &= \delta + \mu (q^{2i-d} - q^{d-2i})  &&  (0 \leq i \leq d),
\\
\th^*_i &= \delta^* + \mu^* q^{2i-d} + h^* q^{d-2i}   &&  (0 \leq i \leq d),
\\
\vphi_i &= - (q^i - q^{-i})(q^{d-i+1} - q^{i-d-1})\mu  (\mu^* q^{2i-d-1} - h^* q^{d-2i+1})  && (1 \leq i \leq d),
\\
\phi_i &= (q^i - q^{-i})(q^{d-i+1} - q^{i-d-1}) \mu (\mu^* q^{2i-d-1} - h^* q^{d-2i+1})  && (1 \leq i \leq d).
\end{align*}
\end{lemma}

\begin{proof}
By Lemmas \ref{lem:type1param} and \ref{lem:essbiptype1}.
\end{proof}

\begin{lemma}    \label{lem:essbiptype1bb}     \samepage
\ifDRAFT {\rm lem:essbiptype1bb}. \fi
Referring to Lemma \ref{lem:essbiptype1},
assume that $A,A^*$ is essentially bipartite.
Then
\begin{align*}
 & q^{2i} \neq 1 \qquad\qquad (1 \leq i \leq d),
\\
 & q^{2i} \neq -1 \qquad\qquad (1 \leq i \leq d-1),
\\
& \mu \neq 0,
\\
& \mu^* q^i \neq h^* q^{-i}   \qquad\qquad( 1-d \leq i \leq d-1).
\end{align*}
\end{lemma}

\begin{proof}
Evaluate the inequalities in Lemma \ref{lem:type1cond} using
Lemma \ref{lem:essbiptype1}(ii).
\end{proof}

\begin{lemma}     \label{lem:essbiptype1c}     \samepage
\ifDRAFT {\rm lem:essbiptype1c}. \fi
Referring to Lemma \ref{lem:essbiptype1},
$A,A^*$ is essentially bipartite if and only if
\begin{align*}
a_i &= \delta  &&  (0 \leq i \leq d),
\\
x_i &=
 \frac{ (q^i-q^{-i}) (q^{d-i+1}-q^{i-d-1}) \mu^2 (\mu^* q^i - h^* q^{-i})  (\mu^* q^{i-d-1} - h^* q^{d-i+1})}
        { ( \mu^* q^{2i-d} - h^* q^{d-2i} ) (\mu^* q^{2i-d-2} - h^* q^{d-2i+2}) }
             && (1 \leq i \leq d-1),
\\
x_d &=\frac{ (q^d-q^{-d})(q - q^{-1})\mu^2 (\mu^* q^{-1} - h^* q) }
                { \mu^* q^{d-2} - h^* q^{2-d} }.
\end{align*}
\end{lemma}

\begin{proof}
By Lemmas \ref{lem:aixiparam} and \ref{lem:essbiptype1b}.
\end{proof}

Next consider the case of type II.

\begin{lemma}    {\rm (See  \cite[Theorem 7.11]{NT:balanced}.)}
\label{lem:essbiptype2}    \samepage
\ifDRAFT {\rm lem:essbiptype2}. \fi
Assume that $A,A^*$ has type II.
Let
\[
(\delta, \mu, h, \delta^*, \mu^*, h^*, \tau)
\]
denote the primary data corresponding to the parameter array \eqref{eq:parray23b}.
Then the following are equivalent:
\begin{itemize}
\item[\rm (i)]
$A,A^*$ is essentially bipartite;
\item[\rm (ii)]
$h=0$ and $\tau=0$.
\end{itemize}
\end{lemma}

\begin{lemma}    \label{lem:essbiptype2b}    \samepage
\ifDRAFT {\rm lem:essbiptype2b}. \fi
Referring to Lemma \ref{lem:essbiptype2},
$A,A^*$ is essentially bipartite if and only if
\begin{align*}
\th_i &= \delta + \mu (i - d/2)   &&  (0 \leq i \leq d),
\\
\th^*_i &= \delta^* + \mu^* (i - d/2) + h^* i (d-i)  &&  (0 \leq i \leq d),
\\
\vphi_i &= - i (d-i+1) \mu \big( \mu^*/2 - h^* ( i - (d+1)/2 ) \big)        && (1 \leq i \leq d),
\\
\phi_i &=   i (d-i+1) \mu \big( \mu^*/2 - h^* ( i - (d+1)/2 ) \big)       && (1 \leq i \leq d).
\end{align*}
\end{lemma}

\begin{proof}
By Lemmas \ref{lem:type2param} and \ref{lem:essbiptype2}.
\end{proof}

\begin{lemma}    \label{lem:essbiptype2bb}     \samepage
\ifDRAFT {\rm lem:essbiptype2bb}. \fi
Referring to Lemma \ref{lem:essbiptype2},
assume that $A,A^*$ is essentially bipartite.
Then
\begin{align*}
& \text{\rm $\text{\rm Char}(\F)$ is equal to $0$ or greater than $d$},
\\ 
& \mu \neq 0,
\\
& \mu^* \neq h^*  i \qquad \qquad  (1-d \leq i \leq d-1).
\end{align*}
\end{lemma}

\begin{proof}
Evaluate the inequalities in Lemma \ref{lem:type2cond} using
Lemma \ref{lem:essbiptype2}(ii).
\end{proof}

\begin{lemma}    \label{lem:essbiptype2c}    \samepage
\ifDRAFT {\rm lem:essbiptype2c}. \fi
Referring to Lemma \ref{lem:essbiptype2},
$A,A^*$ is essentially bipartite if and only if
\begin{align*}
a_i &= \delta  &&  (0 \leq i \leq d),
\\
x_i &=  \frac{ i (d-i+1) \mu^2 (\mu^* - i h^*) \big( \mu^* + (d-i+1) h^* \big) }
          { 4 \big( \mu^* + (d-2i) h^* \big) \big( \mu^* + (d-2i+2) h^* \big) }
&& (1 \leq i \leq d-1),
\\
x_d &= \frac{ d \mu^2 (\mu^* + h^*) }
                {4 (\mu^* + (2-d) h^*) }.
\end{align*}
\end{lemma}

\begin{proof}
By Lemmas \ref{lem:aixiparam} and \ref{lem:essbiptype2b}.
\end{proof}

Next consider the case of type III$^+$.

\begin{lemma}    {\rm (See \cite[Theorem 8.11]{NT:balanced}.) }
\label{lem:essbiptype3+}    \samepage
\ifDRAFT {\rm lem:essbiptype3+}. \fi
Assume that $A,A^*$ has type III$^+$,
and let
\[
 (\delta, s, h, \delta^*, s^*, h^*, \tau)
\]
denote the primary data corresponding to the parameter array \eqref{eq:parray23b}.
Then the following are equivalent:
\begin{itemize}
\item[\rm (i)]
$A,A^*$ is essentially bipartite;
\item[\rm (ii)]
$s=0$ and $\tau = 0$.
\end{itemize}
\end{lemma}

\begin{lemma}    \label{lem:essbiptype3+b}    \samepage
\ifDRAFT {\rm lem:essbiptype3+b}. \fi
Referring to Lemma \ref{lem:essbiptype3+},
$A,A^*$ is essentially bipartite if and only if
\begin{align*}
\theta_i &=
  \begin{cases}
     \delta + h (i-d/2)  & \text{\rm if $i$ is even}, \\
     \delta - h (i-d/2) & \text{\rm if $i$ is odd}
  \end{cases}        
  &&  (0 \leq i \leq d),              
\\
\theta^*_i &=
   \begin{cases}
     \delta^* +s^* +h^*(i-d/2)  &   \text{\rm if $i$ is even}, \\
     \delta^* - s^* -h^*(i-d/2) &  \text{\rm if $i$ is odd}
   \end{cases}  
   &&  (0 \leq i \leq d),    
\\
\varphi_i &=
   \begin{cases}
      - i h \big( s^* +  h^* ( i- (d+1)/2 ) \big)        &   \text{\rm if $i$ is even}, \\
      (d-i+1) h \big( s^* +  h^* ( i- (d+1)/2 ) \big)   & \text{\rm if $i$ is odd} 
   \end{cases}      
     &&  (1 \leq i \leq d),
\\
\phi_i &=
    \begin{cases}
       i h \big( s^* +  h^* ( i- (d+1)/2 ) \big)        &   \text{\rm if $i$ is even}, \\
      - (d-i+1) h \big( s^* +  h^* ( i- (d+1)/2 ) \big)   & \text{\rm if $i$ is odd}
   \end{cases}      
        &&  (1 \leq i \leq d).
\end{align*}
\end{lemma}

\begin{proof}
By Lemmas \ref{lem:type3+param} and \ref{lem:essbiptype3+}.
\end{proof}

\begin{lemma}    \label{lem:essbiptype3+bb}     \samepage
\ifDRAFT {\rm lem:essbiptype3+bb}. \fi
Referring to Lemma \ref{lem:essbiptype3+},
assume that $A,A^*$ is essentially bipartite.
Then
\begin{align*}
& \text{\rm $\text{\rm Char}(\F)$ is equal to $0$ or greater than $d$},
\\ 
& h \neq 0, \qquad h^* \neq 0,
\\
& 2 s^* \neq h^* i \qquad\text{\rm if $i$ is odd} \qquad \qquad  (1-d \leq i \leq d-1).
\end{align*}
\end{lemma}

\begin{proof}
Evaluate the inequalities in Lemma \ref{lem:type3+cond} using
Lemma \ref{lem:essbiptype3+}(ii).
\end{proof}

\begin{lemma}    \label{lem:essbiptype3+c}    \samepage
\ifDRAFT {\rm lem:essbiptype3+c}. \fi
Referring to Lemma \ref{lem:essbiptype3+},
$A,A^*$ is essentially bipartite if and only if
\begin{align*}
a_i &= \delta  &&  (0 \leq i \leq d),
\\
x_i &= 
\begin{cases}
   i h^2 \big( (d-i+1) h^* - 2 s^* \big)  / ( 4 h^* )  &  \text{\rm if $i$ is even},   \\
   (d-i+1) h^2 ( i h^* + 2 s^*)  / ( 4 h^* )        &  \text{\rm if $i$ is odd}
\end{cases}
&& (1 \leq i \leq d).
\end{align*}
\end{lemma}

\begin{proof}
By Lemmas \ref{lem:aixiparam} and \ref{lem:essbiptype3+b}.
\end{proof}

\begin{lemma}   {\rm (See \cite[Theorem 9.6]{NT:balanced}.) }
 \label{lem:essbiptype3-}    \samepage
\ifDRAFT {\rm lem:essbiptype3-}. \fi
A Leonard pair of type III$^-$ is not essentially bipartite.
\end{lemma}

\section{Some characterizations of the bipartite Leonard pairs}
\label{sec:bipbip}
\ifDRAFT {\rm sec:bipbip}. \fi

In this section, we characterize the bipartite Leonard pairs
in terms of the primary data,
the parameter array,
and the TD/D sequence.

Throughout this section, let $A,A^*$ denote a Leonard pair over $\F$,
with parameter array
\begin{equation}
  (\{\th_i\}_{i=0}^d; \{\th^*_i\}_{i=0}^d; \{\vphi_i\}_{i=1}^d; \{\phi_i\}_{i=1}^d)  \label{eq:parray23c}
\end{equation}
and corresponding TD/D sequence
\[
 (\{a_i\}_{i=0}^d; \{x_i\}_{i=1}^d; \{\th^*_i\}_{i=0}^d).
\]

We first consider the case of type I.

\begin{lemma}       \label{lem:biptype1}    \samepage
\ifDRAFT {\rm lem:biptype1}. \fi
Fix a nonzero $q \in \F$, and assume that $A,A^*$ has type I 
with fundamental constant $\beta = q^2 + q^{-2}$.
Let
\[
 (\delta, \mu, h, \delta^*, \mu^*, h^*, \tau)
\]
denote the primary $q$-data corresponding to the parameter array
\eqref{eq:parray23c}.
Then the following are equivalent:
\begin{itemize}
\item[\rm (i)]
$A,A^*$ is bipartite;
\item[\rm (ii)]
$\delta = 0$ and $\mu + h = 0$ and $\tau=0$.
\end{itemize}
\end{lemma}

\begin{proof}
By Lemmas \ref{lem:essbiptype1} and \ref{lem:essbiptype1c}.
\end{proof}

\begin{lemma}    \label{lem:biptype1b}    \samepage
\ifDRAFT {\rm lem:biptype1b}. \fi
Referring to Lemma \ref{lem:biptype1},
$A,A^*$ is bipartite if and only if
\begin{align*}
\th_i &= \mu (q^{2i-d} - q^{d-2i})  &&  (0 \leq i \leq d),
\\
\th^*_i &= \delta^* + \mu^* q^{2i-d} + h^* q^{d-2i}   &&  (0 \leq i \leq d),
\\
\vphi_i &= - (q^i - q^{-i})(q^{d-i+1} - q^{i-d-1}) \mu (\mu^* q^{2i-d-1} - h^* q^{d-2i+1})  && (1 \leq i \leq d),
\\
\phi_i &= (q^i - q^{-i})(q^{d-i+1} - q^{i-d-1}) \mu (\mu^* q^{2i-d-1} - h^* q^{d-2i+1})  && (1 \leq i \leq d).
\end{align*}
\end{lemma}

\begin{proof}
By Lemmas \ref{lem:essbiptype1b} and \ref{lem:biptype1}.
\end{proof}

\begin{lemma}    \label{lem:biptype1bb}     \samepage
\ifDRAFT {\rm lem:biptype1bb}. \fi
Referring to Lemma \ref{lem:biptype1},
assume that $A,A^*$ is bipartite.
Then
\begin{align*}
 & q^{2i} \neq 1 \qquad\qquad (1 \leq i \leq d),
\\
 & q^{2i} \neq -1 \qquad\qquad (1 \leq i \leq d-1),
\\
& \mu \neq 0,
\\
& \mu^* q^i \neq h^* q^{-i}   \qquad\qquad( 1-d \leq i \leq d-1).
\end{align*}
\end{lemma}

\begin{proof}
Evaluate the inequalities in Lemma \ref{lem:type1cond} using
Lemma \ref{lem:biptype1}(ii).
\end{proof}

\begin{lemma}    \label{lem:biptype1c}    \samepage
\ifDRAFT {\rm lem:biptype1c}. \fi
Referring to Lemma \ref{lem:biptype1},
$A,A^*$ is bipartite if and only if
\begin{align*}
a_i &= 0  &&  (0 \leq i \leq d),
\\
x_i &=
 \frac{ (q^i-q^{-i}) (q^{d-i+1}-q^{i-d-1})\mu^2  (\mu^* q^i - h^* q^{-i})  (\mu^* q^{i-d-1} - h^* q^{d-i+1})}
        { ( \mu^* q^{2i-d} - h^* q^{d-2i} ) (\mu^* q^{2i-d-2} - h^* q^{d-2i+2}) }
             && (1 \leq i \leq d-1),
\\
x_d &=\frac{ (q^d-q^{-d})(q - q^{-1})\mu^2 (\mu^* q^{-1} - h^* q) }
                { \mu^* q^{d-2} - h^* q^{2-d} }.
\end{align*}
\end{lemma}

\begin{proof}
By Lemma \ref{lem:essbiptype1c}.
\end{proof}

Next consider the case of type II.

\begin{lemma}    \label{lem:biptype2}    \samepage
\ifDRAFT {\rm lem:biptype2}. \fi
Assume that $A,A^*$ has type II,
and let
\[
(\delta, \mu, h, \delta^*, \mu^*, h^*, \tau)
\]
denote the primary data corresponding to the parameter array \eqref{eq:parray23c}.
Then the following are equivalent:
\begin{itemize}
\item[\rm (i)]
$A,A^*$ is bipartite;
\item[\rm (ii)]
$\delta = 0$ and $h = 0$  and $\tau=0$.
\end{itemize}
\end{lemma}

\begin{proof}
By Lemmas \ref{lem:essbiptype2} and \ref{lem:essbiptype2c}.
\end{proof}

\begin{lemma}    \label{lem:biptype2b}    \samepage
\ifDRAFT {\rm lem:biptype2b}. \fi
Referring to Lemma \ref{lem:biptype2},
$A,A^*$ is bipartite if and only if
\begin{align*}
\th_i &= \mu (i - d/2)   &&  (0 \leq i \leq d),
\\
\th^*_i &= \delta^* + \mu^* ( i - d/2) + h^* i (d-i)   &&  (0 \leq i \leq d),
\\
\vphi_i &= - i (d-i+1) \mu \big(  \mu^* /2 - h^* (i - (d+1)/2) \big)   && ( 1 \leq i \leq d),
\\
\phi_i &= i (d-i+1)  \mu \big(  \mu^* /2 - h^* (i - (d+1)/2) \big)   && ( 1 \leq i \leq d).
\end{align*}
\end{lemma}

\begin{proof}
By Lemmas \ref{lem:essbiptype2b} and \ref{lem:biptype2}.
\end{proof}

\begin{lemma}    \label{lem:biptype2bb}     \samepage
\ifDRAFT {\rm lem:biptype2bb}. \fi
Referring to Lemma \ref{lem:biptype2},
assume that $A,A^*$ is bipartite.
Then
\begin{align*}
& \text{\rm $\text{\rm Char}(\F)$ is equal to $0$ or greater than $d$},
\\ 
& \mu \neq 0,
\\
& \mu^* \neq h^*  i \qquad \qquad  (1-d \leq i \leq d-1).
\end{align*}
\end{lemma}

\begin{proof}
Evaluate the inequalities in Lemma \ref{lem:type2cond} using
Lemma \ref{lem:biptype2}(ii).
\end{proof}

\begin{lemma}    \label{lem:biptype2c}    \samepage
\ifDRAFT {\rm lem:biptype2c}. \fi
Referring to Lemma \ref{lem:biptype2},
$A,A^*$ is bipartite if and only if
\begin{align*}
a_i &= 0  &&  (0 \leq i \leq d),
\\
x_i &=  \frac{ i (d-i+1) \mu^2 (\mu^* - i h^*) \big( \mu^* + (d-i+1) h^* \big) }
          { 4 \big( \mu^* + (d-2i) h^* \big) \big( \mu^* + (d-2i+2) h^* \big) }
&& (1 \leq i \leq d-1),
\\
x_d &= \frac{ d \mu^2 (\mu^* + h^*) }
                {4 (\mu^* + (2-d) h^*) }.
\end{align*}
\end{lemma}

\begin{proof}
By Lemma \ref{lem:essbiptype2c}.
\end{proof}

Next consider the case of type III$^+$.

\begin{lemma}    \label{lem:biptype3+}    \samepage
\ifDRAFT {\rm lem:biptype3+}. \fi
Assume that $A,A^*$ has type III$^+$,
and let
\[
 (\delta, s, h, \delta^*, s^*, h^*, \tau)
\]
denote the primary data corresponding to the parameter array
\eqref{eq:parray23c}.
Then the following are equivalent:
\begin{itemize}
\item[\rm (i)]
$A,A^*$ is bipartite;
\item[\rm (ii)]
$\delta = 0$ and $s = 0$  and $\tau=0$.
\end{itemize}
\end{lemma}

\begin{proof}
By Lemmas \ref{lem:essbiptype3+} and \ref{lem:essbiptype3+c}.
\end{proof}

\begin{lemma}    \label{lem:biptype3+b}    \samepage
\ifDRAFT {\rm lem:biptype3+b}. \fi
Referring to Lemma \ref{lem:biptype3+},
$A,A^*$ is bipartite if and only if
\begin{align*}
\th_i &=
 \begin{cases}
  h (i - d/2)   &  \text{\rm if $i$ is even},  \\
  - h (i-d/2)  &   \text{\rm if $i$ is odd}
 \end{cases}
  &&  (0 \leq i \leq d),
\\
\th^*_i &=
 \begin{cases}
  \delta^* + s^* + h^* (i-d/2)  &  \text{\rm if $i$ is even},  \\
  \delta^* - s^* - h^* (i-d/2)  &  \text{\rm if $i$ is odd}
 \end{cases} 
   &&   (0 \leq i \leq d),
\\
\vphi_i &=
 \begin{cases}
  - i h \big( s^* + h^* (i - (d+1)/2) \big)   &  \text{\rm if $i$ is even},  \\
 (d-i+1) h \big( s^* + h^* (i-(d+1)/2) \big)  &  \text{\rm if $i$ is odd}
 \end{cases}
   &&  ( 1 \leq i \leq d),
\\
 \phi_i &=
 \begin{cases}
   i h \big( s^* + h^* (i - (d+1)/2) \big)   &  \text{\rm if $i$ is even},  \\
   - (d-i+1) h \big( s^* + h^* (i-(d+1)/2) \big)  &  \text{\rm if $i$ is odd}
 \end{cases}
   &&  ( 1 \leq i \leq d).
\end{align*}
\end{lemma}

\begin{proof}
By Lemmas \ref{lem:essbiptype3+b} and \ref{lem:biptype3+}.
\end{proof}

\begin{lemma}    \label{lem:biptype3+bb}     \samepage
\ifDRAFT {\rm lem:biptype3+bb}. \fi
Referring to Lemma \ref{lem:biptype3+},
assume that $A,A^*$ is bipartite.
Then
\begin{align*}
& \text{\rm $\text{\rm Char}(\F)$ is equal to $0$ or greater than $d$},
\\ 
& h \neq 0, \qquad h^* \neq 0,
\\
& 2 s^* \neq h^* i \qquad\text{\rm if $i$ is odd} \qquad \qquad  (1-d \leq i \leq d-1).
\end{align*}
\end{lemma}

\begin{proof}
Evaluate the inequalities in Lemma \ref{lem:type3+cond} using
Lemma \ref{lem:biptype3+}(ii).
\end{proof}

\begin{lemma}    \label{lem:biptype3+c}    \samepage
\ifDRAFT {\rm lem:biptype3+c}. \fi
Referring to Lemma \ref{lem:biptype3+},
$A,A^*$ is bipartite if and only if
\begin{align*}
a_i &= 0  &&  (0 \leq i \leq d),
\\
x_i &= 
\begin{cases}
   i h^2 \big( (d-i+1) h^* - 2 s^* \big)  / ( 4 h^* )  &  \text{\rm if $i$ is even},   \\
   (d-i+1) h^2 ( i h^* + 2 s^*)  / ( 4 h^* )        &  \text{\rm if $i$ is odd}
\end{cases}
&& (1 \leq i \leq d).
\end{align*}
\end{lemma}

\begin{proof}
By Lemma \ref{lem:essbiptype3+c}.
\end{proof}

\begin{lemma}  \label{lem:biptype3-}    \samepage
\ifDRAFT {\rm lem:biptype3-}. \fi
A Leonard pair of type III$^-$ is not bipartite.
\end{lemma}

\begin{proof}
By Lemma \ref{lem:essbiptype3-}.
\end{proof}

\section{Leonard pairs of dual $q$-Krawtchouk type}
\label{sec:dualqKrawt}
\ifDRAFT {\rm sec:dualqKrawt}. \fi

In this section, we describe a family of Leonard pairs, 
said to have dual $q$-Krawtchouk type.

Throughout this section,
let $A,A^*$ denote a Leonard pair over $\F$ with parameter array 
\begin{equation}
  (\{\th_i\}_{i=0}^d; \{\th^*_i\}_{i=0}^d; \{\vphi_i\}_{i=1}^d; \{\phi_i\}_{i=1}^d).  \label{eq:parray11}
\end{equation}

\begin{defi}        \label{def:LPdualqKrawt}    \samepage
\ifDRAFT {\rm def:LPdualqKrawt}. \fi
The Leonard pair $A,A^*$ is said to have {\em dual $q$-Krawtchouk type}
whenever the following {\rm (i)--(iii)} hold:
\begin{itemize}
\item[\rm (i)]
$A,A^*$ has type I;
\item[\rm (ii)]
the expression $(\th^*_{i-1} - \th^*_{i})/(\th^*_i - \th^*_{i+1})$ is
independent of $i$ for $1 \leq i \leq d-1$;
\item[\rm (iii)]
the expression $\vphi_i / \phi_i$ is independent of $i$ for $1 \leq i \leq d$.
\end{itemize}
\end{defi}

\begin{note}     \label{note:dualqKrawt}    \samepage
\ifDRAFT {\rm note:dualqKrawt}. \fi
The Leonard pairs of dual $q$-Krawtchouk type are attached to the
dual $q$-Krawtchouk polynomials \cite[Example 35.8]{T:survey}.
\end{note}

For the rest of this section, assume that $A,A^*$ has type I,
and let
\begin{equation}
 (\delta, \mu, h, \delta^*, \mu^*, h^*, \tau)                 \label{eq:basicseq}
\end{equation}
denote the primary $q$-data corresponding to the parameter array
\eqref{eq:parray11}.

\begin{lemma}   \label{lem:qKrawtpre}    \samepage
\ifDRAFT {\rm lem:qKrawtpre}. \fi
We have
\begin{align}
& \frac{\th^*_{i-1} - \th^*_i } {\th^*_i - \th^*_{i+1}}
 = \frac{\mu^* q^{4i} - h^* q^{2d+2} } { \mu^* q^{4i+2} - h^* q^{2d} }
 &&  ( 1 \leq i \leq d-1),                                 \label{eq:qKrawtratio}
\\
& \frac{\vphi_i}{\phi_i} =
 \frac{\tau  - \mu \mu^* q^{2i-d-1} - h h^* q^{d-2i+1} }
        { \tau - h \mu^* q^{2i-d-1} - \mu h^* q^{d-2i+1}}
  &&  ( 1 \leq i \leq d).                                    \label{eq:qKrawtratio2}
\end{align}
\end{lemma}

\begin{proof}
Use \eqref{eq:type1ths}--\eqref{eq:type1phi}.
\end{proof}

Note by \eqref{eq:type1paramcond3} that $\mu^*$ and $h^*$ are not both zero.

\begin{lemma}   \label{lem:defLPdualqKrawt}    \samepage
\ifDRAFT {\rm lem:defLPdualqKrawt}. \fi
The following are equivalent:
\begin{itemize}
\item[\rm (i)]
$A,A^*$ has dual $q$-Krawtchouk type;
\item[\rm (ii)]
$\mu^* h^* = 0$ and $\tau = 0$.
\end{itemize}
Suppose the equivalent conditions {\rm (i)} and {\rm (ii)} hold.
Then
\begin{align}
(\th^*_{i-1} - \th^*_{i})/(\th^*_i - \th^*_{i+1}) &=
\begin{cases}
    q^{2} & \text{\rm if $\mu^*=0$},
 \\
    q^{-2} & \text{\rm if $h^*=0$}
\end{cases}
  &&  (1 \leq i \leq d-1),                      \label{eq:qKrawtcond1}
\\
\vphi_i / \phi_i &= 
 \begin{cases}
  h / \mu  & \text{\rm if $\mu^* = 0$},
 \\
 \mu / h
     & \text{\rm if $h^*=0$}
 \end{cases}
&&    (1 \leq i \leq d).                        \label{eq:qKrawtcond2}
\end{align}
\end{lemma}

\begin{proof}
(i) $\Rightarrow$ (ii)
Using Definition \ref{def:LPdualqKrawt}(ii), 
compare the values of \eqref{eq:qKrawtratio} for $i=1$ and $i=2$
to find that
\[
  \mu^* h^* q^{2d+4}(q^4-1)^2 = 0.
\]
By this and \eqref{eq:type1paramcond1} we get  $\mu^* h^* = 0$.
First assume that $\mu^* = 0$.
By \eqref{eq:qKrawtratio2},
\begin{align}
 \frac{ \vphi_i } { \phi_i } &=
  \frac{ \tau - h h^* q^{d-2i+1} } { \tau - \mu h^* q^{d-2i+1} }  &&  (1 \leq i \leq d).
                         \label{eq:qKrawtratio3}
\end{align}
Using Definition \ref{def:LPdualqKrawt}(iii), compare the values of
\eqref{eq:qKrawtratio3} for $i=1$ and $i=2$ to find that
\[
   \tau h^* (\mu - h) (q^2 -1) = 0.
\]
By this and \eqref{eq:type1paramcond1}, \eqref{eq:type1paramcond2}, 
we get  $\tau=0$.
Next assume that $h^*=0$.
We get $\tau=0$ in a similar way.

(ii) $\Rightarrow$ (i)
By \eqref{eq:qKrawtratio} and \eqref{eq:qKrawtratio2} we find that 
\eqref{eq:qKrawtcond1} and \eqref{eq:qKrawtcond2} hold.
Thus in Definition \ref{def:LPdualqKrawt} the conditions (ii) and (iii) hold.
\end{proof}

\begin{lemma}    \label{lem:qKrawtparam}    \samepage
\ifDRAFT {\rm lem:qKrawtparam}. \fi
The following hold:
\begin{itemize}
\item[\rm (i)]
$A,A^*$ has dual $q$-Krawtchouk type with $\mu^* = 0$ if and only if
\begin{align*}
\th_i &= \delta + \mu q^{2i-d} + h q^{d-2i}    &&  (0 \leq i \leq d),
\\
\th^*_i &= \delta^* + h^* q^{d-2i}                &&  ( 0 \leq i \leq d),
\\
\vphi_i &= - h h^* q^{d-2i+1} (q^i - q^{-i})(q^{d-i+1} - q^{i-d-1})   &&  (1 \leq i \leq d),
\\
\phi_i &= - \mu h^*  q^{d-2i+1} (q^i - q^{-i})(q^{d-i+1} - q^{i-d-1})   &&  (1 \leq i \leq d);
\end{align*}
\item[\rm (ii)]
$A,A^*$ has dual $q$-Krawtchouk type with $h^* = 0$ if and only if
\begin{align*}
\th_i &= \delta + \mu q^{2i-d} + h q^{d-2i}    &&  (0 \leq i \leq d),
\\
\th^*_i &= \delta^* + \mu^* q^{2i-d}                &&  ( 0 \leq i \leq d),
\\
\vphi_i &= - \mu \mu^* q^{2i-d-1} (q^i - q^{-i})(q^{d-i+1} - q^{i-d-1})   &&  (1 \leq i \leq d),
\\
\phi_i &= - h \mu^*  q^{2i-d-1} (q^i - q^{-i})(q^{d-i+1} - q^{i-d-1})   &&  (1 \leq i \leq d).
\end{align*}
\end{itemize}
\end{lemma}

\begin{proof}
By Lemmas \ref{lem:type1param} and \ref{lem:defLPdualqKrawt}.
\end{proof}

\begin{lemma}    \label{lem:qKrawtparamb}    \samepage
\ifDRAFT {\rm lem:qKrawtparamb}. \fi
Assume that $A,A^*$ has dual $q$-Krawtchouk type.
Then
\begin{align*}
 &  q^{2i} \neq 1 \qquad (1 \leq i \leq d),
\\
 & \mu \neq h q^{2i}  \qquad (1-d \leq i \leq d-1),
\\
 &\mu \neq 0, \qquad h\neq 0,
\\
 & \text{\rm $\mu^*$, $h^*$ not both zero}.
\end{align*}
\end{lemma}   

\begin{proof}
Evaluate the inequalities in Lemma \ref{lem:type1cond} using
Lemma \ref{lem:defLPdualqKrawt}(ii).
\end{proof}

Let 
$(\{a_i\}_{i=0}^d; \{x_i\}_{i=1}^d; \{\th^*_i\}_{i=0}^d)$
denote the TD/D sequence of $A,A^*$ corresponding to
the parameter array \eqref{eq:parray11}.

\begin{lemma}   \label{lem:qKrawtaixi}    \samepage
\ifDRAFT {\rm lem:qKrawtaixi}. \fi
The following hold:
\begin{itemize}
\item[\rm (i)]
$A,A^*$ has dual $q$-Krawtchouk type with $\mu^* = 0$ if and only if
\begin{align}
a_i &= \delta + (\mu+h) q^{2i-d}  &&   (0 \leq i \leq d),     \label{eq:qKrawtai}
\\
x_i &= - \mu h q^{2i-d-1}  (q^i - q^{-i})(q^{d-i+1} - q^{i-d-1})   &&  (1 \leq i \leq d),
                                                                              \label{eq:qKrawtxi}
\\
\th^*_i &= \delta^* + h^* q^{d-2i}                &&  ( 0 \leq i \leq d);   \label{eq:qKrawtthsi}
\end{align}
\item[\rm (ii)]
$A,A^*$ has dual $q$-Krawtchouk type with $h^* = 0$ if and only if
\begin{align*}
a_i &= \delta + (\mu+h) q^{d-2i}  &&   (0 \leq i \leq d),
\\
x_i &= - \mu h q^{d-2i+1}  (q^i - q^{-i})(q^{d-i+1} - q^{i-d-1})   &&  (1 \leq i \leq d),
\\
\th^*_i &= \delta^* + \mu^* q^{2i-d}                &&  ( 0 \leq i \leq d).
\end{align*}
\end{itemize}
\end{lemma}

\begin{proof}
(i)
First assume that $A,A^*$ has dual $q$-Krawtchouk type with $\mu^* = 0$.
Using Lemmas \ref{lem:aixiparam} and \ref{lem:qKrawtparam} we get
\eqref{eq:qKrawtai}--\eqref{eq:qKrawtthsi}.
Next assume that \eqref{eq:qKrawtai}--\eqref{eq:qKrawtthsi} hold.
Comparing \eqref{eq:qKrawtthsi} with \eqref{eq:type1ths} we find that $\mu^* = 0$.
Now by \eqref{eq:a0param} and \eqref{eq:type1th}--\eqref{eq:type1vphi},
\[
  a_0 = \delta + (\mu+h) q^{-d} + \frac{\tau}{h^*} q(1-q^{-2d}).
\]
Comparing this with \eqref{eq:qKrawtai}, we find that $\tau q (1-q^{-2d}) =0$.
By this and \eqref{eq:type1paramcond1}, we get $\tau=0$.
So by Lemma \ref{lem:defLPdualqKrawt} we find that $A,A^*$ has dual $q$-Krawtchouk type.

(ii)
Similar.
\end{proof}

\begin{defi}    \label{def:parraydualqKrawt}     \samepage
\ifDRAFT {\rm def:parraydualqKrawt}. \fi
A parameter array over $\F$ is said to have
{\em dual $q$-Krawtchouk type}
whenever the corresponding Leonard pair has dual $q$-Krawtchouk type.
\end{defi}

\section{Leonard pairs of Krawtchouk type}
\label{sec:Krawt}
\ifDRAFT {\rm sec:Krawt}. \fi

In this section, we describe a family of Leonard pairs, 
said to have Krawtchouk type.

Throughout this section,
let $A,A^*$ denote a Leonard pair over $\F$ with parameter array 
\begin{equation}
  (\{\th_i\}_{i=0}^d; \{\th^*_i\}_{i=0}^d; \{\vphi_i\}_{i=1}^d; \{\phi_i\}_{i=1}^d).   \label{eq:parray12}
\end{equation}

\begin{defi}        \label{def:Krawt}    \samepage
\ifDRAFT {\rm def:Krawt}. \fi
The Leonard pair $A,A^*$ is said to have {\em Krawtchouk type}
whenever the following {\rm (i)--(iii)} hold:
\begin{itemize}
\item[\rm (i)]
$A,A^*$ has type II;
\item[\rm (ii)]
$\th_i - \th_{i-1}$ is independent of $i$
for $1 \leq i \leq d$;
\item[\rm (iii)]
$\th^*_{i} - \th^*_{i-1}$ is independent of $i$
for $1 \leq i \leq d$.
\end{itemize}
\end{defi}

\begin{note}     \label{note:Krawt}    \samepage
\ifDRAFT {\rm note:Krawt}. \fi
The Leonard pairs of Krawtchouk type are attached to the
Krawtchouk polynomials \cite[Example 35.12]{T:survey}.
\end{note}

For the rest of this section,
assume that $A,A^*$ has type II, and let
\begin{equation}
 (\delta, \mu, h, \delta^*, \mu^*, h^*, \tau)      \label{eq:basicseq2}
\end{equation}
denote the primary data corresponding to the parameter array  \eqref{eq:parray12}.

\begin{lemma}    \label{lem:Krawtpre}    \samepage
\ifDRAFT {\rm lem:Krawtpre}. \fi
For $1 \leq i \leq d$ we have
\begin{align}
\th_{i}-\th_{i-1} &= \mu + (d+1)h - 2 h i,          \label{eq:Krawtdif}
\\
\th^*_{i} - \th^*_{i-1} &=  \mu^* + (d+1) h^* - 2 h^* i.    \label{eq:Krawtdif2}
\end{align}
\end{lemma}

\begin{proof}
Use \eqref{eq:type2th} and \eqref{eq:type2ths}.
\end{proof}

\begin{lemma}    \label{lem:defKrawt}    \samepage
\ifDRAFT {\rm lem:defKrawt}. \fi
The following are equivalent:
\begin{itemize}
\item[\rm (i)]
$A,A^*$ has Krawtchouk type;
\item[\rm (ii)]
$h=0$ and $h^* = 0$.
\end{itemize}
Suppose the equivalent conditions {\rm (i)} and {\rm (ii)} hold.
Then
\begin{align*}
\th_{i} - \th_{i-1} &=  \mu,
&
\th^*_{i} - \th^*_{i-1} &=  \mu^*   &&   (1 \leq i \leq d).
\end{align*}
\end{lemma}

\begin{proof}
Use \eqref{eq:Krawtdif} and \eqref{eq:Krawtdif2}.
\end{proof}

\begin{note}    \label{note:defdualqKrawt}    \samepage
\ifDRAFT {\rm note:defdualqKrawt}. \fi
Definition \ref{def:Krawt} is not the type II version of Definition \ref{def:LPdualqKrawt}.
For the case of type II, the conditions (ii) and (iii) in Definition \ref{def:LPdualqKrawt}
hold if and only if  $\mu h = 0$ and $h^*=0$.
\end{note}

\begin{lemma}    \label{lem:Krawtparam}    \samepage
\ifDRAFT {\rm lem:Krawtparam}. \fi
The Leonard pair $A,A^*$ has Krawtchouk type
if and only if 
\begin{align*}
\th_i &= \delta + \mu (i - d/2)  &&  (0 \leq i \leq d),
\\
\th^*_i &= \delta^* + \mu^* (i - d/2)         &&  ( 0 \leq i \leq d),
\\
\vphi_i &=  i (d-i+1) (\tau - \mu \mu^* /2)     &&  (1 \leq i \leq d),
\\
\phi_i &=   i (d-i+1) (\tau + \mu \mu^* /2)  &&  (1 \leq i \leq d).
\end{align*}
\end{lemma}

\begin{proof}
By Lemmas \ref{lem:type2param} and \ref{lem:defKrawt}.
\end{proof}

\begin{lemma}    \label{lem:Krawtparamb}    \samepage
\ifDRAFT {\rm lem:Krawtparamb}. \fi
Assume that the Leonard pair $A,A^*$ has Krawtchouk type.
Then
\begin{align*}
& \text{\rm $\text{\rm Char}(\F)$ is equal to $0$ or greater than $d$},
\\ 
& \mu \neq 0,  \qquad\qquad \mu^* \neq 0,
\\
& 2 \tau \neq \mu \mu^*,  \qquad 2 \tau \neq - \mu \mu^*.
\end{align*}
\end{lemma}

\begin{proof}
Evaluate the inequalities in Lemma \ref{lem:type2cond}
using Lemma \ref{lem:defKrawt}(ii).
\end{proof}

Let 
$(\{a_i\}_{i=0}^d; \{x_i\}_{i=1}^d; \{\th^*_i\}_{i=0}^d)$
denote the TD/D sequence of $A,A^*$ corresponding to the parameter array
\eqref{eq:parray12}.

\begin{lemma}   \label{lem:Krawtaixi}    \samepage
\ifDRAFT {\rm lem:Krawtaixi}. \fi
The Leonard pair  $A,A^*$ has Krawtchouk type if and only if
\begin{align}
a_i &= \delta + (2i-d) \tau / \mu^*       &&  (0 \leq i \leq d),    \label{eq:Krawtai}
\\
x_i &= i (d-i+1)
          \left( \frac{ \mu^2} { 4 } - \frac{\tau^2}{\mu^{*2} }   \right)
                    && (1 \leq i \leq d),                                              \label{eq:Krawtxi}
\\
\th^*_i &= \delta^* + \mu^* (i - d/2)         &&  ( 0 \leq i \leq d).   \label{eq:Krawtthsi}
\end{align}
\end{lemma}

\begin{proof}
First assume that $A,A^*$ has Krawtchouk type.
By Lemma \ref{lem:defKrawt} we have $h=0$ and $h^*=0$.
Now using Lemmas \ref{lem:aixiparam} and \ref{lem:Krawtparam} we get
\eqref{eq:Krawtai}--\eqref{eq:Krawtthsi}.
Next assume that \eqref{eq:Krawtai}--\eqref{eq:Krawtthsi} hold.
Comparing \eqref{eq:type2ths} with \eqref{eq:Krawtthsi}  we get $h^* = 0$.
By \eqref{eq:aiparam} and \eqref{eq:type2th}--\eqref{eq:type2vphi} with $h^*=0$,
\[
  a_0 = \delta - \frac{d \tau}{\mu^*} + \frac{d(d-1)h } { 2 }.
\]
By \eqref{eq:Krawtai},
\[
  a_0 = \delta - \frac{ d \tau } { \mu^*}.
\]
Comparing the above two equations, we find that $h=0$.
Now $A,A^*$ has Krawtchouk type by Lemma \ref{lem:defKrawt}.
\end{proof}

\begin{defi}    \label{def:parrayKrawt}     \samepage
\ifDRAFT {\rm def:parrayKrawt}. \fi
A parameter array over $\F$ is said to have
{\em Krawtchouk type}
whenever the corresponding Leonard pair has Krawtchouk type.
\end{defi}

\section{Leonard pairs that are bipartite and have  dual $q$-Krawtchouk type}
\label{sec:bipdualqKrawt}
\ifDRAFT {\rm sec:bipdualqKrawt}. \fi

In Section \ref{sec:bipbip} we described the bipartite Leonard pairs.
In Section \ref{sec:dualqKrawt} we described the Leonard pairs of dual $q$-Krawtchouk type.
In this section we describe the Leonard pairs that are bipartite and have
dual $q$-Krawtchouk type.

Throughout this section, let $A,A^*$ denote a Leonard pair over $\F$
that has type I with parameter array 
\begin{equation}
  (\{\th_i\}_{i=0}^d; \{\th^*_i\}_{i=0}^d; \{\vphi_i\}_{i=1}^d; \{\phi_i\}_{i=1}^d).   \label{eq:parray11b}
\end{equation}
Denote the corresponding primary $q$-data by
\begin{equation}
 (\delta, \mu, h, \delta^*, \mu^*, h^*, \tau).                 \label{eq:basicseqb}
\end{equation}

\begin{lemma}    \label{lem:qKrawtparambip0}     \samepage
\ifDRAFT {\rm lem:qKrawtparambip0}. \fi
The Leonard pair $A,A^*$ is bipartite and has dual $q$-Krawtchouk type 
if and only if
\begin{align*}
\delta &= 0, &
\mu+h &= 0,  &
\mu^* h^* &= 0,  &
\tau &=0.
\end{align*}
\end{lemma}

\begin{proof}
By Lemmas \ref{lem:biptype1} and \ref{lem:defLPdualqKrawt}.
\end{proof}

\begin{lemma}    \label{lem:qKrawtparambip}    \samepage
\ifDRAFT {\rm lem:qKrawtparambip}. \fi
The following hold:
\begin{itemize}
\item[\rm (i)]
$A,A^*$ is bipartite and has dual $q$-Krawtchouk type with
$\mu^* = 0$ if and only if
\begin{align*}
\th_i &= \mu (q^{2i-d} - q^{d-2i})    &&  (0 \leq i \leq d),
\\
\th^*_i &= \delta^* + h^* q^{d-2i}                &&  ( 0 \leq i \leq d),
\\
\vphi_i &= \mu h^* q^{d-2i+1} (q^i - q^{-i})(q^{d-i+1} - q^{i-d-1})   &&  (1 \leq i \leq d),
\\
\phi_i &= - \mu h^*  q^{d-2i+1} (q^i - q^{-i})(q^{d-i+1} - q^{i-d-1})   &&  (1 \leq i \leq d);
\end{align*}
\item[\rm (ii)]
$A,A^*$ is bipartite and has dual $q$-Krawtchouk type with
$h^* = 0$ if and only if
\begin{align*}
\th_i &=  \mu (q^{2i-d} - q^{d-2i})    &&  (0 \leq i \leq d),
\\
\th^*_i &= \delta^* + \mu^* q^{2i-d}                &&  ( 0 \leq i \leq d),
\\
\vphi_i &= - \mu \mu^* q^{2i-d-1} (q^i - q^{-i})(q^{d-i+1} - q^{i-d-1})   &&  (1 \leq i \leq d),
\\
\phi_i &=  \mu \mu^*  q^{2i-d-1} (q^i - q^{-i})(q^{d-i+1} - q^{i-d-1})   &&  (1 \leq i \leq d).
\end{align*}
\end{itemize}
\end{lemma}

\begin{proof}
By Lemmas \ref{lem:biptype1} and \ref{lem:qKrawtparam}.
\end{proof}

\begin{lemma}    \label{lem:qKrawtbipparamb}    \samepage
\ifDRAFT {\rm lem:qKrawtbipparamb}. \fi
Assume that $A,A^*$ is bipartite and has dual $q$-Krawtchouk type.
Then
\begin{align*}
 &  q^{2i} \neq 1 \qquad (1 \leq i \leq d),
\\
 & q^{2i} \neq -1  \qquad (1 \leq i \leq d-1),
\\
 &\mu \neq 0,
\\
 & \text{\rm $\mu^*$, $h^*$ not both zero}.
\end{align*}
\end{lemma}   

\begin{proof}
By Lemmas \ref{lem:biptype1bb} and \ref{lem:qKrawtparambip0}.
\end{proof}

Let 
$(\{a_i\}_{i=0}^d; \{x_i\}_{i=1}^d; \{\th^*_i\}_{i=0}^d)$
denote the TD/D sequence of $A,A^*$ corresponding to
the parameter array \eqref{eq:parray11b}.

\begin{lemma}   \label{lem:qKrawtaixibip}    \samepage
\ifDRAFT {\rm lem:qKrawtaixibip}. \fi
The following hold:
\begin{itemize}
\item[\rm (i)]
$A,A^*$ is bipartite and has dual $q$-Krawtchouk type with $\mu^*=0$
if and only if
\begin{align*}
a_i &= 0  &&   (0 \leq i \leq d),
\\
x_i &=  \mu^2 q^{2i-d-1}  (q^i - q^{-i})(q^{d-i+1} - q^{i-d-1})   &&  (1 \leq i \leq d),
\\
\th^*_i &= \delta^* + h^* q^{d-2i}                &&  ( 0 \leq i \leq d);
\end{align*}
\item[\rm (ii)]
$A,A^*$ is bipartite and has dual $q$-Krawtchouk type with $h^*=0$
if and only if
\begin{align*}
a_i &= 0  &&   (0 \leq i \leq d),
\\
x_i &= \mu^2 q^{d-2i+1}  (q^i - q^{-i})(q^{d-i+1} - q^{i-d-1})   &&  (1 \leq i \leq d),
\\
\th^*_i &= \delta^* + \mu^* q^{2i-d}                &&  ( 0 \leq i \leq d).
\end{align*}
\end{itemize}
\end{lemma}

\begin{proof}
By Lemmas \ref{lem:qKrawtaixi} and \ref{lem:qKrawtparambip0}.
\end{proof}

\section{Leonard pairs that are bipartite and have Krawtchouk type}
\label{sec:bipKrawt}
\ifDRAFT {\rm sec:bipKrawt}. \fi

In Section \ref{sec:bipbip} we described the bipartite Leonard pairs.
In Section \ref{sec:Krawt} we described the Leonard pairs of Krawtchouk type.
In this section, we describe the Leonard pairs that are bipartite and have
Krawtchouk type.

Throughout this section, let $A,A^*$ denote a Leonard pair over $\F$ that has type II
with parameter array 
\begin{equation}
  (\{\th_i\}_{i=0}^d; \{\th^*_i\}_{i=0}^d; \{\vphi_i\}_{i=1}^d; \{\phi_i\}_{i=1}^d).   \label{eq:parray12d}
\end{equation}
Denote the corresponding primary data by
\begin{equation}
 (\delta, \mu, h, \delta^*, \mu^*, h^*, \tau).     \label{eq:basicseq2d}
\end{equation}

\begin{lemma}    \label{lem:Krawtparambip0}     \samepage
\ifDRAFT {\rm lem:Krawtparambip0}. \fi
The Leonard pair $A,A^*$ is bipartite and has Krawtchouk type 
if and only if
\begin{align*}
\delta &= 0, &
h &= 0,  &
h^* &= 0,  &
\tau &=0.
\end{align*}
\end{lemma}

\begin{proof}
By Lemmas \ref{lem:biptype2} and \ref{lem:defKrawt}.
\end{proof}

\begin{lemma}    \label{lem:Krawtparambip}    \samepage
\ifDRAFT {\rm lem:Krawtparambip}. \fi
The Leonard pair $A,A^*$ is bipartite and has Krawtchouk type
if and only if 
\begin{align*}
\th_i &= \mu (i - d/2)  &&  (0 \leq i \leq d),
\\
\th^*_i &= \delta^* + \mu^* (i - d/2)         &&  ( 0 \leq i \leq d),
\\
\vphi_i &=   - \mu \mu^*  i (d-i+1)/2      &&  (1 \leq i \leq d),
\\
\phi_i &=   \mu \mu^* i (d-i+1) /2  &&  (1 \leq i \leq d).
\end{align*}
\end{lemma}

\begin{proof}
By Lemmas \ref{lem:biptype2} and  \ref{lem:Krawtparam}.
\end{proof}

\begin{lemma}    \label{lem:Krawtbipparamb}    \samepage
\ifDRAFT {\rm lem:Krawtbipparamb}. \fi
Assume that the Leonard pair $A,A^*$ is bipartite and has Krawtchouk type.
Then
\begin{align*}
& \text{\rm $\text{\rm Char}(\F)$ is equal to $0$ or greater than $d$},
\\ 
& \mu \neq 0,  \qquad\qquad \mu^* \neq 0.
\end{align*}
\end{lemma}

\begin{proof}
By Lemmas \ref{lem:Krawtparamb}
and \ref{lem:Krawtparambip0} 
\end{proof}

Let 
$(\{a_i\}_{i=0}^d; \{x_i\}_{i=1}^d; \{\th^*_i\}_{i=0}^d)$
denote the TD/D sequence of $A,A^*$ corresponding to the parameter array
\eqref{eq:parray12d}.

\begin{lemma}   \label{lem:Krawtaixibip}    \samepage
\ifDRAFT {\rm lem:Krawtaixibip}. \fi
The Leonard pair  $A,A^*$ is bipartite and has Krawtchouk type if and only if
\begin{align*}
a_i &= 0       &&  (0 \leq i \leq d),
\\
x_i &=  \mu^2 i (d-i+1) / 4       && (1 \leq i \leq d),
\\
\th^*_i &= \delta^* + \mu^* (i - d/2)         &&  ( 0 \leq i \leq d).
\end{align*}
\end{lemma}

\begin{proof}
Use Lemmas \ref{lem:biptype2} and  \ref{lem:Krawtaixi}.
\end{proof}

\section{Near-bipartite Leonard pairs of dual $q$-Krawtchouk type}
\label{sec:dualqKrawt2}  
\ifDRAFT {\rm sec:dualqKrawt2}. \fi

In this section, we determine the near-bipartite Leonard pairs of dual $q$-Krawtchouk type,
and describe their bipartite contraction.

Throughout this section, let $A,A^*$ denote a Leonard pair over $\F$
that has dual $q$-Krawtchouk type.
Let 
\begin{equation}
  (\{\th_i\}_{i=0}^d; \{\th^*_i\}_{i=0}^d; \{\vphi_i\}_{i=1}^d; \{\phi_i\}_{i=1}^d)   \label{eq:parray11cx}
\end{equation}
denote a parameter array of $A,A^*$ and let
\begin{equation}
(\{a_i\}_{i=0}^d; \{x_i\}_{i=1}^d; \{\th^*_i\}_{i=0}^d)         \label{eq:tdseqauxx}
\end{equation}
denote the corresponding TD/D sequence.
Let
\begin{equation}
 (\delta, \mu, h, \delta^*, \mu^*, h^*, \tau)                 \label{eq:basicseqcx}
\end{equation}
denote the primary $q$-data corresponding to the parameter array
\eqref{eq:parray11cx}.
Note by Lemma \ref{lem:defLPdualqKrawt} that $\mu^* h^* = 0$ and $\tau=0$.
In view of Lemma \ref{lem:biptype1} and \eqref{eq:type1cond1}, we make a definition.

\begin{defi}     \label{def:pseq2x}    \samepage
\ifDRAFT {\rm def:pseq2x}. \fi
Define scalars $\delta'$, $\mu'$, $h'$, $\tau'$ in $\F$ such that
\begin{align}
 \delta' &=0,  &
 \mu' + h' &= 0,  &
 \mu' h' &= \mu h,  &
 \tau' &= 0.                \label{eq:mu'h'x}
\end{align}
\end{defi}

\begin{defi}     \label{def:thdvphidphidx}     \samepage
\ifDRAFT {\rm def:thdvphidphidx}. \fi
Define  scalars $\{\th'_i\}_{i=0}^d$, $\{\vphi'_i\}_{i=1}^d$, $\{\phi'_i\}_{i=1}^d$ in $\F$ as follows.
\begin{itemize}
\item[\rm (i)]
If $\mu^* = 0$, then
\begin{align*}
\th'_i &= \mu' q^{2i-d} + h'  q^{d-2i}  &&  (0 \leq i \leq d),
\\
\vphi'_i &= \mu' h^* q^{d-2i+1} (q^i - q^{-i})(q^{d-i+1}-q^{i-d-1})   &&  (1 \leq i \leq d),
\\
\phi'_i &= h' h^* q^{d-2i+1} (q^i - q^{-i})(q^{d-i+1}-q^{i-d-1})   &&  (1 \leq i \leq d).
\end{align*}
\item[\rm (ii)]
If $h^*=0$, then
\begin{align*}
\th'_i &= \mu' q^{2i-d} + h' q^{d-2i}  &&  (0 \leq i \leq d),
\\
\vphi'_i &= h' \mu^* q^{2i-d-1} (q^i - q^{-i})(q^{d-i+1}-q^{i-d-1})   &&  (1 \leq i \leq d),
\\
\phi'_i &= \mu' \mu^* q^{2i-d-1} (q^i - q^{-i})(q^{d-i+1}-q^{i-d-1})   &&  (1 \leq i \leq d).
\end{align*}
\end{itemize}
\end{defi}

\begin{lemma}    \label{lem:A-Feigen}    \samepage
\ifDRAFT {\rm lem:A-Feigen}. \fi
Let $F$ denote the flat part of $A$.
Then the eigenvalues of $A-F$ are $\{\th'_i\}_{i=0}^d$. 
\end{lemma}

\begin{proof}
We may assume that $A,A^*$ is in normalized TD/D form.
The matrix $A-F$ is shown in \eqref{eq:FA-F}.
Define a matrix $B \in \Matd$ that has $(i,i)$-entry $\th'_i$ for $0 \leq i \leq d$
and $(i,i-1)$-entry $1$ for $1 \leq i \leq d$.
All other entries of $B$ are zero.
Define the upper triangular matrix $P \in \Matd$ that  has $(i,j)$-entry
\begin{align*}
P_{i,j} &= 
     \frac{\vphi'_1 \vphi'_2 \cdots \vphi'_j} { \vphi'_1 \vphi'_2 \cdots \vphi'_i }
     \frac{\eta^*_{d-j} (\th^*_i)  } { \eta^*_{d-i} (\th^*_i) }
\end{align*}
for $0 \leq i \leq j \leq d$.
Observe that $P_{i,i} = 1$ for $0 \leq i \leq d$.
Therefore $P$ is invertible.
By matrix multiplication, we routinely find
\[
  (A-F)P = P B.
\]
The matrices $A-F$ and $B$ are similar, so they have the same eigenvalues.
The eigenvalues of $B$ are $\{\th'_i\}_{i=0}^d$.
The result follows.
\end{proof}

\begin{lemma}    \label{lem:dualqKrawtaux0x}    \samepage
\ifDRAFT {\rm lem:dualqKrawtaux0x}. \fi
With reference to Definition \ref{def:thdvphidphidx},
the following are equivalent:
\begin{itemize}
\item[\rm (i)]
$\{\th'_i\}_{i=0}^d$ are mutually distinct;
\item[\rm (ii)]
$q^{2i} \neq -1$ for $1 \leq i \leq d-1$.
\end{itemize}
\end{lemma}

\begin{proof}
For $0 \leq i, j \leq d$,
\[
 \th'_i - \th'_j  = \mu' (q^{i-j} - q^{j-i})(q^{d-i-j} + q^{i+j-d}).
\]
Thus for $0 \leq i < j \leq d$, $\th'_i \neq \th'_j$ if and only if
$q^{2(j-i)} \neq 1$ and $q^{2(d-i-j)} \neq -1$.
By \eqref{eq:type1paramcond1}, $q^{2(i-j)} \neq 1$ for $0 \leq i < j \leq d$.
So $\th'_i \neq \th'_j$ for $0 \leq i < j \leq d$  
if and only if $q^{2(d-i-j)} \neq -1$ for $0 \leq i < j \leq d$
if and only if $q^{2i} \neq -1$ for $1 \leq i \leq d-1$.
\end{proof}

\begin{lemma}   \label{lem:dualqKrawtaux0x2}    \samepage
\ifDRAFT {\rm lem:dualqKrawtaux0x2}. \fi
With reference to Lemma \ref{lem:dualqKrawtaux0x}, assume that the equivalent
conditions {\rm (i)} and {\rm (ii)} hold.
Then the following hold:
\begin{itemize}
\item[\rm (i)]
the sequence
\begin{equation}
 (\{\th'_i\}_{i=0}^d; \{\th^*_i\}_{i=0}^d; \{\vphi'_i\}_{i=1}^d; \{\phi'_i\}_{i=1}^d)   \label{eq:parrayaux2x}
\end{equation}
is a parameter array over $\F$;
\item[\rm (ii)]
the sequence
\begin{equation}
(\delta', \mu', h', \delta^*, \mu^*, h^*, \tau')                   \label{eq:basicseqaux2x}
\end{equation}
is the primary $q$-data corresponding to the parameter array \eqref{eq:parrayaux2x};
\item[\rm (iii)]
the parameter array \eqref{eq:parrayaux2x} is bipartite and has dual $q$-Krawtchouk type.
\end{itemize}
\end{lemma}

\begin{proof}
(i), (ii)
Use Lemmas \ref{lem:type1param} and \ref{lem:type1cond}.

(iii)
By Lemma \ref{lem:qKrawtparambip}.
\end{proof}

\begin{lemma}    \label{lem:xixdix}    \samepage
\ifDRAFT {\rm lem:xixdix}. \fi
We have $\vphi_i \phi_i = \vphi'_i \phi'_i$ for $1 \leq i \leq d$.
\end{lemma}

\begin{proof}
Use Lemma \ref{lem:comptype1}.
\end{proof}

\begin{prop}     \label{prop:dualqKrawtx}    \samepage
\ifDRAFT {\rm prop:dualqKrawtx}. \fi
Let $F$ denote the flat part of $A$.
Then the following are equivalent:
\begin{itemize}
\item[\rm (i)]
the Leonard pair $A,A^*$ is near-bipartite;
\item[\rm (ii)]
$A-F$ is diagonalizable;
\item[\rm (iii)]
$A-F$ is multiplicity-free;
\item[\rm (iv)]
$q^{2i} \neq -1$ for $1 \leq i \leq d-1$.
\end{itemize}
Suppose {\rm (i)--(iv)} hold.
Then the bipartite contraction of $A,A^*$ has parameter array \eqref{eq:parrayaux2x}.
\end{prop}

\begin{proof}
(i) $\Rightarrow$ (ii)
By Definition \ref{def:nearbip}, $A-F, A^*$ is a Leonard pair.
By this and Definition \ref{def:LP},  $A-F$ is diagonalizable.

(ii) $\Rightarrow$ (iii)
We may assume that $A-F$ is an irreducible tridiagonal matrix.
Now use Lemma \ref{lem:tridpre}.

(iii) $\Rightarrow$ (iv)
By Lemma \ref{lem:A-Feigen}, $\{\th'_i\}_{i=0}^d$ are the eigenvalues of $A-F$.
These eigenvalues are mutually distinct since $A-F$ is multiplicity-free.
By this and Lemma \ref{lem:dualqKrawtaux0x} we get (iv)

(iv) $\Rightarrow$ (i)
Note by Lemma \ref{lem:dualqKrawtaux0x2}(i)  that \eqref{eq:parrayaux2x} is a
parameter array over $\F$.
Now use Lemmas \ref{lem:nbipTDD3b} and \ref{lem:xixdix}.
\end{proof}

\begin{defi}     \label{def:reinforced}     \samepage
\ifDRAFT {\rm def:reinforced}. \fi
The Leonard pair $A,A^*$ is said to be {\em reinforced} whenever
$q^{2i} \neq -1$ for $1 \leq i \leq d-1$.
\end{defi}

\begin{corollary}    \label{cor:qKrawt1}    \samepage
\ifDRAFT {\rm cor:qKrawt1}. \fi
The Leonard pair $A,A^*$ is near-bipartite if and only if
$A,A^*$ is reinforced.
In this case, the bipartite contraction of $A,A^*$ has
reinforced dual $q$-Krawtchouk type with parameter array
\eqref{eq:parrayaux2x}.
\end{corollary}

\begin{proof}
By Lemma \ref{lem:dualqKrawtaux0x2}(iii) and Proposition \ref{prop:dualqKrawtx}.
\end{proof}

We have a comment.

\begin{lemma}    \label{lem:qKrawt}    \samepage
\ifDRAFT {\rm lem:qKrawt}. \fi
Assume that $q$ is not a root of unity.
Then $A,A^*$ is reinforced.
\end{lemma}

\section{Leonard pairs of Krawtchouk type are near-bipartite}
\label{sec:Krawt2}   
\ifDRAFT {\rm sec:Krawt2}. \fi

In this section,
we show that a Leonard pair of Krawtchouk type is near-bipartite,
and we describe its bipartite contraction.

Throughout this section, let $A,A^*$ denote a Leonard pair over $\F$
that has Krawtchouk type, with parameter array
\begin{equation}
  (\{\th_i\}_{i=0}^d; \{\th^*_i\}_{i=0}^d; \{\vphi_i\}_{i=1}^d; \{\phi_i\}_{i=1}^d)   \label{eq:parray11d}
\end{equation}
and corresponding TD/D sequence 
\begin{equation}
(\{a_i\}_{i=0}^d; \{x_i\}_{i=1}^d; \{\th^*_i\}_{i=0}^d).          \label{eq:tdseqauxd}
\end{equation}
Let
\begin{equation}
 (\delta, \mu, h, \delta^*, \mu^*, h^*, \tau)                 \label{eq:basicseqd}
\end{equation}
denote the primary data corresponding to the parameter array
\eqref{eq:parray11d}.
In view of Lemma \ref{lem:biptype2} and \eqref{eq:type2cond3}, we make a definition.

\begin{defi}     \label{def:pseq2d}    \samepage
\ifDRAFT {\rm def:pseq2d}. \fi
Define scalars $\delta'$, $\mu'$, $h'$, $\tau'$ in $\F$ such that
\begin{align}
 \delta' &=0,  &
  h' &= 0,  &
 \tau' &= 0,  &
 \mu^{\prime 2} &= \mu^2 - \frac{4 \tau^2}{\mu^{* 2} }.                 \label{eq:mu'd}
\end{align}
\end{defi}

\begin{defi}     \label{def:thdvphidphidd}     \samepage
\ifDRAFT {\rm def:thdvphidphidd}. \fi
Define  scalars
\begin{align*}
\th'_i &= \mu' (i - d/2)  &&  (0 \leq i \leq d),
\\
\vphi'_i &=   - \mu' \mu^*  i (d-i+1)/2      &&  (1 \leq i \leq d),
\\
\phi'_i &=   \mu' \mu^* i (d-i+1) /2  &&  (1 \leq i \leq d).
\end{align*}
\end{defi}

\begin{lemma}    \label{lem:Krawtaux1d}     \samepage
\ifDRAFT {\rm lem:Krawtaux1d}. \fi
The following hold:
\begin{itemize}
\item[\rm (i)]
the sequence
\begin{equation}
 (\{\th'_i\}_{i=0}^d; \{\th^*_i\}_{i=0}^d; \{\vphi'_i\}_{i=1}^d; \{\phi'_i\}_{i=1}^d)   \label{eq:parrayaux2d}
\end{equation}
is a parameter array over $\F$;
\item[\rm (ii)]
the sequence
\begin{equation}
(\delta', \mu', h', \delta^*, \mu^*, h^*, \tau')                   \label{eq:basicseqaux2d}
\end{equation}
is the primary data corresponding to the parameter array \eqref{eq:parrayaux2d};
\item[\rm (iii)]
the parameter array \eqref{eq:parrayaux2d} is bipartite and has Krawtchouk type.
\end{itemize}
\end{lemma}

\begin{proof}
(i), (ii)
Note by Lemma \ref{lem:Krawtparamb} that $\mu' \neq 0$.
Now use Lemmas \ref{lem:type2param} and \ref{lem:type2cond}.

(iii)
By Lemma \ref{lem:Krawtparambip}.
\end{proof}

\begin{lemma}    \label{lem:xixdid}    \samepage
\ifDRAFT {\rm lem:xixdid}. \fi
We have $\vphi_i \phi_i = \vphi'_i \phi'_i$ for $1 \leq i \leq d$.
\end{lemma}

\begin{proof}
Use Lemma \ref{lem:comptype2}.
\end{proof}

\begin{prop}     \label{prop:Krawt}    \samepage
\ifDRAFT {\rm prop:Krawt}. \fi
The Leonard pair $A,A^*$ is near-bipartite.
The bipartite contraction of $A,A^*$ has Krawtchouk type 
with parameter array \eqref{eq:parrayaux2d}.
\end{prop}

\begin{proof}
By Lemma \ref{lem:Krawtaux1d}, the sequence \eqref{eq:parrayaux2d} is a
parameter array over $\F$ that is bipartite and has Krawtchouk type.
Now use Lemmas \ref{lem:nbipTDD3b} and \ref{lem:xixdid}.
\end{proof}

\section{The classification of near-bipartite Leonard pairs}
\label{sec:classify}  
\ifDRAFT {\rm sec:classify}. \fi

In this section, we classify the near-bipartite Leonard pairs with diameter at least $3$.
In Section \ref{sec:nbip} we saw that an essentially bipartite Leonard pair is
near-bipartite.
In Section \ref{sec:dualqKrawt2} we showed that a Leonard pair $A,A^*$ of dual $q$-Krawtchouk type
is near-bipartite if and only if $A,A^*$ is reinforced,
and in that case the bipartite contraction of $A,A^*$ has reinforced dual $q$-Krawtchouk type.
In Section \ref{sec:Krawt2} we showed that a Leonard pair of Krawtchouk type
is near-bipartite, and its bipartite contraction has Krawtchouk type.
We now state our classification result.

\begin{theorem}    \label{thm:main}    \samepage
\ifDRAFT {\rm thm:main}. \fi
Let $A,A^*$ denote a Leonard pair over $\F$ with diameter $d \geq 3$.
Then $A,A^*$ is near-bipartite if and only if at least one of the following
{\rm (i)--(iii)} holds:
\begin{itemize}
\item[\rm (i)]
$A,A^*$ is essentially bipartite;
\item[\rm (ii)]
$A,A^*$ has reinforced dual $q$-Krawtchouk type;
\item[\rm (iii)]
$A,A^*$ has Krawtchouk type.
\end{itemize}
\end{theorem}

\begin{proof}
First assume that at least one of (i)--(iii) holds.
If (i) holds, then $A,A^*$ is near-bipartite by Note \ref{note:nearess}.
If (ii) holds, then $A,A^*$ is near-bipartite by Corollary \ref{cor:qKrawt1}.
If (iii) holds, then $A,A^*$ is near-bipartite by Proposition \ref{prop:Krawt}.
We are done in one logical direction.
Next assume that $A,A^*$ is near-bipartite, and let $B,A^*$
denote the bipartite contraction of $A,A^*$.
By Lemma \ref{lem:TDD3} we may assume that $A,A^*$ is in normalized TD/D form.
Let
\begin{align}
& (\{ \th_i\}_{i=0}^d; \{\th^*_i\}_{i=0}^d; \{\vphi_i\}_{i=1}^d; \{\phi_i\}_{i=1}^d), 
&&  (\{ \th'_i\}_{i=0}^d; \{\th^*_i\}_{i=0}^d; \{\vphi'_i\}_{i=1}^d; \{\phi'_i\}_{i=1}^d) 
                                      \label{eq:parrays4}
\end{align}
denote a parameter array of $A,A^*$ and $B,A^*$, respectively.
By Lemma \ref{lem:essbiptype3-},
$B,A^*$ does not have type III$^-$.
Recall that $A,A^*$ has the same type as $B,A^*$.
So the type of $A,A^*$ is one of  I, II, III$^+$.

First assume that $A,A^*$ has type I.
For the parameter arrays in \eqref{eq:parrays4},
let
\begin{align*}
&  (\delta, \mu, h, \delta^*, \mu^*, h^*, \tau),
&&  (\delta', \mu', h', \delta^*, \mu^*, h^*, \tau')
\end{align*}
denote the corresponding primary $q$-data.
By Lemmas \ref{lem:nbipTDD3} and \ref{lem:comptype1},
we have \eqref{eq:type1cond1}--\eqref{eq:type1cond3}.
Since $B,A^*$ is bipartite,
we have by Lemma \ref{lem:biptype1} that
\begin{align*}
   \tau' &=0,  & \mu' + h' &= 0.  
\end{align*}
By these comments,
\begin{align}
   \tau (\mu+h) &= 0,                  \label{eq:type1c2}
\\
\tau^2 + (\mu+h)^2 \mu^* h^*  &= 0.   \label{eq:type1c3}
\end{align}
First assume that $\mu+h$ = 0.
Then $\tau=0$ by \eqref{eq:type1c3}.
Now by Lemma \ref{lem:essbiptype1}, $A,A^*$ is essentially bipartite.
Next assume that $\mu + h \neq 0$.
By \eqref{eq:type1c2}, $\tau=0$.
By this and \eqref{eq:type1c3}, $\mu^* h^* = 0$.
Thus $A,A^*$ has dual $q$-Krawtchouk type by Lemma \ref{lem:defLPdualqKrawt}.
By Corollary \ref{cor:qKrawt1}, $A,A^*$ is reinforced.
We have shown that 
$A,A^*$ is either essentially bipartite or of reinforced dual $q$-Krawtchouk type.

Next assume that $A,A^*$ has type II.
For the parameter arrays in \eqref{eq:parrays4},
let
\begin{align*}
& (\delta, \mu, h, \delta^*, \mu^*, h^*, \tau),  
&& 
 (\delta', \mu', h', \delta^*, \mu^*, h^*, \tau')
\end{align*}
denote the corresponding primary data.
By Lemmas \ref{lem:nbipTDD3} and \ref{lem:comptype2},
we have \eqref{eq:type2cond1}--\eqref{eq:type2cond3}.
Since $B,A^*$ is bipartite,
we have by Lemma \ref{lem:biptype2} that
\begin{align*}
   \tau' &=0,  &  h' &= 0.  
\end{align*}
By these comments,
\begin{align}
   h &= 0,                  \label{eq:type2c1}
\\
(\mu^2 - \mu^{\prime 2}) h^* &= 0,            \label{eq:type2c2}
\\
4 \tau^2 &= (\mu^2 - \mu^{\prime 2}) \big( \mu^{* 2} + (d-1)^2 h^{* 2} \big).   \label{eq:type2c3}
\end{align}
First assume that $h^*$ = 0.
Then $A,A^*$ has Krawtchouk type by Lemma \ref{lem:defKrawt}.
Next assume that $h^* \neq 0$.
Then by \eqref{eq:type2c2}, $\mu^2 - \mu^{\prime 2} = 0$.
By this and \eqref{eq:type2c3} we get $\tau = 0$.
By these comments and Lemma \ref{lem:essbiptype2}, 
$A,A$ is essentially bipartite.
We have shown that $A,A^*$ is either essentially bipartite or 
of Krawtchouk type.

Next assume that $A,A^*$ has type III$^+$.
For the parameter arrays in \eqref{eq:parrays4},
let
\begin{align*}
& (\delta, s, h, \delta^*, s^*, h^*, \tau),  
&& 
 (\delta', s', h', \delta^*, s^*, h^*, \tau')
\end{align*}
denote the corresponding primary data.
By Lemmas \ref{lem:nbipTDD3} and \ref{lem:comptype3+},
we have \eqref{eq:type3+cond1}--\eqref{eq:type3+cond3}.
Since $B,A^*$ is bipartite,
we have by Lemma \ref{lem:biptype3+} that
\begin{align*}
   \tau' &=0,  &  s' &= 0.  
\end{align*}
By these comments,
\begin{align*}
\tau + s h^* &= 0, 
&
\tau - s h^* &= 0.
\end{align*}
We have $h^* \neq 0$ by \eqref{eq:type3+paramcond1b},
so  $\tau = 0$ and $s = 0$.
By this and Lemma \ref{lem:essbiptype3+},
$A,A^*$ is essentially bipartite.
\end{proof}

\section{Near-bipartite expansions of a bipartite Leonard pair that has
reinforced dual $q$-Krawtchouk type}
\label{sec:expansiondualqKrawt}
\ifDRAFT {\rm sec:expansiondualqKrawt}. \fi

In Theorem \ref{thm:main} we classified the near-bipartite Leonard pairs with diameter $d \geq 3$.
Our next goal is to describe the near-bipartite expansions $A,A^*$ of a given bipartite Leonard pair $B,A^*$
with diameter $d \geq 3$.
In view of Theorem \ref{thm:main} and the discussion above it,
we may assume that $B,A^*$ has either reinforced dual $q$-Krawtchouk type
or Krawtchouk type.
In the present section we assume that $B,A^*$ has reinforced dual $q$-Krawtchouk type.
In Section \ref{sec:expansionKrawt} we assume that $B,A^*$ has Krawtchouk type.

For the rest of this section,
we assume that $B,A^*$ has reinforced dual $q$-Krawtchouk type.
Let 
\begin{equation}
(\{ \th'_i\}_{i=0}^d; \{\th^*_i\}_{i=0}^d; \{\vphi'_i\}_{i=1}^d; \{\phi'_i\}_{i=1}^d) 
                                      \label{eq:parray31}
\end{equation}
denote a parameter array of $B,A^*$.
Let
\begin{equation}
 (\delta', \mu', h', \delta^*, \mu^*, h^*, \tau')    \label{eq:primaryBBs}
\end{equation}
denote the primary $q$-data corresponding to the parameter array
\eqref{eq:parray31}.
By Lemma \ref{lem:qKrawtparambip0},
\begin{align}
\delta' &=0, &
\mu'+h'&=0, &
\mu^* h^* &= 0,  &
\tau' &= 0.    \label{eq:taudmudhd}
\end{align}
Note by Lemma \ref{lem:qKrawtparamb} that $\mu ' h' \neq 0$.

\begin{theorem}    \label{thm:qKrawtexpansion}    \samepage
\ifDRAFT {\rm thm:qKrawtexpansion}. \fi
For scalars  $\delta$, $\mu$, $h$, $\tau$ in $\F$
the following are equivalent:
\begin{itemize}
\item[\rm (i)]
the sequence
\begin{equation}
   (\delta, \mu, h, \delta^*, \mu^*, h^*, \tau)           \label{eq:basicseq0}
\end{equation}
is a primary $q$-data of a near-bipartite expansion $A,A^*$ of $B,A^*$;
\item[\rm (ii)]
the following conditions hold:
\begin{align}
 \mu &\neq 0,                               \label{eq:muneq0}
\\
 \mu &\neq \pm \sqrt{-1} \mu' q^i  \qquad (1-d \leq i \leq d-1),   \label{eq:mucond}
\\
\tau &= 0,                                  \label{eq:tau0}
\\
h &= \mu' h' / \mu.                     \label{eq:h}                         
\end{align}
\end{itemize}
Assume that {\rm (i)} and {\rm (ii)} hold.
Then $A,A^*$ has reinforced dual $q$-Krawtchouk type.
\end{theorem}

\begin{proof}
(i) $\Rightarrow$ (ii)
Let
\begin{equation}
   (\{\th_i\}_{i=0}^d; \{\th^*_i\}_{i=0}^d; \{\vphi_i\}_{i=1}^d; \{\phi_i\}_{i=1}^d)
                                                   \label{eq:ABsparray}
\end{equation}
denote the parameter array of $A,A^*$ corresponding to the primary $q$-data \eqref{eq:basicseq0}.
By Lemma \ref{lem:nbipTDD3c}, $\vphi_i \phi_i = \vphi'_i \phi'_i$ for $1 \leq i \leq d$.
By this and Lemma \ref{lem:comptype1}, \eqref{eq:type1cond1}--\eqref{eq:type1cond3} hold.
By \eqref{eq:type1cond1} and since $\mu' h' \neq 0$,
we get \eqref{eq:muneq0}.
By this and \eqref{eq:type1cond1} we get \eqref{eq:h}.
By \eqref{eq:type1cond2} and \eqref{eq:taudmudhd},
\begin{equation}
   \tau (\mu+h)=0.                            \label{eq:taumuh}
\end{equation}
If $\mu+h \neq 0$ then $\tau=0$ by \eqref{eq:taumuh}.
If $\mu+h = 0$ then $\tau=0$ by \eqref{eq:type1cond3} and \eqref{eq:taudmudhd}.
Thus we have \eqref{eq:tau0}.
By Lemma \ref{lem:type1cond}, 
$\mu \neq h q^{2i}$ for $1-d \leq i \leq d-1$.
By this and \eqref{eq:taudmudhd},  \eqref{eq:h} we get \eqref{eq:mucond}.
We have shown that \eqref{eq:muneq0}--\eqref{eq:h} hold.

(ii) $\Rightarrow$ (i)
We first show that \eqref{eq:type1paramcond1}--\eqref{eq:type1paramcond5} hold.
We have \eqref{eq:type1paramcond1} and \eqref{eq:type1paramcond3} since \eqref{eq:primaryBBs}
is a primary $q$-data.
By \eqref{eq:taudmudhd} we have $h' = - \mu'$.
By this and \eqref{eq:mucond}, \eqref{eq:h}  we get \eqref{eq:type1paramcond2}.
By \eqref{eq:taudmudhd} we have $\mu^* h^* = 0$.
By Lemma \ref{lem:qKrawtbipparamb}, $\mu^*$ and $h^*$ are not both zero.
By \eqref{eq:muneq0}, $\mu \neq 0$.
We have $h \neq 0$ by \eqref{eq:h}.
By \eqref{eq:tau0}, $\tau=0$.
By these comments,  we get 
\eqref{eq:type1paramcond4} and \eqref{eq:type1paramcond5}.
We have shown that \eqref{eq:type1paramcond1}--\eqref{eq:type1paramcond5} hold.
Therefore \eqref{eq:basicseq0} is a primary $q$-data.
Define scalars $\{\th_i\}_{i=0}^d$, $\{\vphi_i\}_{i=1}^d$, $\{\phi_i\}_{i=1}^d$
by \eqref{eq:type1th}, \eqref{eq:type1vphi}, \eqref{eq:type1phi}.
By Lemmas \ref{lem:type1param} and \ref{lem:type1cond} the sequence
\begin{equation}
   (\{\th_i\}_{i=0}^d; \{\th^*_i\}_{i=0}^d; \{\vphi_i\}_{i=1}^d; \{\phi_i\}_{i=1}^d)
                                                   \label{eq:ABsparray2}
\end{equation}
is a parameter array over $\F$ corresponding to \eqref{eq:basicseq0}.
By \eqref{eq:taudmudhd} and \eqref{eq:muneq0}--\eqref{eq:h},
we find that \eqref{eq:type1cond1}--\eqref{eq:type1cond3} hold.
By this and Lemma \ref{lem:comptype1},  $\vphi_i \phi_i = \vphi'_i \phi'_i$ for
$1 \leq i \leq d$.
By this and Lemma \ref{lem:nbipTDD3c}, there exists a near-bipartite expansion $A,A^*$
of $B,A^*$ that has parameter array \eqref{eq:ABsparray2}.
By construction, the sequence \eqref{eq:basicseq0} is a primary $q$-data of $A,A^*$.

Assume (i) and (ii) hold.
By Lemma \ref{lem:defLPdualqKrawt} and \eqref{eq:taudmudhd}, \eqref{eq:tau0}
we find that $A,A^*$ has dual $q$-Krawtchouk type.
Moreover $A,A^*$ is reinforced since $B,A^*$ is reinforced.
\end{proof}

Our next goal is to describe the near-bipartite expansions of $B,A^*$ in terms of matrices.
In this description we will use the following matrix.

\begin{defi}    \label{def:K}    \samepage
\ifDRAFT {\rm def:K}. \fi
Let $K \in \Matd$ denote the diagonal matrix
that has $(i,i)$-entry $q^{2i-d}$ for $0 \leq i \leq d$.
\end{defi}

By Lemma \ref{lem:TDD3} we may assume that $B,A^*$ is in normalized TD/D form:
\begin{align*}
B &=
 \begin{pmatrix}
  0 & x'_1 & & & & \text{\bf 0}  \\
  1 & 0 & x'_2 \\
       & 1 & \cdot & \cdot  \\
       &       &  \cdot & \cdot & \cdot \\
       &        &         & \cdot & \cdot & x'_d   \\
 \text{\bf 0} & & & & 1 & 0
 \end{pmatrix},
&
A^* &= \text{\rm diag} ( \th^*_0, \th^*_1, \ldots, \th^*_d).
\end{align*}

\begin{theorem}    \label{thm:conv2no2}    \samepage
\ifDRAFT {\rm thm:conv2no2}. \fi
Given scalars $\delta$, $\mu$, $h$, $\tau$ in $\F$ that satisfy the equivalent conditions
{\rm (i)} and {\rm (ii)} in Theorem \ref{thm:qKrawtexpansion}.
Let $A,A^*$ denote the near-bipartite expansion of $B,A^*$ that has
primary $q$-data
$(\delta, \mu, h, \delta^*, \mu^*, h^*, \tau)$.
\begin{itemize}
\item[\rm (i)]
Assume that $\mu^* = 0$.
Then 
\[
 A = B + (\mu+h) K + \delta I.
\]
\item[\rm (ii)]
Assume that $h^* = 0$.
Then 
\[
 A = B + (\mu + h) K^{-1} + \delta I.
\]
\end{itemize}
\end{theorem}

\begin{proof}
By Lemma \ref{lem:nbipTDD2} the Leonard pair $A,A^*$ is in normalized TD/D form.
Let $(\{a_i\}_{i=0}^d; \{x_i\}_{i=1}^d; \{\th^*_i\}_{i=0}^d)$
denote the TD/D sequence of $A,A^*$.
By Lemma \ref{lem:qKrawtaixi} we see that
\begin{align*}
  a_i &=
    \begin{cases}
       \delta + (\mu+h) q^{2i-d}  & \text{ if $\mu^* = 0$},
    \\
      \delta + (\mu + h) q^{d-2i}  & \text{ if $h^* = 0$}
   \end{cases}
  &&  (0 \leq i \leq d). 
\end{align*}
By these comments and Definition \ref{def:K} we get the result.
\end{proof}

\section{Near-bipartite expansions of a bipartite Leonard pair that has
Krawtchouk type}
\label{sec:expansionKrawt}
\ifDRAFT {\rm sec:expansionKrawt}. \fi

In this section, we describe the near-bipartite expansions $A,A^*$ of 
a given bipartite Leonard pair $B,A^*$ that has Krawtchouk type.
Let 
\begin{equation}
(\{ \th'_i\}_{i=0}^d; \{\th^*_i\}_{i=0}^d; \{\vphi'_i\}_{i=1}^d; \{\phi'_i\}_{i=1}^d) 
                                      \label{eq:parray32}
\end{equation}
denote a parameter array of $B,A^*$.
Let
\begin{equation}
 (\delta', \mu', h', \delta^*, \mu^*, h^*, \tau')              \label{eq:primaryBBs2}
\end{equation}
denote the primary data corresponding to the parameter array
\eqref{eq:parray32}.
By Lemma \ref{lem:Krawtparambip0},
\begin{align}
\delta'&=0, &
h' &= 0, &
h^* &=0,  &
\tau' &= 0.                                    \label{eq:Krawtbasic}
\end{align}

\begin{theorem}   \label{thm:Krawtexpansion}    \samepage
\ifDRAFT {\rm thm:Krawtexpansion}. \fi
For scalars $\delta$, $\mu$, $h$, $\tau$ in $\F$ the following are equivalent:
\begin{itemize}
\item[\rm (i)]
the sequence
\begin{equation}
 (\delta, \mu, h, \delta^*, \mu^*, h^*, \tau)     \label{eq:Krawtbasic1}
\end{equation}
is a primary data of a near-bipartite expansion $A,A^*$ of $B,A^*$;
\item[\rm (ii)]
the following conditions hold:
\begin{align}
\mu &\neq 0,                                                              \label{eq:Krawtcond0}
\\
h &= 0,                                                                       \label{eq:Krawtcond1}
\\
4 \tau^2 &= (\mu^2 - \mu^{\prime 2}) \mu^{* 2}.               \label{eq:Krawtcond2}
\end{align}
\end{itemize}
Assume that {\rm (i)} and {\rm (ii)} hold.
Then $A,A^*$ has Krawtchouk type.
\end{theorem}

\begin{proof}
(i) $\Rightarrow$ (ii)
Let
\begin{equation}
  (\{\th_i\}_{i=0}^d; \{\th^*_i\}_{i=0}^d; \{\vphi_i\}_{i=1}^d; \{\phi_i\}_{i=1}^d)
                                            \label{eq:Krawtparray}
\end{equation}
denote the parameter array of $A,A^*$ corresponding to \eqref{eq:Krawtbasic1}.
By Lemmas \ref{lem:nbipTDD3c}, $\vphi_i \phi_i = \vphi'_i \phi'_i$ for $1 \leq i \leq d$.
By this and Lemma \ref{lem:comptype2}, \eqref{eq:type2cond1}--\eqref{eq:type2cond3} hold.
By \eqref{eq:type2cond1} and \eqref{eq:Krawtbasic} we get \eqref{eq:Krawtcond1}.
By this and \eqref{eq:type2paramcond2} we get \eqref{eq:Krawtcond0}.
By \eqref{eq:type2cond3}, \eqref{eq:Krawtbasic}, \eqref{eq:Krawtcond1}
we get \eqref{eq:Krawtcond2}.

(ii) $\Rightarrow$ (i)
We first show that \eqref{eq:type2paramcond1}--\eqref{eq:type2paramcond5} hold.
We have \eqref{eq:type2paramcond1} and \eqref{eq:type2paramcond3} since
\eqref{eq:primaryBBs2} is a primary data of type II.
By \eqref{eq:Krawtcond0} and \eqref{eq:Krawtcond1} we get \eqref{eq:type2paramcond2}.
By \eqref{eq:Krawtbasic}, \eqref{eq:Krawtcond1}, \eqref{eq:Krawtcond2} we get
\eqref{eq:type2paramcond4} and \eqref{eq:type2paramcond5}.
We have shown that \eqref{eq:type2paramcond1}--\eqref{eq:type2paramcond5} hold.
Therefore \eqref{eq:Krawtbasic1} is a primary data of type II.
Define scalars $\{\th_i\}_{i=0}^d$, $\{\vphi_i\}_{i=1}^d$, $\{\phi_i\}_{i=1}^d$
by \eqref{eq:type2th}, \eqref{eq:type2vphi}, \eqref{eq:type2phi}.
By Lemmas \ref{lem:type2param} and \ref{lem:type2cond}, the sequence
\begin{equation}
  (\{\th_i\}_{i=0}^d; \{\th^*_i\}_{i=0}^d; \{\vphi_i\}_{i=1}^d; \{\phi_i\}_{i=1}^d)
                                            \label{eq:Krawtparray2}
\end{equation}
is the parameter array  corresponding to the primary data \eqref{eq:Krawtbasic1} of type II.
By \eqref{eq:Krawtbasic} and \eqref{eq:Krawtcond0}--\eqref{eq:Krawtcond2},
we get \eqref{eq:type2cond1}--\eqref{eq:type2cond3}.
By this and Lemma \ref{lem:comptype2}, $\vphi_i \phi_i = \vphi'_i \phi'_i$ for $1 \leq i \leq d$.
By this and Lemma \ref{lem:nbipTDD3c}, there exists a near-bipartite expansion $A,A^*$ of $B,A^*$
that has parameter array \eqref{eq:Krawtparray2}.
By construction \eqref{eq:Krawtbasic1} is a primary data of $A,A^*$.

Assume that (i) and (ii) hold.
By \eqref{eq:Krawtbasic} and \eqref{eq:Krawtcond1} we have $h=0$ and $h^*=0$.
By this and Lemma \ref{lem:defKrawt} the Leonard pair $A,A^*$ has Krawtchouk type.
\end{proof}

Our next goal is to describe the near-bipartite expansions of $B,A^*$ in terms of matrices.
In this description we will use the following matrix.

\begin{defi}    \label{def:KrawtK}   \samepage
\ifDRAFT {\rm def:KrawtK}. \fi
Define a diagonal matrix $H \in \Matd$ that has $(i,i)$-entry $2i-d$ for $0 \leq i \leq d$.
\end{defi}

By Lemma \ref{lem:TDD3} we may assume that $B,A^*$ is in normalized TD/D form:
\begin{align*}
B &=
 \begin{pmatrix}
  0 & x'_1 & & & & \text{\bf 0}  \\
  1 & 0 & x'_2 \\
       & 1 & \cdot & \cdot  \\
       &       &  \cdot & \cdot & \cdot \\
       &        &         & \cdot & \cdot & x'_d   \\
 \text{\bf 0} & & & & 1 & 0
 \end{pmatrix},
&
A^* &= \text{\rm diag} ( \th^*_0, \th^*_1, \ldots, \th^*_d).
\end{align*}

\begin{theorem}     \label{thm:converse3}    \samepage
\ifDRAFT {\rm thm:converse3}. \fi
Given scalars $\delta$, $\mu$, $h$, $\tau$ in $\F$ that satisfy the
equivalent conditions {\rm (i)} and {\rm (ii)} in Theorem \ref{thm:Krawtexpansion}.
Let $A,A^*$ denote the near-bipartite expansion of $B,A^*$ that has primary data
$(\delta, \mu, h, \delta^*, \mu^*, h^*, \tau)$.
Then
\[
    A = B + (\tau/\mu^*) H + \delta I.
\]
\end{theorem}

\begin{proof}
By Lemma \ref{lem:nbipTDD2} the Leonard pair $A,A^*$ is in normalized TD/D form.
Let $(\{a_i\}_{i=0}^d; \{x_i\}_{i=1}^d; \{\th^*_i\}_{i=0}^d)$ denote the TD/D sequence
of $A,A^*$.
By Lemma \ref{lem:Krawtaixi},
\begin{align*}
 a_i &= \delta + (2i-d) \tau/\mu^*  &&  (0  \leq i \leq d).
\end{align*}
By these comments and Definition \ref{def:KrawtK}, we get the result.
\end{proof}

{
\small

\medskip

\begin{itemize}
\item[]
Kazumasa Nomura:
\item[]
Tokyo Medical and Dental University,
Kohnodai Ichikawa 272-0827 Japan.
\item[]
Email: knomura@pop11.odn.ne.jp
\end{itemize}

\medskip
\begin{itemize}
\item[]
Paul Terwilliger:
\item[]
Department of Mathematics,
University of Wisconsin, 
480 Lincoln Drive,
Madison,
Wisconsin,
53706 USA
\item[]
Email: terwilli@math.wisc.edu
\end{itemize}

\end{document}